\documentclass[12pt]{amsart}

\usepackage{amsmath, amssymb, amscd, verbatim, xspace,amsthm}
\usepackage{latexsym, epsfig, color}
\usepackage[hidelinks]{hyperref}

\usepackage[OT2,T1]{fontenc}
\DeclareSymbolFont{cyrletters}{OT2}{wncyr}{m}{n}
\DeclareMathSymbol{\Sha}{\mathalpha}{cyrletters}{"58}

\newcommand\an{\operatorname{an}}

\newcommand{\A}{\ensuremath{{\mathbb{A}}}}
\newcommand{\C}{\ensuremath{{\mathbb{C}}}}
\newcommand{\Z}{\ensuremath{{\mathbb{Z}}}\xspace}
\renewcommand{\P}{\ensuremath{{\mathbb{P}}}}

\newcommand{\Q}{\ensuremath{{\mathbb{Q}}}}
\newcommand{\R}{\ensuremath{{\mathbb{R}}}}
\newcommand{\F}{\ensuremath{{\mathbb{F}}}}

\newcommand{\E}{\ensuremath{{\mathbb{E}}}}

\newcommand{\ra}{\rightarrow}

\newcommand\Conf{\operatorname{Conf}}
\newcommand\Hom{\operatorname{Hom}}

\newcommand\Aut{\operatorname{Aut}}

\newcommand\im{\operatorname{im}}

\newcommand\Gal{\operatorname{Gal}}

\newcommand\Sur{\operatorname{Sur}}

\newcommand\Tr{\operatorname{Tr}}

\newcommand\isom{\simeq}
\newcommand\sub{\subset}

\newcommand\Disc{\operatorname{Disc}}

\newcommand\Id{\ensuremath{\operatorname{Id}}\xspace}

\newcommand\Spec{\operatorname{Spec}}

\newcommand\Frob{\operatorname{Frob}}

\renewcommand\O{\mathcal{O}}

\newcommand\Pic{\operatorname{Pic}}

\newcommand{\bad}{\Delta}

\newcommand\bq{\begin{equation}}
\newcommand\eq{\end{equation}}

\numberwithin{equation}{section}
\newtheorem{proposition}[equation]{Proposition}
\newtheorem{theorem}[equation]{Theorem}
\newtheorem{corollary}[equation]{Corollary}

\newtheorem{lemma}[equation]{Lemma}

\newtheorem{conjecture}[equation]{Conjecture}

\newtheorem{definition}[equation]{Definition}
\newtheorem{notation}[equation]{Notation}

\theoremstyle{remark}
\newtheorem{remark}[equation]{Remark}

\usepackage{fullpage,pdfsync}
\usepackage[all,cmtip]{xy}
\usepackage{tikz-cd}
\usepackage{faktor}
\usepackage{bbm}

\newtheorem{nts}{Note to self}

\newcommand{\Hur}{\operatorname{Hur}}

\newcommand{\one}{\langle -1 \rangle}
\newcommand{\defi}[1]{{\emph{#1}}}

\newcommand{\id}{\operatorname{id}}
\newcommand{\sm}{\operatorname{sm}}

\newcommand{\Cl}{\operatorname{Cl}}
\newcommand{\rD}{\operatorname{rDisc}}
\newcommand{\nr}{\operatorname{nr}}
\newcommand{\chr}{\operatorname{char}}
\newcommand{\FF}{\mathcal{F}}

\newcommand{\Prob}{\operatorname{Prob}}
\newcommand{\calF}{\mathcal{F}}
\newcommand{\calC}{\mathcal{C}}

\newcommand{\calD}{\mathcal{D}}
\newcommand{\calA}{\mathcal{A}}
\newcommand{\calN}{\mathcal{N}}
\newcommand{\calE}{\mathcal{E}}
\newcommand{\calO}{\mathcal{O}}

\newcommand\Cen{\mathbf{C}}

\newcommand\Inn{\operatorname{Inn}}
\newcommand\CF{\mathcal{CF}}

\newcommand{\Res}{\operatorname{Res}}

\newcommand{\Rr}{S}
\newcommand{\cHur}{\mathsf{Hur}}

\title{A predicted distribution for Galois groups of maximal unramified extensions}

\author{Yuan Liu}
\address{Department of Mathematics\\
University of Michigan \\ 2074 East Hall
\\530 Church Street\\
Ann Arbor, MI 48109-1043 USA}  
\email{yyyliu@umich.edu}

\author{Melanie Matchett Wood}
\address{Department of Mathematics\\
Harvard University\\
Science Center Room 325\\
1 Oxford Street\\
Cambridge, MA 02138 USA}  
\email{mmwood@math.harvard.edu}

\author{David Zureick-Brown}
\address{Department of Mathematics\\
Emory University  \\ 400 Dowman Drive \\
Atlanta, GA 30322 USA}  
\email{dzb@mathcs.emory.edu}

\date{\today}

\begin{document}
\begin{abstract}
We consider the distribution of the Galois groups $\Gal(K^{\operatorname{un}}/K)$ of maximal unramified extensions as $K$ ranges over $\Gamma$-extensions of $\Q$ or $\F_q(t)$.  We prove two properties of 
$\Gal(K^{\operatorname{un}}/K)$ coming from number theory, which we use as motivation to build a probability distribution on profinite groups with these properties.  In Part I, we build such a distribution as a limit of distributions on $n$-generated profinite groups. In Part II, we prove as $q\ra\infty$, agreement of 
$\Gal(K^{\operatorname{un}}/K)$ as $K$ varies over totally real $\Gamma$-extensions of $\F_q(t)$ with our distribution from Part I, in the moments that are relatively prime to $q(q-1)|\Gamma|$.  In particular, we prove for every finite group $\Gamma$, in the $q\ra\infty$ limit, the prime-to-$q(q-1)|\Gamma|$-moments of the 
distribution of class groups of totally real $\Gamma$-extensions  of $\F_q(t)$
agree with the prediction of the Cohen--Lenstra--Martinet heuristics.
\end{abstract}

\maketitle

\section{Introduction}
\subsection{Overview}
Let $\Gamma$ be a finite group.  As $K$ varies over Galois $\Gamma$-extensions of $\Q$, the Cohen--Lenstra \cite{Cohen1984} and Cohen--Martinet \cite{Cohen1987} heuristics  give predictions for the distribution of the class groups $\Cl(K)$.  By class field theory,  $\Cl(K)$ is the Galois group of the maximal abelian, unramified extension of $K$.  We then naturally ask about the distribution of $\Gal(K^{\operatorname{un}}/K),$ where $K^{\operatorname{un}}$ is the maximal unramified extension of $K$.

The question of what the groups $\Gal(K^{\operatorname{un}}/K)$ can be has been studied for at least a century.  Golod and Shafarevich \cite{Golod1964} showed that these groups were sometimes infinite, answering a 40-year-old open question. In the years since,  many papers tried to understand under what conditions on $K$ can $\Gal(K^{\operatorname{un}}/K)$ be infinite; see \cite{Brumer1965,Schoof1986,Hajir2020} for some examples and other references.  The question of the possible structure of the $p$-class field tower group, the maximal pro-$p$ quotient of $\Gal(K^{\operatorname{un}}/K)$, has been an object of much study.  One central result is Shafarevich's \cite{Safarevic1963} proof of finite presentation and bound on the number of relators, but see also \cite{Taussky1937,Venkov1978,Boston2000a} for other examples and  references.  In  \cite{Mayer2016a} one can find a detailed review of the work  to determine the $p$-class field tower group for a given $K$.  Beyond the pro-$p$ quotient, the question of which finite groups can arise as quotients of $\Gal(K^{\operatorname{un}}/K)$ (and for what sort of $K$) has also attracted significant attention.  Fr\"{o}hlich \cite{Frohlich1962} proved that all $G$ arise as quotients for some $K$.
See \cite{Kim2015} for an overview of work in this area, including the papers \cite{Bhargava2016a,
Kedlaya2012, Uchida1970,Yamamoto1970a} finding infinitely many quadratic fields with unramified $A_n$ extensions for each $n$.  The papers \cite{Yamamura2001,Kim2017} give examples of and guides to the literature on determining the whole of $\Gal(K^{\operatorname{un}}/K)$ for given $K$.  

In this paper, we consider $\Gal(K^{\operatorname{un},2|\Gamma|'}/K)$, the maximal quotient of $\Gal(K^{\operatorname{un}}/K)$ whose order is relatively prime to $2|\Gamma|$ (in the profinite sense), as $K$ varies over totally real $\Gamma$-extensions.  
We show a few basic facts about $\Gal(K^{\operatorname{un},2|\Gamma|'}/K)$. These facts inspire a construction in the theory of random groups, and in the first part of our paper we prove that this construction gives a random group, i.e.,~a probability measure on the set of profinite groups.  We then conjecture that the groups $\Gal(K^{\operatorname{un},2|\Gamma|'}/K)$ equidistribute for this measure.  In the second part of the paper, we prove a theorem about these groups in the function field analog (where $\Q$ is replaced by $\F_q(t)$) as $q\ra\infty$, showing that all of the moments of the distribution on Galois groups match the moments of our conjectural distribution.

Since the publication of \cite{Cohen1984} and \cite{Cohen1987}, there have been several aspects of the Cohen--Lenstra--Martinet heuristics that have been found to be incorrect (or likely incorrect) \cite{Malle2008,Malle2010,Bartel2018}, and one naturally wonders whether there might be more issues.  However, our distribution on groups abelianizes to the Cohen--Lenstra--Martinet distribution, and thus our function field theorem provides good evidence that the Cohen--Lenstra--Martinet heuristics are correct when one avoids roots of unity in the base field, orders fields by the product of their ramified primes, and considers only the Sylow subgroups for finitely many primes at once (i.e.,~avoids or corrects all the known issues).

\subsection{Main results}
A \emph{$\Gamma$-group} is a profinite group with a continuous action of $\Gamma$.  A $\Gamma$-group $G$ is \emph{admissible} if it is generated topologically by the elements $\{g^{-1}\gamma(g) | g\in G, \gamma\in\Gamma \}$ and is of order prime to $|\Gamma|$.  We will see that $\Gal(K^{\operatorname{un},2|\Gamma|'}/K)$ is an admissible $\Gamma$-group (Definition~\ref{D:defG}, Proposition~\ref{prop:Galois_admissible}).  We will construct a group $\FF_n$,  \emph{the free admissible $\Gamma$-group on $n$ generators}.  Further, we will see that $\Gal(K^{\operatorname{un},2|\Gamma|'}/K)$ has what we call \emph{Property E}: for every odd prime $p\nmid |\Gamma|$ and every non-split central extension of $\Gamma$-groups
		\begin{equation}
			1 \to \Z/p\Z \to \widetilde{G} \to G \to 1,
		\end{equation}
		(where $\Z/p\Z$ has trivial $\Gamma$-action), any {$\Gamma$-equivariant} surjection $\Gal(K^{\operatorname{un},2|\Gamma|'}/K) \to G$ lifts to a {$\Gamma$-equivariant} surjection $\Gal(K^{\operatorname{un},2|\Gamma|'}/K) \to \widetilde{G}$ (see Section~\ref{S:PropE}). 
		A general admissible $\Gamma$-group, or quotient of $\FF_n$, will not always satisfy Property E. 
However, we show that a quotient of $\FF_n$ has Property E if and only if it is isomorphic to
$\FF_n/[r^{-1}\gamma(r)]_{r\in S,\gamma\in \Gamma }$ for some $S\subset \mathcal{F}_n$
(Proposition~\ref{P:GetE}).  We thus define a random group $X_{\Gamma,n}:=
\FF_n/[r^{-1}\gamma(r)]_{r\in S,\gamma\in \Gamma }$, where $(s_1,\dots,s_{n+1})$ is random from Haar measure on $\FF_n^{n+1}$
and $S=\{s_1,\dots,s_{n+1}\}$.  
(The choice to take $n+1$ elements in $S$ is mainly motivated by it being the choice that will give Theorem~\ref{T:MainFF}.)
Of course, no finite $n$ could possibly be sufficient, and the main theorem of the first part of our paper overcomes this difficulty.
\begin{theorem}\label{T:MainGroup}
For every finite group $\Gamma$, there is a probability measure $\mu_{\Gamma}$ on the set of isomorphism classes of admissible $\Gamma$-groups such that as $n\ra\infty$, the distributions of $X_{\Gamma,n}$ converge weakly to $\mu_\Gamma$.
\end{theorem}

In fact, we give $\mu_\Gamma$ very explicitly, including giving a formula in \eqref{E:measonbasic} for the measure on each basic open $U_{\calC,H}:=\{G\, |\, G^{\calC}\isom H^{\calC}\}$, where $\calC$ is a finite set of finite $\Gamma$-groups, and $G^{\calC}$ denotes the pro-$\calC$ completion of $G$ (see {\eqref{E:proC}}).  We will next precisely state our conjecture
that $\mu_\Gamma$ models the desired unramified Galois groups for a random $K$, in both the number field and function field cases.

Let $Q$ be either $\Q$ or $\F_q(t)$.  A \emph{$\Gamma$-extension} of $Q$ is a field extension $K/Q$ with a choice of isomorphism $\Gal(K/Q)\isom \Gamma$.  We say $K$ is \emph{totally real} if $K/Q$ is split completely over $\infty$.  
We define a ``bad'' integer $\bad$ such that $\bad=2|\Gamma|$ if $Q=\Q$ and $\bad=q(q-1)|\Gamma|$ if $Q=\F_q(t)$.  (The numbers $2$ and $q-1$ arise as the order of the  roots of unity in $Q$.)
Let $K^{\#}$ be the maximal unramified extension of $K$ such that
\begin{enumerate}
\item $K^\#$ is split completely at places of $K$ over $\infty$ (which happens by definition in the number field case), 
and 
\item all finite subextensions $L/K$ of $K^\#$ have degree relatively prime to $\bad$.
\end{enumerate}
We have that $\Gal(K^{\#}/K)$ is a profinite group with a continuous action of $\Gamma$ (see Definition~\ref{D:defG}).
We let $\rD K$ denote the norm of the radical of the ideal $\Disc(K/Q)$; when $Q=\Q$ this is the product of the ramified primes of $K/\Q$.  
Let $E_{\Gamma}(D,Q)$ be the set of isomorphism classes of totally real $\Gamma$-extensions of $Q$ with $\rD K=D$, which is finite by Hermite's Theorem.
We conjecture that as $K$ varies (ordered by $\rD K$), the $\Gamma$-groups $\Gal(K^{\#}/K)$ are  equidistributed according to the distribution $\mu_{\Gamma}$.  
 For concreteness, we specify the most important test functions for this equidistribution, the characteristic functions of basic opens and the functions whose averages are the moments.

\begin{conjecture}\label{C:Main}
Let $\Gamma$ be a finite group and $Q$ be either $\Q$ or $\F_q(t)$ for $q$ relatively prime to $|\Gamma|$.  Let $\calC$ be a finite set of finite $\Gamma$-groups all of whose orders are prime to $\bad$ (defined above), and $H$ a single such admissible $\Gamma$-group.  
Then
$$
\lim_{B\ra\infty}  \frac{\sum_{D\leq B} |\{K\in E_\Gamma(D,Q) \,|\, \Gal(K^{\#}/K)^{\calC} \isom H^{\calC}  \}  | }{\sum_{D\leq B} |E_\Gamma(D,Q)|} =\mu_\Gamma(U_{\calC,H})
$$
and
$$
\lim_{B\ra\infty} \frac{\sum_{D\leq B} \sum_{K\in E_\Gamma(D,Q)} |\Sur_\Gamma( \Gal(K^{\#}/K),H)|  }{\sum_{D\leq B} |E_\Gamma(D,Q)|} =\int_X 
|\Sur_\Gamma( X,H)| 
d\mu_\Gamma(X)=[H:H^\Gamma]^{-1}.
$$
\end{conjecture}

Now one may  naturally wonder if there are 
other facts about $\Gal(K^{\#}/K)$ besides admissibility and Property E that our model should take into account, or whether the enterprise of building a model that incorporates known facts should even provide a reasonable conjecture.  This is where our results in the function field analog provide the key insight.  In particular, given a $\Gamma$-group $H$ as above, we conjecture that the second statement of  Conjecture~\ref{C:Main} holds for every $Q=\F_q(t)$ with $q$ relatively prime to $|\Gamma||H|$ and $(q-1)$ relatively prime to $|H|$.  However, we will prove a modified statement
in which there is an additional limit as $q\ra\infty$ through allowable prime powers (where we write $D$ as $q^n$ and $B$ as $q^b$ below).

\begin{theorem}\label{T:MainFF}
Let $\Gamma$ be a finite group and $H$ be a finite admissible $\Gamma$-group.
Then
$$
\lim_{b\ra\infty} \lim_{\substack{q\ra\infty\\(q,|\Gamma||H|)=1\\(q-1,|H|)=1}} 
\frac{\sum_{n\leq b} \sum_{K\in E_\Gamma(q^n,\F_q(t))} |\Sur_\Gamma( \Gal(K^{\#}/K),H)|  }{\sum_{n\leq b} |E_\Gamma(q^n,\F_q(t))|}
 =\int_X 
|\Sur_\Gamma( X,H)| 
d\mu_{\Gamma}(X)=[H:H^\Gamma]^{-1},$$
where in the limit $q$ is always a prime power.
\end{theorem}

For every finite admissible $\Gamma$-group $H$, the second statement of Conjecture~\ref{C:Main} predicts an exact real number average. For every $H$, we obtain exactly the predicted average in the function field analog as $q\ra\infty$.  We believe this provides very strong evidence for the conjecture (see also below on the extent to which these averages are known to determine a unique distribution).

For a finite extension $K/\F_q(t)$, we write $\O_K$ for the integral closure of $\F_q[t]$ in $K$.  
By class field theory
$\Gal(K^{\#}/K)^{{\operatorname{ab}}}$ is the prime to $q(q-1)|\Gamma|$ part of $\Cl(\O_K)$.  An abelian $\Gamma$-group is admissible if and only if $(|H|,|\Gamma|)=1$ and $H^{\Gamma}=1$ (see Section~\ref{S:CM}).
Thus we have the following corollary of Theorem~\ref{T:MainFF}.

\begin{corollary}\label{C:MainFF}
Let $\Gamma$ be a finite group.
Let $H$ be a finite  abelian $\Gamma$-group with $(|H|,|\Gamma|)=1$ and $H^{\Gamma}=1$.
Then
$$
\lim_{b\ra\infty} \lim_{\substack{q\ra\infty\\(q,|\Gamma||H|)=1\\(q-1,|H|)=1}} 
\frac{\sum_{n\leq b} \sum_{K\in E_\Gamma(q^n,\F_q(t))} |\Sur_\Gamma(\Cl(\O_K),H)|  }{\sum_{n\leq b} |E_\Gamma(q^n,\F_q(t))|}
 =|H|^{-1},$$
where in the limit $q$ is always a prime power.
\end{corollary}
These are the averages predicted by Cohen and Lenstra \cite{Cohen1984} for abelian $\Gamma$ and Cohen and Martinet \cite{Cohen1990} for general $\Gamma$ (see \cite[Theorem 1.2]{WW}),
 except that Cohen, Lenstra, and Martinet order fields by their discriminants (which is known in general  not to give the described predictions \cite[Theorem 6.2]{Bartel2018}).  We have avoided roots of unity in the base field of order sharing a common factor with $H$, as Malle \cite{Malle2008,Malle2010} has suggested necessary (and one can see in our proof how this is necessary).  Our result only sees the part of the class group from finitely many primes at once, avoiding the issue raised by Bartel and Lenstra \cite[Section 4]{Bartel2018}.  Thus our result makes the strong suggestion that aside from these issues with the Cohen--Lenstra--Martinet heuristics, we should otherwise believe their predictions at primes not dividing $|\Gamma|$. 

We emphasize that in the abelian case, these averages (moments) are known to determine a unique distribution.
That is, if $\mu$ is a probability measure on finite admissible abelian $\Gamma$-groups, and for all such groups $H$ we have
$$
\sum_A \mu(A) |\Sur_\Gamma(A,H)|=|H|^{-1},
$$
then it has been shown by Wang and the second author \cite{WW} that $\mu$ must be equal to the Cohen--Lenstra--Martinet prediction for the distribution of the prime-to-$|\Gamma|$ part of the class groups of totally real $\Gamma$-fields.  In the non-abelian setting, Boston and the second author \cite[Theorem~1.4]{Boston2017} have shown that the moments determine a unique distribution in the case $|\Gamma|=2$ when we replace abelian groups with pro-$p$ groups.  We conjecture that this phenomenon of the moments determining the distribution extends even further into the non-abelian setting.  The results so far in this direction and our further conjectures highlight the strong extent to which Theorem~\ref{T:MainFF} provides  evidence for Conjecture~\ref{C:Main}.

When $|\Gamma|=2$, if we take the maximal pro-$p$ quotients of our distribution on random $\Gamma$-groups, we obtain the distribution that Boston, Bush, and Hajir \cite{Boston2018} predict for $p$-class field tower groups of real quadratic fields and thus our conjecture implies theirs.  Their predictions are based on theoretical considerations on the $p$-class field tower groups as well as extensive numerical computations.  In the imaginary quadratic case, as $q\ra\infty$, Boston and the second author \cite[Theorem 1.3]{Boston2017} have  proved a function field theorem for moments of $\Gal(K^\#/K)$, and the current paper for the first time provides a distribution on groups that has the moments from that theorem (see Theorem~\ref{thm:moments} in the case $u=0$).

As a final piece of evidence for Conjecture~\ref{C:Main}, we prove in Proposition~\ref{prop:pClassGroup} that every $p$-class field tower group actually occurs as the maximal pro-$p$ quotient of one of the groups we use in building the distribution $\mu_{\Gamma}$.

For every finite group $G$ of order relatively prime to $2|\Gamma|$, we conjecture a positive density of $\Gamma$-number fields have an unramified $G$-extension.  This is because we can always take the $\Gamma$-subgroup of $G^{|\Gamma|}$ (with $\Gamma$ acting in its regular representation on the factors) that is generated by elements with coordinates $g$ in the identity factor and $g^{-1}$ in the $\gamma$ factor (and $1$ in all other factors) for each $g\in G$ and $\gamma\in\Gamma$.  This subgroup is an admissible $\Gamma$-group 
(using Lemma~\ref{L:YYs}) and has quotient $G$ by projection onto the first factor.

\subsection{Previous work}
Conjecture~\ref{C:Main} generalizes the conjectures of Cohen--Lenstra \cite{Cohen1984} and Cohen--Martinet
 \cite{Cohen1990} on the distribution of class groups of number fields and of Boston, Bush, and Hajir on $p$-class field tower groups of quadratic fields \cite{Boston2017a,Boston2018}.
The only examples of proven non-trivial averages predicted by these conjectures are for the average $3$-part of the class groups of quadratic fields due to Davenport and Heilbronn \cite{Davenport1971} 
(and Datskovsky and Wright \cite{Datskovsky1988} over general global fields) and the average $2$-part of the class groups of cubic fields due to Bhargava \cite{Bhargava2005}.

Achter \cite{Achter2006} proved a $q\ra\infty$ function field analog of the Cohen--Lenstra heuristics for quadratic fields. Ellenberg, Venkatesh, and Westerland \cite{Ellenberg2016} proved new theorems on the homological stability of Hurwitz spaces, which allowed them to let $n\ra \infty$ before $q\ra\infty$ in proving a function field analog of the Cohen--Lenstra heuristics for imaginary quadratic fields, a major theoretical improvement. Work of Boston and the second author \cite{Boston2017,Wood2018} applied that homological stability to  prove analogous results for the Boston--Bush--Hajir heuristics for imaginary quadratic fields and the Cohen--Lenstra heuristics for real quadratic fields.  The paper  \cite{Boston2017} proves moments not just for the $p$-class field tower group but for the pro-odd quotient, but before this paper there was no conjectural distribution giving those moments.
The second author \cite{Wood2017a} has proven $q\ra\infty$ function field moments for $\Gal(K^{\operatorname{un}}/K)$ for quadratic $K/\F_q(t)$ without the restriction of taking the pro-odd quotient.  The papers \cite{Boston2017,Wood2017a} had a technical restriction on which moments they proved, and in this paper we overcome that sort of  restriction by proving new theorems on the existence of the necessary Hurwitz schemes.  Bhargava \cite{Bhargava2014b} has also proven the $A_4,A_5,S_3,S_4,S_5$  moments for $\Gal(K^{\operatorname{un}}/K)$ for quadratic number fields $K$ using parametrizations by prehomogenous vector spaces.

However, except for the paper \cite{Bhargava2005} on cubic extensions, all of the above results above are about $\Gal(K^{\operatorname{un}}/K)$ for quadratic extensions $K$.  Our paper addresses $\Gamma$-extensions $K/\F_q(t)$ for any finite group $\Gamma$.  This has required new ideas on the conjectural side to understand the possible kinds of $\Gamma$-action, and the development of new techniques for our function field proofs to understand precisely the components of the relevant Hurwitz schemes for all finite groups $\Gamma$.

When $|\Gamma|=p$ and one considers the maximal pro-$p$ quotient of $\Gal(K^{\operatorname{un}}/K)$, orthogonally to the work in this paper, there are additional tools available, starting with genus theory, and  more is known, including Smith's \cite{Smith2017} recent remarkable work proving  
Gerth's \cite{GerthIII1987a} extension of the Cohen--Lenstra heuristics to the $2$-primary part of the class group, after Fouvry and Kl\"{u}ners \cite{Fouvry2006a} had proven the distribution on the $4$-ranks.  See \cite{Alberts2016,Alberts2016a,Klys2017,Wood2017a} for some known $2$-group moments of $\Gal(K^{\operatorname{un}}/K)$ for quadratic fields $K$, and \cite{Klys2016a,Koymans2018} on the $p$-primary parts of the class groups of  $\Z/p\Z$-fields.

\subsection{Methods and outline of the paper}
In Part I of the paper, we prove Theorem~\ref{T:MainGroup}, find the moments of our limiting distribution, and relate it to other known distributions.  In Part II of the paper, we prove Theorem~\ref{T:MainFF}.  

In Section~\ref{S:NTfacts}, we use number theoretic methods to show the facts, admissibility and Property E, about $\Gal(K^\#/K)$ that are the motivation for the model and conjecture of our paper.   In Section~\ref{S:MakeGroup}, we construct the \emph{free admissible group on $n$ generators} and the $n$-generated versions of our random groups, and show that $p$-class field tower groups arise in our model.

In Sections~\ref{sect:defofmu} and \ref{S:countadd}, we prove the limit of these $n$-generated random groups exist.  The methods of these sections build upon earlier work \cite{LiuWood2017} of the first two authors on the limit of $F_n/[ r_1,\dots,r_{n+u} ]$ (the quotient of the profinite free group by Haar random relations) as $n\ra\infty$.  Essentially all of the ideas in that paper are necessary for this paper, and we give a very brief treatment of the methods originating in \cite{LiuWood2017} so that we are able to focus on the new ideas that arise from the $\Gamma$-action, admissibility, and taking relations in a more delicate way.
We reduce our probabilities of interest to questions of certain multiplicities of generalized irreducible representations occurring inside $\calF_n$.  These \emph{generalized irreducible representations}
are irreducible actions of groups $G\rtimes \Gamma$ on groups $T^m$, where $T$ is a simple group (abelian or non-abelian).  We relate these multiplicities to a count of certain maps, which we then count another way to eventually obtain an explicit expression for the multiplicities from which we can prove there is limiting behavior as $n\ra\infty$ in Section~\ref{sect:defofmu}. 
 Then, it remains to show that there is a limiting probability distribution, i.e.,~that there is no escape of mass.  The critical observation for this is that the multiplicities above can be controlled in terms of $\Gamma$-chief factor pairs of groups, a concept we introduce generalizing the composition series of a group.  This fact gives us the analytic power necessary to prove there is no escape of mass in Section~\ref{S:countadd}  and to determine the moments in Section~\ref{S:moments}. In Section~\ref{S:other}, we show that our random group has abelian quotient with the Cohen--Lenstra--Martinet distribution and pro-$p$ quotient with the Boston--Bush--Hajir distribution.  

In Section~\ref{S:relation}, we reframe Conjecture~\ref{C:Main} in terms of counting certain extensions of $Q$.  In Section~\ref{S:FFC}, we show how the existence of Hurwitz schemes with certain properties (including statements about their Frobenius fixed components) would imply Theorem~\ref{T:MainFF}.
In Section~\ref{S:Construction}, building on work of Wewers \cite{Wewers1998}, we prove the existence of the required Hurwitz schemes in algebraic, analytic, and topological versions.  
Then in Section~\ref{S:countcomp}, we prove the necessary statement on the Frobenius fixed components of these Hurwitz schemes to imply Theorem~\ref{T:MainFF}.  We use a component invariant introduced by Ellenberg, Venkatesh, and Westerland, and the majority of the work in Section~\ref{S:countcomp} is to understand the possible values of this component invariant.  We eventually use geometry of numbers to count the possible values of the component invariant, obtaining an expression with an unevaluated constant.  This gives us asymptotics for the numerator and denominator of   
Theorem~\ref{T:MainFF}, both with unevaluated constants.  Even though we do not obtain the exact constant in either the numerator or the denominator, we are able to prove an exact relationship between those constants to give Theorem~\ref{T:MainFF}.

\subsection{Further directions}
For every finite set $\calC$ of $\Gamma$-groups (of order relatively prime to $|\Gamma|$), the pro-$\calC$ completion
$\Gal(K^{\operatorname{un}}/K)^{\calC}$ is a finite admissible $\Gamma$-group and
 Conjecture~\ref{C:Main} gives a concrete prediction (e.g.~in terms of their presentations) for which $\Gamma$-groups can arise as  $\Gal(K^{\operatorname{un}}/K)^{\calC}$ for a positive proportion of $\Gamma$-extensions $K/\Q$.  It would be interesting to know if any  groups outside this list can be shown to be $\Gal(K^{\operatorname{un}}/K)^{\calC}$ for some $K$, or whether one can prove that they can never occur.  

There are other known properties of $\Gal(K^{\operatorname{un}}/K)$, e.g.~that every finite index subgroup has finite abelianization, that our random model does not take as input.  Our measure $\mu_{\Gamma}$ is supported
on groups with finite abelianization (see Section~\ref{S:CM}).  It would be interesting to know if
$\mu_{\Gamma}$ is supported on groups all of whose finite index subgroups have finite abelianization.
See \cite[Section 3.1]{Boston2000a} for a sample list of other known properties for which one could ask an analogous question. One could similarly ask about conjectural properties, e.g.~whether $\mu_\Gamma$ is supported on groups satisfying the unramified Fontaine--Mazur conjecture \cite{Fontaine1995}.  These questions are in the domain of random groups.

In this paper, we avoid two parts of $\Gal(K^{\operatorname{un}}/K)$---the part related to primes dividing $|\Gamma|$, and the part related to primes dividing the roots of unity in the base field.  Our function field techniques do not require either of these restrictions, but without the restrictions the moments obtained are different.  For example, in \cite{Wood2017a}, the second author finds the double limit in  Theorem~\ref{T:MainFF} when $|\Gamma|=2$ for many $H$ including $|H|$ even.  It remains, however, to find a probability distribution on groups that has these moments.  This obstacle is one in the theory of random groups, and we plan to address this in future work.

There are many further directions in which one could generalize our conjectures and theorems, including other signatures, other base fields, allowing ramification at a fixed set of primes.  We hope that
this paper provides a framework for such generalizations.

As far as extending the conjectures to replace our $Q$ with a general global field $K_0$  as a base field, we see there as being three regimes.  
The first regime would be the cases in which we would expect our same conjectures to hold.  This includes when $K_0$ is an imaginary quadratic field with trivial class group (all known to have no unramified extensions),  or a  function field with one place designated as ``infinity'' and no unramified extensions where infinity is split completely.  In these cases, the admissibility and Property E arguments in our paper go through as written.  In these number field cases, our conjectured distribution would still abelianize to the Cohen--Martinet predictions (Cohen and Martinet did not include function fields in their conjectures).  Though, we currently have significantly less evidence in these cases---no function field theorem, and no number field experiments.

The second regime would be cases in which the base field $K_0$ has no unramified extensions that are split completely at the infinite places.
In the function field case, we get to arbitrarily mark certain places as infinite and we use these to model archimedean places by requiring, when we require an extension to be unramified, to also require these places to split completely.  Our work leads naturally to a (different) conjectured distribution for these cases.  In this case, our proof of admissibility still holds, but the proof of Property E fails.  However, just as we force selected places in the function field case to behave like archimedean places, we could actually do a similar thing in the number field case.  If we pick $k$ rational primes, and also call them ``pseudo-infinite,''  then, for a $\Gamma$-extension $K$ of $\Q$, we could ask for the distribution of $\Gal(K^{\sharp}/K)$, where $K^{\sharp}$ is the maximal extension of $K$ unramified everywhere, split completely at pseudo-infinite places, and of degree prime to $2|\Gamma|$.

On the one hand, there is a natural conjectural distribution for $\Gal(K^{\sharp}/K)$, because $\Gal(K^{\sharp}/K)$ is  the quotient of $\Gal(K^{\#}/K)$ by the Frobenius at our $k$ pseudo-infinite places.  The conjecture here would follow from assuming Frobenius elements should be modeled by a Haar random element of $\Gal(K^{\#}/K)$.  In the abelianization (class group case), there are conjectures and theoretical evidence for this conjecture in \cite{Klagsbrun2017a,Klagsbrun2017, Wood2018}.   
On the other hand, $\Q$ with $k$ pseudo-infinite places is a natural analog of a number field with $k+1$ infinite places. 
So in conclusion, one naturally conjectures, for base fields $K_0$ with $k+1$ infinite places and with no unramified extensions split completely at the infinite places, that,  for ``totally real'' $\Gamma$-extensions $K/K_0$, the $\Gamma$-group $\Gal(K^{\#}/K)$ is distributed like the quotient of the random $\Gamma$-group with distribution $\mu_\Gamma$  by $k$ Haar random elements.  (Note ``totally real''  just means split completely at infinite places, which may be complex.)

In general, such groups will not have Property E, and it is just as well because in these cases we expect that $\Gal(K^{\#}/K)$ does not generally have Property E.
However, one can prove that in these cases $\Gal(K^{\#}/K)$ is the quotient of a Property E group by $k$ elements 
 at every level $\mathcal{C}$, which provides theoretical motivation for this conjecture.  

The third regime would be base fields $K_0$ that have unramified extensions themselves.  Of course, for an extension $K/K_0,$ this usually provides 
automatic unramified extensions of $K$, that would be essentially constant even as $K$ varies.  In the class group situation, Cohen and Martinet handle this by only making conjectures of the kernel of the norm map $\Cl_K \ra \Cl_{K_0}$.  
Let $B= \Gal(K_0^{{\#}}/K_0)$.
We would like to see conjectures on $\Gal(K^{{\#}}/K)$ not simply as a random group with an action of $\Gamma$, but as a random extension of $B$ with an action of $\Gamma$
(where $\Gamma$ acts trivially on $B$). 
Now the property of admissibility 
 is replaced by $B$-admissibility, which would require that the largest 
$\Gamma$-invariant quotient of the extension of $B$ is in fact $B$.  It will require further work to determine what should replace Property E in these cases and how exactly the distribution should be built. However, this paper and the perspective above provide a natural check on this work as follows.  
If you have a conjectured distribution, and then take $k$ non-infinite places whose Frobenius elements generate $B$, and call them pseudo-infinite, then you will now be in the second regime above where there is a natural conjecture that follows from this paper.  So the random groups for the third regime should likely have the property that if you take such a random quotient (where each Frobenius is modeled by a Haar random element in the appropriate fiber over $B$) of it, the result agrees with the distribution given above in the second regime.

Another important further direction is the collection of numerical evidence for Conjecture~\ref{C:Main}.
  The already known numerical evidence includes the current numerical evidence for Cohen--Martinet over $\Q$, and the work of Boston, Bush, and Hajir \cite{Boston2018} (which includes good numerical evidence for the pro-$p$ quotient when $\Gamma=\Z/2\Z$).  
It would be useful to have  more published
  numerical evidence for Cohen--Martinet over $\Q$.  Some of this is included in Malle's work \cite{Malle2008, Malle2010}. In general these numerics are useful not just to test to conjectures but also to give a sense of what one should expect in terms of speed of convergence of the numerics
   (and also see the next paragraph).  

One of the smallest cases in which  we predict a moment that is not predicted by Cohen--Martinet or Boston--Bush--Hajir is when $\Gamma=\Z/4\Z$ and $H$ is the Heisenberg group mod $3$ (SmallGroup(27,3)).    Here we predict the $H$-moment (i.e., the average of $\Sur_{\Gamma}(\Gal(K^{\operatorname{un}}/K),H)$ over totally real $\Gamma$ fields $K$) should be $[H:H^\Gamma]^{-1}=1/9.$  However, $H^{\operatorname{ab}}$ is a $2$-dimensional irreducible representation $V$ of  $\Z/4\Z$ over $\F_3$.  Since $1\ra\Z/3\Z \ra H\ra V\ra 1$ is a non-split central extension of $\Z/4\Z$ groups, by our proof that $\Gal(K^\#/K)$ has Property E, it follows that if $\Cl_K$ has a quotient isomorphic to $V^m$, then 
$\Gal(K^{\#}/K)$ has $H^m$ as a quotient.  More precisely, if $G$ is an admissible group with Property E, one can prove that 
$|\Sur_{\Z/4\Z}(G,H)|=|\Sur_{\Z/4\Z}(G,V)|$.  This means that the Cohen--Martinet $V$-moment plus our proof of Property E for $\Gal(K^{\#}/K)$ implies the non-abelian $H$-moment.  So even the Cohen--Martinet numerics here would give evidence for a non-abelian moment.   This example also gives a sense of the power of Property E.

Another of the smallest cases to consider is when $\Gamma=(\Z/2\Z)^2$ with the same $H$ (Heisenberg group mod 3) (with the faithful admissible action).  We predict the $H$-moment is $1/27$, and this is a nice target for computation.  
If one wants to consider non-abelian $\Gamma$, for  $\Gamma$ the dihedral group of order $8$, there is an admissible action of the same $H$, and  the predicted moment is $1/9$.  

Another interesting computational target would be the smallest non-nilpotent moment, that is for the group $H=(\Z/5\Z)^2\rtimes \Z/3\Z$ (SmallGroup(75,2)) with $\Gamma=\Z/2\Z$, in which case our predicted moment is $1/15$.  
The computation might be able to proceed along similar lines to that in the appendix of \cite{Wood2017a}.
For the smallest example where $\Gamma$ is not a $2$-group, for $\Gamma=S_3$, there is an admissible action on the Heisenberg group modulo $5$, and the predicted moment is $1/125$.

\subsection{Other remarks}

For $K/\F_q(t)$,
  if we define
$K^{\operatorname{un},\infty}$ to be the maximal extension of $K$ unramified everywhere and split completely at all places above $\infty$, then $\Gal(K^{\operatorname{un},\infty}/K)^{\operatorname{ab}}\isom\Cl(\O_K)$, and $\Gal(K^{\operatorname{un}}/K)^{\operatorname{ab}}\isom\Pic(C_K)(\F_q)$ (where $C_K$ is the smooth projective curve over $\F_q$ associated to $K$).  Hence when generalizing class group heuristics, it makes sense to consider the former, though the latter is also of interest and is possibly approached via similar methods.

The Cohen--Lenstra--Martinet conjectures were originally made with fields ordered by discriminant, and here we order them by the radical of their discriminant.  The generally better statistical behavior of the latter order was noticed by the second author in \cite{Wood2010}.  Bartel and Lenstra \cite[Theorem 1.2]{Bartel2018} found that the Cohen--Lenstra--Martinet prediction for the $3$-part of class groups of $\Z/4\Z$-fields is incorrect for fields ordered by discriminant, and based on \cite{Wood2010} they suggested instead ordering fields by the radical of their discriminant.

It is a subtle question what test functions should be allowed in a conjecture like Conjecture~\ref{C:Main} that makes precise the conjectured equidistribution.  Cohen, Lenstra, and Martinet, in their conjectures on the abelianization \cite{Cohen1984,Cohen1990}, say that one can use any ``reasonable'' function ``probably including all non-negative functions.''   However, observations made by Bartel, Lenstra and Poonen \cite[Theorem 1.7]{Bartel2018} suggest that this is probably too broad a class.
Bartel and Lenstra \cite[Theorem 1.1]{Bartel2018} also showed, as a consequence of the Iwasawa Main Conjecture, that the Cohen--Martinet conjecture does not hold when one considers the part of the class groups from all primes not dividing $|\Gamma|$.  Bhargava, Kane, Lenstra, Poonen, and Rains \cite[Section 5.6]{Bhargava2015b}, on a related problem on statistics of elliptic curves, also suggest that one should not consider all primes at once.  We have erred on the side of caution by stating 
Conjecture~\ref{C:Main} for the two simplest kinds of test functions involving only finitely many primes, but we expect it to hold more broadly.  See \cite[Section 5.6]{Bhargava2015b} and 
\cite[Section 7]{Bartel2018} for some possible precise notions for the class of allowable test functions.

\subsection{Basic definitions and notation}
Throughout the paper $\Gamma$ is a finite group.

A \emph{$\Gamma$-group} is a profinite group with a continuous action of $\Gamma$.  
In this paper, whenever we talk about homomorphisms of profinite groups or $\Gamma$-groups, we always mean continuous homomorphisms.  By an action  of a group $\Gamma$ on a profinite group $G$, we always mean a continuous action, i.e., an action such that $\Gamma \times G \ra G$ is continuous.   When we say a profinite group is generated by some elements, we always mean topologically generated.
When we refer to $|G|$ or the order of a profinite group $G$, we mean the supernatural number that is the order in the profinite sense.

 A morphism of $\Gamma$-groups $H\ra H'$  is a homomorphism of profinite groups $\phi: H\ra H'$ such that for all $h\in H$ and $g\in \Gamma$ we have $g(\phi(h))=\phi(g(h))$.
A $\Gamma$-subgroup of a $\Gamma$-group is a subgroup closed under the $\Gamma$ action.  A $\Gamma$-quotient of a $\Gamma$-group $H$ is the image of a surjective morphism
$H\ra H'$ of $\Gamma$-groups.
We write $\Hom_\Gamma$, $\Sur_\Gamma$ and $\Aut_\Gamma$ to represent the sets of $\Gamma$-equivariant homomorphisms, surjections and automorphisms respectively.
When $H$ is a $\Gamma$-group, we say that $H$ is an \emph{irreducible} $\Gamma$-group if it is non-trivial and has no proper, non-trivial $\Gamma$-subgroups.  
If $x_1,\dots$ are elements of a $\Gamma$-group  $F$,  we say that $F$ is \emph{$\Gamma$-generated} by 
 $x_1,\dots$ if $F$ is generated by the elements $\{ \gamma(x_i)\}_{\gamma\in \Gamma, i\geq 1}$.
 If $x_1,\dots$ are elements of a $\Gamma$-group $F$, we write $[x_1,\dots]$ for the normal closed $\Gamma$-subgroup of $F$ generated by $x_1,\dots$, and when the relevant $\Gamma$ or $F$ might be unclear, we add them as subscripts to the notation. For a $\Gamma$-group $H$, we write $H\rtimes \Gamma$ to be the semidirect product induced by the given $\Gamma$-action on $H$ and let $h_\Gamma(H):=|\Hom_\Gamma(H,H)|.$
 
For a positive integer $n$, a \emph{pro-$n'$ group} (or \emph{profinite $n'$-group}) is a profinite group such that every finite quotient has order relatively prime to $n$.  A  \emph{$n'$-$\Gamma$-group} is a pro-$n'$ group with a continuous $\Gamma$-action.  

For a field $k$, we write $\overline{k}$ for a (fixed) choice of separable closure of $k$.  We let $\Gal_k:=\Gal(\bar{k}/k).$
A \emph{$\Gamma$-extension of $k$} is a Galois extension $K/k$ with a choice of isomorphism $\Gal(K/k)\isom \Gamma$.

Throughout the paper $q$  denotes a prime power and $\F_q$ the finite field of order $q$.
Also,  $Q$ denotes either $\Q$ or $\F_q(t)$,  and $\bad=2|\Gamma|$ if $Q=\Q$ and $\bad=q(q-1)|\Gamma|$ if $Q=\F_q(t)$.

\begin{center}
{\bf Part I: A random $\Gamma$-group}\end{center}

\section{Facts about $\Gal(K^\#/K)$}\label{S:NTfacts}

We recall that $K^\#$ is the maximal unramified everywhere, split completely at places above $\infty$, degree prime-to-$\bad$ extension of $K$.
In this section, we explain the facts from number theory that motivate our conjecture about $\Gal(K^\#/K)$,
 where $K/Q$ is a $\Gamma$-extension.  
Recall that a $\Gamma$-group $G$ is \emph{admissible} if it is generated topologically by the elements $\{g^{-1}\gamma(g) | g\in G, \gamma\in\Gamma \}$ and is of order prime to $|\Gamma|$.
  We show $\Gal(K^\#/K)$ has a $\Gamma$-action, is admissible, and satisfies Property E.

For a $\Gamma$-extension $K/Q$, 
and a pro-$|\Gamma|'$ extension $L/K$ that is Galois over $Q$,
by the Schur--Zassenhaus theorem \cite[Theorem 2.3.15]{Ribes2010}
we can choose a section $s \colon  \Gamma \ra \Gal(L/Q)$,
 which gives $\Gal(L/K)$ an action of $\Gamma$ (via $s$ and conjugation) and an isomorphism $\Gal(L/Q)\isom \Gal(L/K) \rtimes \Gal(K/Q)$.  
 Also by the Schur--Zassenhaus theorem,  all choices of such a section are conjugate, and thus all choices would lead to isomorphic $\Gamma$-groups $\Gal(L/K)$.  

\begin{definition}\label{D:defG}
For each $\Gamma$-extension $K/Q$, we choose an arbitrary section $s \colon  \Gamma \ra \Gal(K^\#/Q)$, and
thus define an action of $\Gamma$ on $\Gal(K^\#/K)$.
\end{definition}

\subsection{Admissibility}

\begin{proposition}\label{prop:Galois_admissible}
	Let $K$ be a $\Gamma$-extension of $Q$ and $L$ an unramified pro-$|\Gamma|'$ extension of $K$ that is Galois over $Q$ and split completely at all places above $\infty$. Then $\Gal(L/K)$ is an admissible $\Gamma$-group.
\end{proposition}

\begin{proof}
	Let $G=\Gal(L/K)$. Since $G$ is a pro-$|\Gamma|'$ group, the Galois group of $L/Q$ is isomorphic to $G\rtimes \Gamma$. 
	Let $N$ be the closed subgroup of $G$ generated by elements $\{g^{-1}\sigma(g) \,\mid\, g \in G, \sigma \in \Gamma\}$. Then $N$ is a normal $\Gamma$-subgroup of $G$, because for any $h \in G$ we have $h^{-1} g^{-1}\sigma(g) h = (gh)^{-1} \sigma(gh) \cdot (h^{-1} \sigma(h))^{-1}$, and for any $\gamma \in \Gamma$ we have $\gamma(g^{-1}\sigma(g))=(g^{-1}\gamma(g))^{-1} (g^{-1}\gamma\sigma(g))$.
	 Let $\overline{G}$ be the quotient $G/N$. Then the induced action of $\Gamma$  on $\overline{G}$ is trivial, so the quotient of $G\rtimes \Gamma$ by $N$ is the direct product $\overline{G}\times \Gamma$.  Thus the fixed field $L^N$ has $\Gal(L^N/Q)=\overline{G}\times \Gamma$.
Every inertia group of $L/Q$, as well as the decomposition group at $\infty$, has order dividing $|\Gamma|$, and hence relatively prime to  $|G|$.
Thus the fixed field $L^{N\Gamma}$  is an unramified $\overline{G}$-extension of $Q$ that is also split completely at $\infty$.
	 Therefore, $\overline{G}$ has to be trivial which implies that $G=N$, the closed subgroup generated by the elements $\{g^{-1}\sigma(g) \,\mid\, g\in G, \sigma \in \Gamma\}$.
\end{proof}

\subsection{Property E}\label{S:PropE}

\begin{definition}
Let $C$ be a full subcategory of the category of  $\Gamma$-groups.   A $\Gamma$-group $H$ (in $C$) has \emph{Property E} (for $C$) if,
for every non-split central extension 
		\begin{equation}\label{eq:extension}
			1 \to \Z/p\Z \to \widetilde{G} \to G \to 1,
		\end{equation}
		in $C$, where $p\in \Z$ is prime and $\Z/p\Z$ has trivial $\Gamma$-action, any $\Gamma$-equivariant surjection $H \to G$ lifts to a $\Gamma$-equivariant surjection $H \to \widetilde{G}$.
\end{definition}

\begin{theorem}\label{T:PropE}
	Let $K/Q$ be a $\Gamma$-extension. Then $\Gal(K^\#/K)$
	satisfies Property E in the category of  
		$\bad'$-$\Gamma$-groups.
\end{theorem}

The authors discovered the importance of Property E from a result of Nigel Boston and Michael Bush
 that $(\Z/4\Z)^2$
 never occurs as a  $2$-class field tower group of a $\Z/3\Z$-extension of $\Q$.  
 Since $\Q$ has $2$ roots of unity,  this example is not  covered by Theorem~\ref{T:PropE}, but trying to understand their result led us to the realization that Property E
 was  a fundamental fact about $\Gal(K^\#/K)$.

 To prove Theorem~\ref{T:PropE}, we first need the following facts on Galois characters.

\begin{lemma} \label{lem:Galois-character}
Let $p\in \Z$ be a prime not dividing  the order of the roots of unity in $Q$.
Given, for each place $v$ of $Q$, a homomorphism $\phi_{v}\colon  \Gal_{Q_v} \to \Z/p\Z$, such that only finitely many are ramified, there exists a homomorphism $\phi\colon  \Gal_{Q} \to \Z/p\Z$ agreeing with $\phi_v$ on the inertia subgroup for each finite place $v$
and agreeing with $\phi_\infty$  on $\Gal_{Q_\infty}$.
\end{lemma}

\begin{proof}
The statement
 of the following lemma  in the case $Q=\Q$ is  \cite[Proposition~2.1.6]{Serre1992},  and the proof is similar in the case $Q=\F_q(t).$
\end{proof}

\begin{lemma}\label{lem:embedding}
	Let $G$ be a  $|\Gamma|'$-$\Gamma$-group.  Let $L/Q$ be a $(G\rtimes \Gamma)$-extension and let $K:=L^G$.  Assume that $L/K$ is unramified everywhere and split completely at all places over $\infty$.  
	Let $p\in \Z$ be a  prime such that $p\nmid \bad$.
If there is a non-split central  extension of $\Gamma$-groups as in \eqref{eq:extension}, 
	 then there exists a field extension $\widetilde{L}$ of $L$ such that $\Gal(\widetilde{L}/Q)\simeq \widetilde{G}\rtimes \Gamma$, and $\Gal(\widetilde{L}/Q) \to \Gal(L/Q)$ agrees with the map in \eqref{eq:extension} and the identity map of $\Gamma$,  and $\widetilde{L}/K$ is unramified everywhere and split completely at places over $\infty$.  
\end{lemma}

\begin{proof} 
First we consider the case when $G$ is finite.
Let $N=\Z/p\Z$.
	The lemma describes an embedding problem
	\begin{equation}\label{eq:global_embedding}
	\begin{tikzcd}
		 & & & \Gal_{Q} \ar["\pi", two heads]{d} \ar[dashed]{dl} & \\
		1 \ar{r} & N \ar{r} & \widetilde{G}\rtimes \Gamma \ar["\varphi"]{r} & G \rtimes \Gamma \ar{r} & 1,
	\end{tikzcd}
	\end{equation}
	where the 
	 field fixed by $\ker \pi$ is $L$. Let $v$ be a nonarchimedean place of $Q$. Then we fix an injection $\Gal_{Q_v} \hookrightarrow \Gal_{Q}$, and consider the local embedding problem associated to $v$. 
	
Let $S:=\pi(\Gal_{Q_v})$, and let $G':=G\cap S$. We claim that there exists a subgroup $\Gamma'$ of $S$ such that $S = G' \times \Gamma'$. Since $G'=\ker(S\ra \Gamma)$, we have that $G'$ is a Hall subgroup of $S$.
Since $L/K$ is unramified,  any inertia subgroup of $\Gal(L/Q)$ at $v$ has trivial intersection with $G$.
 Thus the inertia subgroup $I$ of $S$ is a normal subgroup such that $\gcd(|I|, |G'|)=1$, hence $G'$ injects into $S/I$. Since $S/I$ is cyclic, so is $G'$.  Moreover, since $G'$ is a Hall subgroup of $S/I$, and we have $S/I \simeq G' \times C$, where $C$ is a cyclic group.  Let $\Gamma'$ be the kernel of the projection of $S$ onto the $G'$ factor after the projection onto $S/I$.  So $\Gamma'$ is a normal subgroup of $S$ that is a complement to $G'$ and hence $S= G'\times \Gamma'$.

	So from \eqref{eq:global_embedding}  we obtain an embedding problem of the local field $Q_{v}$:
	\begin{equation}\label{eq:local_embedding}
	\begin{tikzcd}
		 & & & \Gal_{Q_v} \ar["\pi' ", two heads]{d} \ar[dashed]{dl} & \\
		1 \ar{r} & N \ar{r} & E \ar["\varphi' "]{r} & G' \times \Gamma' \ar{r} & 1,
	\end{tikzcd}
	\end{equation} 
where $\pi'$ is the restriction of $\pi$ to $\Gal_{Q_v}$, $E=\varphi^{-1} (S)$ and $\varphi'= \varphi|_E$. 
The short exact sequence in \eqref{eq:local_embedding} is central because the one in \eqref{eq:global_embedding} is central. Since $p\nmid |\Gamma'|$ and the $\Gamma'$ action on $N=\Z/p\Z$ induced by conjugation in $E$ is trivial, the preimage $\varphi'^{-1}(\Gamma')$ is isomorphic to $N \times \Gamma'$, so $\Gamma'$ lifts to a subgroup (also denoted by $\Gamma'$) that is a normal Hall subgroup of $\varphi'^{-1}(\Gamma')$, and hence is a characteristic subgroup. Then $\Gamma'$ is normal in $E$ and $E = \varphi'^{-1}(G')\times \Gamma'$. By composing $\pi'$ with projection maps to $G'$ and $\Gamma'$ we obtain $\pi_1\colon  \Gal_{Q_v} \to G'$ and $\pi_2\colon  \Gal_{Q_v} \to \Gamma'$.  Since $\pi_1$ is unramified,  $\pi_1$ can lift to a homomorphism $\widetilde{\pi}_1\colon  \Gal_{Q_v}\to \varphi'^{-1}(G')$. So $\widetilde{\pi}_1 \times \pi_2$ defines a solution of the embedding problem \eqref{eq:local_embedding}.
	
	In the case that $Q=\Q$, it follows by $N\simeq \Z/p\Z$ for odd $p$ that the embedding problem \eqref{eq:global_embedding} is solvable at the infinite place. By the assumption $\chr{Q}\neq p$, the local-global principle of central embedding problems \cite[Corollary~10.2]{Malle-Matzat} implies that \eqref{eq:global_embedding} has a solution. We denote one solution by $\rho\colon  \Gal_Q \to \widetilde{G} \rtimes \Gamma$, and for each place $v$ we let $\rho_v:=\rho|_{\Gal_{Q_v}}$. By \cite[\S~2]{Neukirch1979}, for a fixed $v$, all the solutions of the local embedding problem at $v$ form a principal homogeneous space over $\Hom(\Gal_{Q_v}, N)$, so there exists a character $\phi_v\colon \Gal_{Q_v} \to N$ such that the homomorphism $\rho_v \phi_v\colon  \Gal_{Q_v} \to \widetilde{G} \rtimes \Gamma$, defined as $\rho_v \phi_v(x)=\rho_v(x) \phi_v(x)$, is the map $\widetilde{\pi}_1 \times \pi_2$ described in the previous paragraph for finite $v$, and is the trivial map for infinite $v$. Then the image of the inertia (and the decomposition group at infinity) under $\rho_v \phi_v$ does not meet $N$. Then using Lemma~\ref{lem:Galois-character} we obtain a global character $\phi\colon  \Gal_{Q} \to N$ such that the map $\rho \phi$, which has to be surjective by the nonsplitness of \eqref{eq:extension}, satisfies the desired conditions in the lemma.
	Note if we have two homomorphisms $\phi_1,\phi_2 : \Gal_{Q} \ra \widetilde{G}\rtimes \Gamma$ satisfying the conditions of the lemma, then they differ by a homomorphism $\Gal_{Q} \ra N$ that is unramified at all places, and in particular, there are only finitely many choices for lifts 
	$\Gal_{Q} \ra \widetilde{G}\rtimes \Gamma$ satisfying our conditions.

If $G$ is profinite, then for every open subgroup $U$, the above argument gives a lift of $\Gal_Q\ra \widetilde{G}/(U N)\rtimes \Gamma$ to 
$\Gal_Q\ra \widetilde{G}/U \rtimes \Gamma$ satisfying the conditions on ramification and at $\infty$.  
These lifts form an inverse system of non-empty finite sets, and thus have non-empty inverse limit,  which gives a lift
of $\Gal_Q\ra G \rtimes \Gamma$ to 
$\Gal_Q\ra \widetilde{G} \rtimes \Gamma$ satisfying the lemma's conditions on ramification and at $\infty$.
\end{proof}

Theorem~\ref{T:PropE}  then follows from Lemma~\ref{lem:embedding}.

\section{A random admissible group with Property E}\label{S:MakeGroup}
In this section, we will define natural random $\Gamma$-groups that are admissible and satisfy Property E, with the aim of eventually constructing one that has a distribution that we would conjecture as the distribution for $\Gal(K^{\#}/K).$

\subsection{Definition of free admissible $\Gamma$-groups}\label{sect:Definition-Free-Admissible}

We define the \emph{free profinite $\Gamma$-group on $n$ generators}, $F_n(\Gamma)$, to be the free pro-$|\Gamma|'$ group on $\{x_{i,\gamma} \,\mid\, i=1, \dots, n \text{ and } \gamma\in \Gamma\}$ 
where $\sigma\in \Gamma$ acts on $F_n(\Gamma)$ by $\sigma(x_{i,\gamma})=x_{i,\sigma \gamma}$.
We fix a generating set $\{\gamma_1, \dots, \gamma_d\}$ of the finite group $\Gamma$ throughout the paper.
We let $x_i:=x_{i,\Id_{\Gamma}}$ and define $\calF_n(\Gamma)$ to be the closed $\Gamma$-subgroup of $F_n(\Gamma)$ that is generated (as a closed $\Gamma$-subgroup) by the elements
\[
	\{x_{i}^{-1}\gamma_j (x_{i})\, \mid\, i=1, \dots, n \text{ and } j=1, \dots, d\}.
\]
When the choice of the group $\Gamma$ is clear, we will denote $F_n(\Gamma)$ and $\calF_n(\Gamma)$ by $F_n$ and $\calF_n$ respectively.
For the purposes of this paper, one may take the generating set $\{\gamma_1, \dots, \gamma_d\}$ to be all of the elements of $\Gamma$, but we expect the flexibility of allowing any generating set will be useful in further applications, e.g.  for computations when $\Gamma$ is cyclic.

Since $x_i\not\in \calF_n$, 
it is not obvious that $\calF_n$ is admissible, but we will see that in fact it is, and we call $\calF_n$ \emph{the free admissible $\Gamma$-group on $n$ generators} (see  Remark~\ref{R:gens} and Corollary~\ref{C:FHoms}  for some motivation for the name).

\begin{lemma}\label{L:admiss}
For any finite group $\Gamma$, the $\Gamma$-group $\calF_n$ is admissible.  Moreover, there is a quotient map $\pi: F_n \ra \calF_n$ of $\Gamma$-groups, such that the composition of the inclusion $i\colon  \calF_n\sub  F_n$ with $\pi$ is the identity map on $\calF_n$.
\end{lemma}

\begin{remark}\label{R:gens}

\begin{enumerate}
\item	As a pro-$|\Gamma|'$-group (forgetting the $\Gamma$-action), it is clear from the definition that ${\calF}_n$ is free on generators 
	 $\{x_i^{-1} \sigma(x_i) \,\mid\, i=1, \dots, n \text{ and } \sigma \in \Gamma \backslash\{1\}\}$.
\item Since ${\calF}_n$ is a $\Gamma$-group quotient of $F_n$ by Lemma~\ref{L:admiss},
	 ${\calF}_n$ can be topologically generated as a $\Gamma$-group by $n$ elements. 
\item	 If $H$ is a $\Gamma$-group generated by elements $h_i^{-1}\gamma(h_i)$ with $h_i\in H$ (for $1\leq i\leq n$ and all $\gamma \in \Gamma$), then we see that the map $F_n\ra H$ sending $x_i\mapsto h_i$ gives a surjection  $\mathcal{F}_n\ra H$.

\end{enumerate}

\end{remark}

Now we prove some basic properties of $\calF_n$ including Lemma~\ref{L:admiss}.
For any $\Gamma$-group $G$, we define the following map of sets
\begin{eqnarray}\label{eq:map_alpha}
	Y\colon  G &\to& G^d\\
	g &\mapsto& (g^{-1}\gamma_1(g), g^{-1}\gamma_2(g), \dots, g^{-1}\gamma_d(g)).\nonumber
\end{eqnarray}
Consider the map $\Phi$ of sets from $G$ to homomorphic sections of $G\rtimes \Gamma \ra \Gamma$ defined by sending
$g$ to the conjugate, by $g$,  of the trivial section. By the Schur--Zassenhaus theorem, $\Phi$ is surjective.
Since $(g^{-1},1)(1, \gamma_i)(g,1)=(g^{-1}\gamma_i(g), \gamma_i)$, the $i$th coordinate of $Y(G)$ gives the $G$ coordinate of $\Phi(g)(\gamma_i)$.
 Since $\gamma_1, \dots, \gamma_d$ generate $\Gamma$,  we have $\Phi(g_1)=\Phi(g_2)$ if and only if $Y(g_1)=Y(g_2)$. 
 Thus we also have a natural bijection between $Y(G)$ and the sections of $G\rtimes \Gamma \to \Gamma$.
 
 \begin{lemma}\label{L:YYs}
Let $G$ be a  $|\Gamma|'$-$\Gamma$-group, and $E$ a $\Gamma$-subgroup of $G$. 
Then $Y(E)=Y(G)\cap E^d$.
\end{lemma}
\begin{proof}
Consider the  bijection described above between $Y(G)$ and the set of all homomorphic sections of $G\rtimes \Gamma \to \Gamma$.  
Since a section of $G\rtimes \Gamma \ra \Gamma$ that has image in $E\rtimes \Gamma$ is a section of $E\rtimes \Gamma\ra \Gamma$ and conversely, the lemma follows.
\end{proof}

\begin{proof}[Proof of Lemma~\ref{L:admiss}]
We use Lemma~\ref{L:YYs} with $G=F_n$ and $E=\calF_n$ to choose $y_i\in \calF_n$ such that
 $Y(x_i)=Y(y_i)$.  By comparing coordinates of $Y(x_i)$ and $Y(y_i)$, it is clear that $\calF_n$ is admissible.  We take $\pi(x_i)=y_i.$
\end{proof}

\begin{lemma}\label{L:rightcoset}
Let $G$ be a finite $\Gamma$-group.   For $g\in G$ and $y=Y(g)$, the fiber $Y^{-1}(y)$ is the right coset $G^{\Gamma}g$ of $G^{\Gamma}$ in $G$, and so $|G|=|G^{\Gamma}||Y(G)|$.
\end{lemma}
\begin{proof}
An element $h \in G$ satisfies $h^{-1}\gamma_i(h)=g^{-1}\gamma_i(g)$ if and only if $(hg^{-1})^{-1} \gamma_i(hg^{-1})=1$, so $h\in Y^{-1}(y)$ if and only if $hg^{-1}$ is fixed by $\gamma_i$ for all $i$.
\end{proof}

\begin{lemma}\label{lem:XYproperties}

	Let $G$ be a finite $|\Gamma|'$-$\Gamma$-group, and $N$ a  normal $\Gamma$-subgroup of $G$.  This defines an exact sequence
	\begin{equation}\label{eq:es}
		1 \to N \to G \overset{\pi}{\to} H \to 1.
	\end{equation}
	Then the quotient $H$ naturally obtains a $\Gamma$-action and
	\begin{enumerate}
		\item\label{item:1} $1\ra N^{\Gamma} \ra G^{\Gamma} \ra H^{\Gamma}\ra 1 $ is exact;
		\item\label{item:3} $|Y(N)|=|Y(G)|/|Y(H)|$;
		\item\label{item:4} $|Y(G) \cap \pi^{-1}(x)|=|Y(G)|/|Y(H)|$ for any $x \in Y(H)$, where we use $\pi$ to also denote the map $G^d\ra H^d$ induced by $\pi: G \ra H$.

	\end{enumerate}

\end{lemma}

\begin{proof}

We obtain a long exact sequence from \eqref{eq:es} of {pointed sets}
	\[
		1\ra N^{\Gamma} \ra G^{\Gamma} \ra H^{\Gamma}\ra H^1(\Gamma, N) \to \cdots.
	\]
	Elements of $H^1(\Gamma,N)$ correspond to $N$-conjugacy classes of splittings of $N\rtimes \Gamma\ra \Gamma$. By the Schur--Zassenhaus theorem, we have $H^1(\Gamma, N)=0$, thus \eqref{item:1} follows,
and then \eqref{item:3} follows from Lemma~\ref{L:rightcoset} and \eqref{item:1}.

	Let $x\in Y(H)$ and fix $h\in H$ and $g\in G$ such that $x=Y(h)$ and $g\in \pi^{-1}(h)$. 
The self-bijection on $G^d$ such that $z\mapsto (g,\dots,g)z(\gamma_1(g)^{-1},\dots,\gamma_d^{-1}(g))$ sends $Y(G)$ to $Y(G)$ and $\pi^{-1}(x)$ to $N^d$.  Thus we have that $|Y(G)\cap \pi^{-1}(x)|=|Y(G)\cap N^d|$. Then because $\pi$ maps every element in $Y(G)$ to an element in $Y(H)$, we have $|Y(G)|=\sum_{x\in Y(H)} |Y(G) \cap \pi^{-1}(x)|=|Y(H)| |Y(G) \cap N^d|$, proving \eqref{item:4}.
\end{proof}

\begin{corollary}\label{C:FHoms}
For any $\Gamma$-group $G$, the map
$$
Y(G)^n \ra \Hom_{\Gamma}(\calF_n,G)
$$
taking $(Y(g_1),\dots,Y(g_n))$ to the restriction of the map $F_n\ra G$ with $x_i\mapsto g_i$ is well-defined and bijective.
\end{corollary}
\begin{proof}
The map is well-defined and injective because $Y(g_i)=Y(g'_i)$ for all $i$ if and only if the maps $F_n \ra G$ taking $x_i \mapsto g_i$ and $x_i\mapsto g'_i$ agree on $\calF_n$.
 The map is surjective because we can use Lemma~\ref{L:admiss} to lift any homomorphism in 
 $\Hom_{\Gamma}(\calF_n,G)$ to one in $\Hom_{\Gamma}(F_n,G).$
\end{proof}

\begin{lemma}\label{lem:admissible_Y}
	If $G$ is an admissible $\Gamma$-group, then $G$ is $\Gamma$-generated by the coordinates of elements in $Y(G)$.  Thus, $\calF_n$ does not depend on the choice of the generating set $\{\gamma_1, \dots, \gamma_d\}$.
\end{lemma}

\begin{proof}
	Let $g\in G$ and $\sigma\in \Gamma$. Then $\sigma$ can be written as a word $\sigma_1 \sigma_2 \cdots \sigma_s$ where $\sigma_i \in \{\gamma_1, \dots, \gamma_d\}$ for each $i$, and so $g^{-1}\sigma(g)$ can be written as
	\[
		g^{-1}\sigma_1(g) [\sigma_1(g^{-1}\sigma_2(g))] \cdots [\sigma_1\sigma_2 \cdots \sigma_{s-1}(g^{-1}\sigma_s(g))],
	\]
	which finishes the proof.
\end{proof}

\begin{corollary}\label{cor:abelianizationCalF}
Let $Z$ be the pro-$|\Gamma|'$ completion of $\Z$, i.e., $Z=\prod_{p \nmid |\Gamma|} \Z_p$.  Then $\calF_n^{\operatorname{ab}}$ is isomorphic to 
$(Z[\Gamma]/(\sum_{g\in \Gamma} g))^n$,  as a   $\Gamma$-group. In particular, $(\calF_n^{\operatorname{ab}})^{\Gamma}=1$.
\end{corollary}

\begin{proof}
By Lemma~\ref{L:admiss}, we have that $\calF_n^{\operatorname{ab}}$ injects into $F_n^{\operatorname{ab}}$ (since $x,y\in F_n$ with $[x,y]\in \calF_n$ implies $[x,y]=\pi([x,y])=[\pi(x),\pi(y)]\in [\calF_n, \calF_n]$). Then 
$\calF_n^{\operatorname{ab}}$ is $I^n$ in $F_n^{\operatorname{ab}}=Z[\Gamma]^n$, where $I$ is the augmentation ideal of $Z[\Gamma]$.  Since  $|\Gamma|$ is invertible in $Z$, the group algebra $Z[\Gamma]$ can be decomposed as a $\Gamma$-module as $I \oplus (\sum_{g \in \Gamma} g)$, and the first statement in the corollary follows.   Any element of $I^\Gamma$ is fixed under left multiplication by $|\Gamma|^{-1} \sum_{g \in \Gamma} g$, and hence $I^\Gamma=0$, which implies the second statement of the corollary.
\end{proof}

\subsection{A random quotient}\label{S:randq}

Now by construction, any quotient of $\FF_n$ will be admissible.  However, not every quotient will satisfy Property E.  For example, letting $H=G\simeq (\Z/2\Z)^2$ with the unique (up to isomorphism) non-trivial $\Z/3\Z$-action and letting $\widetilde{G}$ be the quaternion group of order $8$ with the unique non-trivial $\Z/3\Z$-action, the identity map $H=G$ can not lift to a surjection $H \to \widetilde{G}$ because $|H|=4$ and $|\widetilde{G}|=8$.

We will next show that quotients of a certain form do satisfy Property E.
By a slight abuse of notation, for a $\Gamma$-group $G$ and a subset $\Rr$ of elements of $G$, we write $[ Y(\Rr) ]$ for the closed, normal $\Gamma$-subgroup of $G$ generated by the coordinates of all of the  $Y(r)$ for $r\in \Rr$.

Let $\calC$ be a set of isomorphism classes of finite $\Gamma$-groups, and $\overline{\calC}$ the smallest set of isomorphism classes of $\Gamma$-groups containing $\calC$ that is closed under taking finite direct product, $\Gamma$-quotients and $\Gamma$-subgroups. For a given $\Gamma$-group $G$, define the pro-$\calC$ completion of $G$ to be
\begin{equation}\label{E:proC}
	G^{\calC} =\varprojlim_{M} G/M,
\end{equation}
where the inverse limit runs over all closed normal $\Gamma$-subgroups $M$ of $G$ such that the $\Gamma$-group $G/M$ is contained in  $\overline{\calC}$. 
We call a $\Gamma$-group $H$ \emph{level $\calC$} if $H^\mathcal{C}=H$. 
One can see from this definition that taking pro-$\calC$ completion is a  functor from the category of $\Gamma$-groups to the category of $\Gamma$-groups of level $\calC$,  and a standard argument shows that if $F\ra H$ is a surjection then $F^\calC\ra H^\calC$ is also a surjection.

\begin{lemma}\label{lem:kernel-is-simple}
	Let $\widetilde{G}$ be an admissible $\Gamma$-group, and let $N$ be a normal sub-$\Gamma$-group of $\widetilde{G}$ such that $\Gamma$ acts trivially on $N$.
	Then $\widetilde{G}$ acts trivially on $N$ by conjugation.
\end{lemma}

\begin{proof}
Because $\Gamma$ acts trivially on $N$, we have $g^{-1} \gamma(g) n \gamma(g^{-1}) g = g^{-1} \gamma(g n g^{-1}) g = g^{-1} g n g^{-1} g=n$, for any $n\in N$, $\gamma\in \Gamma$ and $g\in \widetilde{G}$. Since $\widetilde{G}$ is admissible, this shows that $\widetilde{G}$ also acts trivially on $N$. 
\end{proof}

\begin{proposition}\label{P:GetE}
	Let $n\geq 1$ be an integer, $\calC$ a set of finite $|\Gamma|'$-$\Gamma$-groups, and $\phi: (\calF_n)^{\calC} \to H$ a $\Gamma$-equivariant surjection. Then $\ker\phi = [Y(\ker \phi)]_{(\calF_n)^{\calC}}$ if and only if $H$ satisfies Property E for the category of $\Gamma$-groups of level $\calC$.
\end{proposition}

This proposition shows that having relations generated by elements in $Y({(\calF_n)^{\calC}})$ is equivalent to Property E.
If desired, we can take $\calC$ to be the set of all finite $|\Gamma|'$-$\Gamma$-groups in Proposition~\ref{P:GetE}.

\begin{proof}
	``$\Longrightarrow$'': 
Suppose $\ker\phi = [Y(\ker \phi)]_{(\calF_n)^{\calC}}$, and we will show $H$ satisfies Property E.	
Let $\theta \colon H \to G$ be a surjective $\Gamma$-equivariant homomorphism, where $G$ has a non-split central $\Gamma$-group extension $\widetilde{G}$ as in \eqref{eq:extension}, and let  $\varpi=\theta \circ \phi $. We choose a $\pi \colon (F_n)^{\calC} \to (\calF_n)^{\calC}$ as given in Lemma~\ref{L:admiss}, and let $g_i= \varpi \circ \pi(x_i)$.
 Then we fix $\widetilde{g}_i \in \widetilde{G}$ which maps to $g_i$. Because \eqref{eq:extension} is non-split, elements $g_i^{-1}\gamma_j(g_i)$ generating $G$ implies that $\widetilde{G}$ is generated by elements $\widetilde{g}_i^{-1}\gamma_j(\widetilde{g}_i)$.
Therefore the map $(F_n)^{\calC} \to \widetilde{G}$ mapping $x_i$ to $\widetilde{g}_i$ induces a surjection $\rho\colon  (\FF_n)^{\calC} \to \widetilde{G}$. 
\begin{center}
	\begin{tikzcd}
		(\calF_n)^{\calC} \arrow["\phi", two heads]{r} \arrow["\rho", two heads]{d} \arrow["\varpi", two heads]{dr} & H \arrow["\theta", two heads]{d} \\
		\widetilde{G} \arrow[two heads]{r} & G
	\end{tikzcd}
\end{center}
Considering the short exact sequence $1\to \ker \rho \to \ker \varpi \to \Z/p\Z \to 1$, by $|Y(\Z/p\Z)|=1$ (since $\Gamma$ acts trivially on $\Z/p\Z$) and Lemma~\ref{lem:XYproperties}~\eqref{item:3}, we have $Y(\ker \varpi)=Y(\ker \rho)$.
We have that $Y(\ker\phi)\sub Y((\calF_n)^{\calC}) \cap (\ker \varpi)^d$, which is $Y(\ker \varpi)$ by Lemma~\ref{L:YYs}.
Thus  $Y(\ker\phi)\sub Y(\ker \rho)\sub (\ker \rho)^d$, and so $\rho$ factors through $H=(\FF_n)^{\calC}/[Y(\ker \phi)]$ and  the forward direction of the proposition follows.

	``$\Longleftarrow$'': Conversely, we suppose $H$ satisfies Property E for the category of level $\calC$ groups and $[Y(\ker \phi)] \subsetneq \ker \phi$. Define $M:=\ker \phi/[Y(\ker \phi)]$ and $E:=(\calF_n)^{\calC}/[Y(\ker \phi)]$. The group $\Gamma$ acts trivially on $M$. Since $E$ is a quotient of $(\calF_n)^{\calC}$, we have that $E$ is admissible and thus by 
	Lemma~\ref{lem:kernel-is-simple}, it follows that $E$ acts trivially on its subgroup $M$ by conjugation.  
	
	We choose a maximal proper subgroup $D$ of $M$, then by taking quotients modulo $D$ we have an exact sequence of $\Gamma$-groups
	\begin{equation}\label{eq:ses-PropE-H}
		1\to N \to \widetilde{H} \to H \to 1,
	\end{equation}
	where $N:=M/D$ and $\widetilde{H}:=E/D$. 
By choice of $D$, we have that $N$ is a simple group. Since $E$ acts trivially on $M$ by conjugation, $\widetilde{H}$ acts trivially on $N$ by conjugation and therefore $N$ is abelian.  
	Thus $N$ is isomorphic to $\Z/p\Z$ with the trivial $\widetilde{H}\rtimes \Gamma$-action for some prime $p$, so \eqref{eq:ses-PropE-H} is a central extension. If \eqref{eq:ses-PropE-H} is split, then $\widetilde{H}=N\times H$ contradicts the admissibility of $\widetilde{H}$; otherwise, $\eqref{eq:ses-PropE-H}$ is non-split, which contradicts the assumption that $H$ has Property E. Therefore we proved the other direction of the proposition.
\end{proof}

We will thus build random groups of the form $\FF_n/[Y(\Rr) ]$ with $S \subset \calF_n$. First, however, we need a topology on the set of isomorphism classes of  $\Gamma$-groups. We consider the set $\mathcal{P}$ of isomorphism classes of admissible  $|\Gamma|'$-$\Gamma$-groups $G$ such that $G^{\calC}$ is finite for all finite sets $\calC$ of finite $|\Gamma|'$-$\Gamma$-groups.
We define a topology on $\mathcal{P}$ in which the basic opens are
$U_{\calC, H}:=\{ X\in \mathcal{P} \,\mid\, X^{\calC}\isom H \text{ as $\Gamma$-groups}\}$, 
for each finite set $\calC$ of finite $|\Gamma|'$-$\Gamma$-groups and each finite  $|\Gamma|'$-$\Gamma$-group $H$.  

For $\calC$ a finite set of finite  $|\Gamma|'$-$\Gamma$-groups, we can see that $(\calF_n)^\calC$ is finite in the following way.   Let $W$ be the normal $\Gamma$-subgroup of elements of $\calF_n$ that vanish in every map to a $\Gamma$-group in $\calC$. The subgroup $W$ is of finite index since there are only finitely many maps to such groups from $\calF_n$.  Then as we enlarge $\calC$ by taking finite direct products, $\Gamma$-subgroups and $\Gamma$-quotients, elements of $W$ will also continue to vanish in maps to all those groups. Thus $W \sub \ker (\calF_n\ra (\calF_n)^\calC )$,  and in fact they are equal.
(To see elements of $W$ vanish in maps to a quotient of a group in $\mathcal{C}$, we use Lemma~\ref{L:admiss} to lift the map so that it factors through a map to a group in $\mathcal{C}$.)

\begin{definition}[Main definition of random groups for each $n$]
	For a positive integer $n$ and an integer $u>-n$, we define the random  $\Gamma$-group $X_{\Gamma,n,u}:=\FF_n/[Y(\Rr) ]$  where $\Rr$ is a random (multi)-set of $n+u$ elements of $\FF_n$ chosen independently from Haar measure.   
	We generally consider only a single $\Gamma$ at a time, and thus sometimes omit it from the notation.  When we omit $u$ from the notation, we are taking the case $u=1$.
\end{definition}

For integers $n\geq 1$ and $u>-n$, we have a measure $\mu_{u,n}$ on the $\sigma$-algebra of Borel sets of $\mathcal{P}$ such that
\[
	\mu_{n,u}(A)=\Prob(X_{u,n} \in A).
\]
However,  given  $n$, the groups in the support of $\mu_{u,n}$ have a bounded number of generators, which is not a feature we expect to be true of $\Gal(K^{\#}/K)$ as $K$ varies.
So for each integer $u$, we will define a measure $\mu_u$, at first as a measure on the algebra $\calA$ of sets generated by basic opens $U_{\calC, H}$. For $A\in \calA$, we will define
\begin{equation}\label{E:defmualgebra}
	\mu_u(A):=\lim_{n\to \infty}\mu_{n,u}(A).
\end{equation}
We will prove, in the next two sections, that the limit defining $\mu_u$ exists, and that $\mu_u$ extends uniquely to a Borel probability measure on $\mathcal{P}$.  This  (in the case $u=1$) is the distribution on $\Gamma$-groups that we conjecture models the distribution of $\Gal(K^{\#}/K)$ in Conjecture~\ref{C:Main}.
We remark that, when $|\Gamma|$ is odd,  our random groups are not necessarily pro-solvable.

\subsection{On the $p$-class field tower groups} 

We cannot prove  that every $\Gal(K^{\#}/K)$ can be expressed as $\FF_n/[ Y(\Rr) ]$ for some $n$ and $\Rr$, since in particular we do not know if $\Gal(K^{\#}/K)$ is finitely generated
(which is a question of Shafarevich, see \cite[Section 3]{Safarevic1963a} 
and \cite[Example 16.10.9 (c)]{Fried2008}).  (Our conjecture does not, at least immediately, imply that $\Gal(K^{\#}/K)$ is finitely generated because of the behavior of limits of measures.)  However, we will prove in Proposition~\ref{prop:pClassGroup} that the pro-$p$ quotients of $\Gal(K^{\#}/K)$, the $p$-class field tower groups, can always be expressed by groups that occur in our model. 

We now define notation that will be used for the lemma and proposition in this section and their proofs.
For a prime $p\nmid \bad$,
let $G_{K,p}$ be the Galois group of the maximal pro-$p$ extension of $K$ that is unramified everywhere and split completely at the places of $K$ over $\infty$, i.e.,~the maximal pro-$p$ quotient of $\Gal(K^{\#}/K)$. 
We define $\widetilde{G}_{K,p}$ to be the Galois group of the maximal unramified pro-$p$ extension of $K$, so $\widetilde{G}_{K,p}$ surjects onto $G_{K,p}$.
 
Recall that $G_{K,p}$ is finitely presented (its generator rank and relator rank can be explicitly computed \cite[Chapter~11]{Koch2002}) and admissible (see Proposition~\ref{prop:Galois_admissible}). When $n$ is sufficiently large, we can find $h_1, \dots, h_n\in G_{K,p}$ such that $G_{K,p}=[Y(\{h_1, \dots, h_{n}\})]$.

Consider a fixed $n$ as above. Let $(F_n)_p$ and $(\calF_n)_p$ denote the pro-$p$ completion of $F_n$ and $\calF_n$ respectively. We define $\varpi\colon  (\calF_n)_p \to G_{K,p}$ to be the $\Gamma$-equivariant surjection given by the map $F_n \to G_{K, p}$ mapping $x_i \mapsto h_i$.
\begin{enumerate}
	\item If $K$ is a number field, then we write $\alpha\colon  (F_n)_p \to (\calF_n)_p$ to be the quotient map induced by $\pi: F_n \to \calF_n$ in Lemma~\ref{L:admiss}, and we define $\rho: (F_n)_p \to \widetilde{G}_{K,p}$ to be the composition $\varpi \circ \alpha$.
	\item If $K$ is a function field, then we define $\alpha\colon  (F_{n+1})_p \to (\calF_n)_p$ to be induced by the composition of $\Gamma$-equivariant quotient maps $F_{n+1}\xrightarrow{x_{n+1}\mapsto 1} F_n \xrightarrow{\pi} \calF_n$. We fix $g_i \in \widetilde{G}_{K,p}$ that maps to $\varpi\circ \alpha(x_i)$ in $G_{K,p}$ for each $i=1, \dots, n$, and fix a generator $g_{n+1}$ of a decomposition subgroup at one infinite place of $K$. We define $\rho: (F_{n+1})_p \to \widetilde{G}_{K,p}$ by $x_i \mapsto g_i$ for $i=1, \dots, n+1$.
\end{enumerate}
In either case, we have $\alpha \circ \varpi$ is the composition of $\rho$ and the natural quotient map $\widetilde{G}_{K,p} \to G_{K,p}$. We let $F:=(F_{n})_p$ in the number field case and $F:=(F_{n+1})_p$ in the function field case.
Both $F$ and $(\calF_n)_p$ are free pro-$p$ groups, so $\rho$ and $\varpi$ are presentations of pro-$p$ groups $\widetilde{G}_{K,p}$ and $G_{K,p}$ respectively. We write
\[	N:=\ker \rho, \quad R:=\ker \varpi,	\]
\[	\text{and}\quad N^*:=[N,F]N^p, \quad R^*:=[R, (\calF_n)_p]R^p.	\]
Using the fact that $\rho|_{(\calF_n)_p}=\varpi$, we have $\alpha(N)=R$ and hence $\alpha(N^*)=R^*$. 

By the pro-$p$ group generation theory \cite[Corollary~3.9.3]{NSW2008}, we have several facts about $R/R^*$ and $N/N^*$.
A finite set of elements of $(\calF_n)_p$ generates $R$ as a normal, closed subgroup if and only if their images generate $R/R^*$. Since $R$ is $\Gamma$-closed, $R^*$ is preserved by the $\Gamma$-action by definition, so $R/R^*$ obtains a $\Gamma$-group structure from $R$, and $\Hom(R/R^*,\F_p)\simeq H^1(R,\F_p)^{(\calF_n)_p}$ as $\Gamma$-groups. The analogous statements in this paragraph are also true for $F$ and $N$ in place of $(\calF_n)_p$ and $R$.

\begin{lemma}\label{lem:N-structure}
Using the definitions and notation above, if $K$ does not contain any non-trivial $p$th roots of unity, then there is a surjection of $\F_p[\Gamma]$-modules $\F_p[\Gamma]^{n+1} \ra N/N^*$.
\end{lemma}

\begin{proof}
For an $\F_p[\Gamma]$-module $A$, we let $A^{\vee}$ denote the $\F_p[\Gamma]$-module $\Hom(A, \F_p)$.
So we have $(N/N^*)^{\vee} \simeq H^1(N, \F_p)^F$.
	Applying the Hochschild--Serre spectral sequence, we have 
	\begin{equation}\label{eq:H-S}
		H^1(F, \F_p) \xrightarrow{\Res} H^1(N, \F_p)^{F} \to H^2(\widetilde{G}_{K,p}, \F_p) \to H^2(F, \F_p) =1,
	\end{equation}
	where the last equality uses that $F$ is free as a pro-$p$ group.  The group $\Gamma$ naturally acts on every cohomology group in \eqref{eq:H-S} by acting on groups $F$, $N$ and $\widetilde{G}_{K,p}$. One can check that 
	\eqref{eq:H-S} is an exact sequence of $\F_p[\Gamma]$-modules. As $\gcd(p, |\Gamma|)=1$, the ring $\F_p[\Gamma]$ is semisimple, so every $\F_p[\Gamma]$-module is projective, and we have the following isomorphism of $\F_p[\Gamma]$-modules
	\begin{equation}\label{eq:decomposition}
		H^1(N, \F_p)^{F} \simeq \im \Res \oplus H^2(\widetilde{G}_{K,p}, \F_p).
	\end{equation}
	 By \cite[Theorem~11.3]{Koch2002},  we have an injection of $\F_p[\Gamma]$-modules
	\[
		H^2(\widetilde{G}_{K,p}, \F_p) \hookrightarrow (V/K^{\times p})^{\vee},
	\]
	where $V=\{\alpha\in K^{\times}\,\mid\, (\alpha)=\mathfrak{a}^p \text{ for some ideal $\mathfrak{a}$}\}$. 
	(We remark that $V/K^{\times p}$ is the $p$-Selmer group of a number field, see, e.g. \cite{Dummit2018}.)
	The homomorphism mapping $\alpha\in V$ to the ideal $\mathfrak{a}$ with $(\alpha)=\mathfrak{a}^p$ induces a surjection from $V/K^{\times p}$ to the $p$-torsion subgroup $\Cl(\calO_K)[p]$ of the class group whose kernel is isomorphic to $\mathcal{O}_K^{\times}/ \mathcal{O}_K^{\times p}$. Since $p \nmid |\Gamma|$,  we have that $\F_p[\Gamma]$ is a semisimple ring, so $V/K^{\times p}\simeq \calO^{\times}_K/\calO_K^{\times p} \oplus \Cl(\calO_K)[p] \simeq \calO^{\times}_K/\calO_K^{\times p} \oplus \Cl(\calO_K)/p\Cl(\calO_K)$.
	 Therefore we have a (noncanonical) $\Gamma$-equivariant injection
		\begin{equation}\label{eq:H2}
		H^2(\widetilde{G}_{K,p}, \F_p) \hookrightarrow (\mathcal{O}_K^{\times} / \mathcal{O}_K^{\times p} \oplus \Cl(\calO_K)/p\Cl(\calO_K))^{\vee}.
	\end{equation}
	
	For every $\varphi \in H^1(F, \F_p)$, the map $\Res$ takes $\varphi$ to $\varphi |_N$. Since the Frattini subgroup $\Phi(F):=[F,F]F^p$ of $F$ is contained in $\ker \varphi$ for every $\varphi$, one can show that $\Res \varphi \in (N/ (N\cap \Phi(F)))^{\vee}$. On the other hand, since $N/( N\cap \Phi(F))$ is a sub-$\F_p[\Gamma]$-module of $F/ \Phi(F)$, there exists a complement $M$ of $N/(N \cap \Phi(F))$, i.e.,
	\[
		F/\Phi(F) = N/(N\cap \Phi(F)) \oplus M,
	\]
We know that $M$ is isomorphic to the maximal $p$-elementary abelian quotient of $\widetilde{G}_{K,p}$.
		Writing the preimage of $M$ under the map $F\to F/\Phi(F)$ by $\widetilde{M}$, then every $\phi \in (N/ (N\cap \Phi(F)))^{\vee}$ lifts to $\widetilde{\phi} \in F^{\vee}$ with $\ker \widetilde{\phi}\supset \widetilde{M}$. 
	So we showed 
	\[
		\im \Res = (N/(N\cap \Phi(F)))^{\vee}.
	\]
So we have
	\begin{eqnarray}
		(N/N^*)^\vee &\hookrightarrow& (N/ (N\cap \Phi(F)) \oplus \calO_K^{\times}/ \calO_K^{\times p} \oplus \Cl(\calO_K)/ p\Cl(\calO_K))^{\vee}  \nonumber\\
		&\hookrightarrow& (N/ (N\cap \Phi(F)) \oplus \calO_K^{\times} / \calO_K^{\times p} \oplus  \widetilde{G}_{K,p}/ \Phi(\widetilde{G}_{K,p}))^{\vee}  \nonumber\\
		&\simeq& (F/ \Phi(F) \oplus \calO_K^{\times}/ \calO_K^{\times p})^{\vee}.\label{eq:decomposition-1} 
	\end{eqnarray}
	Because $\Cl(\calO_K)_p$ is the abelianization of $G_{K,p}$, the surjection $\widetilde{G}_{K,p} \to G_{K,p}$ induces a surjection $\widetilde{G}_{K,p}/ \Phi(\widetilde{G}_{K,p}) \to \Cl(\calO_K) / p\Cl(\calO_K)$ of $\F_p[\Gamma]$-modules, which implies the second injection in \eqref{eq:decomposition-1}.  
	When $K$ is a number field containing no non-trivial $p$th roots of unity, we have the injection of $\F_p[\Gamma]$-modules
	$\calO_K^{\times}/\calO_K^{\times p} \hookrightarrow \F_p[\Gamma]/(\sum_{g\in \Gamma} g)$
	by \cite[Cor.~8.7.3]{NSW2008}, and hence an injection
	\[
		(\calO_K^{\times}/ \calO_K^{\times p})^{\vee} \hookrightarrow (\F_p[\Gamma]/(\sum_{g\in \Gamma} g))^{\vee} \hookrightarrow 
	\F_p[\Gamma]	^{\vee}.
	\]
	When $K$ is a function field, we have $\calO_K^{\times}=\F^{\times}_{q^r}$ (where $\F_{q^r}^{\times}$ is the constant field of $K$) and using that fact that $K$ does not contain non-trivial $p$th roots of unity, it follows that $\calO_K^{\times}/\calO_K^{\times p}=1$. By \eqref{eq:decomposition-1}, and the fact that $F/\Phi(F)=\F_p[\Gamma]^m$, where $m$ is $n$ in the number field case and $n+1$ in the function field case, we prove the lemma.
\end{proof}

\begin{proposition}\label{prop:pClassGroup}
Let $K$ be a $\Gamma$-extension of $Q$ and $p \nmid \bad$ a prime, such that
	 $K$ does not contain any non-trivial $p$th roots of unity. Assume that $n$ is sufficiently large such that there exist $h_1, \dots, h_n \in G_{K,p}$ satisfying $G_{K,p}=[Y(\{h_1, \dots, h_n\})]$. Then there is a subset $S$ of $(\calF_n)_p$ of size $n+1$ such that $G_{K,p}\simeq (\calF_n)_p/ [Y(S)]$.
\end{proposition}

\begin{proof}
	Recall $\alpha(N)=R$ and $\alpha(N^*)=R^*$. 
	By Theorem~\ref{T:PropE}, we have that $G_{K,p}$ satisfies Property E in the category of pro-$p$ $\Gamma$-groups, and hence by
	 Proposition~\ref{P:GetE}, we conclude that $R=[Y(R)]$.  It follows that the $\Gamma$-invariant subgroup of $R/R^*$ is trivial.
	Then $\alpha$ induces a quotient map $N/N^* \to R/R^*$ of $\F_p[\Gamma]$-modules, and hence a quotient map $(\F_p[\Gamma]/(\sum_{g\in \Gamma} g))^{n+1} \to R/R^*$ by Lemma~\ref{lem:N-structure}. Let $\Id_{\Gamma}$ denote the image of the multiplicative identity of $\F_p[\Gamma]$ in $\F_p[\Gamma]/(\sum_{g \in \Gamma} g)$. Because $\F_p[\Gamma]/(\sum_{g\in \Gamma} g)$ is admissible and $\Gamma$-generated by coordinates of the image of $Y(\Id_{\Gamma})$, we have $R/R^*$ is $\Gamma$-generated by coordinates of $Y(\{r_1, \dots, r_{n+1}\})$ for $r_i \in R/R^*$, and hence $R$ is generated (as a closed normal $\Gamma$-subgroup of $(\calF_n)_p$) by coordinates of $Y(\{s_1, \dots, s_{n+1}\})$ where $s_i \in R$ is a chosen preimage of $r_i$ for each $i$.
\end{proof}

\section{Determination of $\mu_{u,n}$ and $\mu_{u}$}\label{sect:defofmu}
In this section and the next, we will prove the convergence of $\mu_{u,n}$ as $n\ra\infty$ to a probability measure $\mu_u$.  In this section, we prove that there is limiting behavior (specifically in Theorem~\ref{thm:probability_u,n} that the limits of measures of basic opens exist).  In the next section, we show that the limit distribution has total measure $1$, i.e.,~there is no escape of mass.  

In order to prove the limit in $n$ exists of the measures of basic opens, we will express these measures very explicitly in Theorem~\ref{thm:probability_u,n}.  Our approach builds on the strategy of the first two authors in \cite{LiuWood2017}, in which they consider the limit in $n$ of the quotient of the profinite free group on $n$ generators by $n+u$ independent Haar random relations.  Even though that is not a special case of what we do in this paper, it could be considered a toy model for our construction
and several of the key features of our argument are already present there.  The new features of the current problem include the $\Gamma$-action, and more importantly that we start with $\calF_n$ instead of the  free $\Gamma$-group $F_n(\Gamma)$, and we do not take Haar relations from $\calF_n$ but rather take our relations in a more subtle way.  These new features require new ideas, and in the proofs below we emphasize those ideas and treat more briefly the aspects of the argument that are similar to those in 
\cite{LiuWood2017}.

Let $n$ be a positive integer, $\calC$ a finite set of finite  $|\Gamma|'$-$\Gamma$-groups, and $H$  a finite $\Gamma$-group of level $\calC$.  Recall that $(\calF_n)^\calC$ is finite.
We consider a $\Gamma$-equivariant surjection $\varpi\colon (\calF_n)^\calC \to H$ and let $N:=\ker \varpi$.
Let $M$ be the intersection of all maximal proper $(\calF_n)^{\calC}$-normal $\Gamma$-subgroups of $N$, and let $F:=(\calF_n)^{\calC}/M$ and $R:=N/M$, which are naturally $\Gamma$-groups. Then we obtain a short exact sequence of $\Gamma$-groups
\begin{equation}\label{eq:ses}
	1\to R\to F \to H \to 1,
\end{equation}
which we will call the \emph{fundamental short exact sequence} associated to $\calC$, $n$, the $\Gamma$-group $H$, and the surjection $\varpi$.
The short exact sequence \eqref{eq:ses} is induced by
\begin{equation}\label{eq:ses2}
	1\to R \to F\rtimes \Gamma \to H \rtimes \Gamma \to 1.
\end{equation}
So one can check by \cite[Lemma~5.11]{LiuWood2017} that $R$ is a direct product of irreducible $F\rtimes \Gamma$-groups.  A large part of our work will be to compute which irreducible $F\rtimes \Gamma$-groups appear in $R$ and with what multiplicities.

We are now going to give notation for the multiplicities of the various irreducible $(F\rtimes \Gamma)$-groups in $R$, which we do in slightly different ways in the abelian and non-abelian cases (as it will be more convenient for the proofs, which require slightly different arguments).
Let $\calA_H$ be the set of isomorphism classes of non-trivial finite abelian irreducible $|\Gamma|'$-$(H\rtimes \Gamma)$-groups. Let $\calN$ be the set of isomorphism classes of finite $|\Gamma|'$-$\Gamma$-groups, that are isomorphic to $G^j$ for some finite simple non-abelian group $G$ and a positive integer $j$. 
For any $G\in \calA_H$, we define $m(\calC, n, \varpi, G)$ to be the multiplicity of $G$ in $R$ as an $H\rtimes \Gamma$-group under conjugation. For $G\in \calN$, let $G_i$ be the irreducible $F\rtimes \Gamma$-group structures one can put on $G$ that are compatible with the $\Gamma$-action on $G$. 
Then we define $m(\calC, n, \varpi, G)$ to be the sum (over $i$) of the multiplicity of the $G_i$ in $R$ as an $F\rtimes \Gamma$-group under conjugation. 
(This $G_i$ notation is only used in this paragraph.)
Later in this section, by Remark~\ref{rmk:R_not_depend_surj}, we will show that for fixed $\Gamma$-group $H$, the multiplicities $m(\calC, n, \varpi, G)$ do not depend on the choice of the $\Gamma$-equivariant surjection $\varpi\colon  (\calF_n)^\calC \to H$. So after that, we will denote $m(\calC, n, \varpi, G)$ by $m(\calC, n, H, G)$.

\subsection{$\mu_{u,n}$ and $\mu_u$ on basic open sets}\label{SS:41}
We will now consider the general situation (of which \eqref{eq:ses2} is a special case)
 where $F$ is a $\Gamma$-group and $R$ is a finite direct product of irreducible $|\Gamma|'$-$(F\rtimes \Gamma)$-groups, and express a probability that $R$ is generated by certain random elements in terms of the multiplicities of the irreducible $F\rtimes \Gamma$-groups in $R$.  This motivates us to understand these multiplicities.
If $H$ is a $G$-group, we define $h_G(H):=|\Hom_G(H,H)|.$

\begin{proposition}\label{prop:formula_probability}
	Let $F$ be a $\Gamma$-group. Let $G_i$ be finite irreducible $F\rtimes \Gamma$-groups for $i=1, \dots, k$, such that for $i\neq j$, we have that $G_i$ and $G_j$ are not isomorphic as $F\rtimes \Gamma$-groups. Let $R=\prod_{i=1}^k G_i^{m_i}$ for non-negative integers $m_i$. Then for integers $n,u$ with $n+u\geq 1$,
	\begin{eqnarray*}
		&& \Prob([ Y(\{r_1, \dots, r_{n+u}\}) ]_{F\rtimes \Gamma} =R)\\
		&=& \prod_{\substack{1\leq i \leq k \\ G_i \text{ abelian}}} \prod_{j=0}^{m_i-1} (1-h_{F\rtimes \Gamma}(G_i)^j |Y(G_i)|^{-n-u}) \prod_{\substack{1\leq i \leq k \\ G_i \text{ non-abelian}}} (1-|Y(G_i)|^{-n-u})^{m_i},
	\end{eqnarray*}
	where $r_1, \dots, r_{n+u}$ are independent, uniform random elements of $R$.
\end{proposition}

\begin{remark}\label{rmk:5.2}
	By Proposition~\ref{prop:formula_probability}, given a finite abelian irreducible $F\rtimes \Gamma$-group $G$, if we let $\mathfrak{m}$ be the maximal integer such that $G^{\mathfrak{m}}$ can be generated as an $F\rtimes\Gamma$-group by coordinates of $Y(g)$ for one element $g\in G^{\mathfrak{m}}$, then we have $h_{F\rtimes \Gamma}(G)^{\mathfrak{m}}=|Y(G)|$.
\end{remark}

\begin{proof}
A normal $F\rtimes \Gamma$-subgroup of $R$ is a direct product of normal subgroups of the $G_i^{m_i}$ \cite[Lem.~5.5]{LiuWood2017}.  Thus it remains to show that for a finite irreducible $F\rtimes \Gamma$-group $G$, an integer $m\geq 1$, and independent, uniform random elements $y_1, \dots, y_{n+u}$ of $G^m$,
	\begin{equation}\label{eq:formula_prob}
		\Prob([ Y(\{y_1, \dots, y_{n+u}\}) ] = G^m) = \begin{cases}
			\prod_{j=0}^{m-1} (1-h_{F\rtimes\Gamma}(G)^j |Y(G)|^{-n-u}) & \text{ if $G$ is abelian}\\
			(1-|Y(G)|^{-n-u})^m & \text{ if $G$ is non-abelian}.
		\end{cases}
	\end{equation}
For $y\in G$,  the tuple $Y(y)$ (as an element of $G^d$) has all coordinates trivial if and only if $y \in G^{\Gamma}$.   For a uniform random $y\in G$, this happens with probability $|Y(G)|^{-1}$ by Lemma~\ref{L:rightcoset}.
	For a non-abelian $G$, we know from \cite[Lem.~5.8]{LiuWood2017} that $G^m$ is normally generated by elements $x_1, \dots, x_k$ if and only if under the projection from $G^m$ onto each of the $m$ factors of $G$, at least one $x_i$ has non-trivial image, and the formula in \eqref{eq:formula_prob} follows.

	 For abelian $G$, we have that elements  generate $G^m$ if and only if they generate the projection onto the first $m-1$ factors $G^{m-1}$ and further that the projection (of the subgroup they generate) onto the last $G$ factor does not factor through the projection onto the first $m-1$ factors $G^{m-1}$ \cite[Lem.~5.8]{LiuWood2017}. Choices of $y_1,\dots,y_{n+u} \in G^m$ give a map $\mathcal{F}_n \ra G^m$ whose image is the subgroup generated the $y_i^{-1}\gamma(y_i)$, for
$\gamma\in 	F\rtimes \Gamma$.
Thus,  conditionally on $Y[(\{y_1, \dots, y_{n+u}\}) ]$ surjecting onto the first $m-1$ factors $G^{m-1}$ with a specific choice of  associated map $\mathcal{F}_n\ra G^{m-1}$, there are $h_{F\rtimes\Gamma}(G)^{m-1}$ choices of maps $\mathcal{F}_n\ra G$ that will factor through $G^{m-1}$.
 Using Lemma~\ref{L:rightcoset} and Corollary~\ref{C:FHoms} to count the choices of $y_i$ giving such maps, the formula in \eqref{eq:formula_prob} follows.	(See the proof of	\cite[Cor.~5.10]{LiuWood2017} for a more  detailed version of a similar argument.)
\end{proof}

Now to find our desired multiplicities $m(\calC, n, \varpi, G)$, we will count the same thing in two different ways.  
We require some definitions and a supporting lemma to describe what we will count.

For a given $\Gamma$-group $H$, an \emph{$H$-extension} is a $\Gamma$-group $E$ with a $\Gamma$-equivariant surjection $\pi\colon  E\to H$.  
For an admissible $\Gamma$-group $H$, we call an $H$-extension $(E, \pi)$ \emph{admissible} if $E$ is admissible.
A morphism between two $H$-extensions $(E, \pi)$ and $(E', \pi')$ is a $\Gamma$-equivariant isomorphism $f\colon E\to E'$ such that $\pi' \circ f = \pi$, and we write $\Sur_{\Gamma,H}(\pi,\pi')$ or $\Sur_{\Gamma,H}(E,E')$ for the surjective morphisms (and similarly for $\Aut_{\Gamma,H}$). 
If $(E,\pi)$ is an $H$-extension, a \emph{sub-$H$-extension} of $(E, \pi)$ is a $\Gamma$-subgroup $E'$ of $E$ such that $\pi|_{E'}$ is surjective.
For an $H$-extension $E$ (with the map to $H$ implicit),  let $\calE^E_H$ be the poset of sub-$H$-extensions of $E$, where $(E', \pi') \leq (E'', \pi'')$ if $(E' ,\pi')$ is a sub-$H$-extension of $(E'', \pi'')$.
We let $\nu(D,E)$ be the M\"obius function of this poset of $H$-extensions, i.e.,~for  an $H$-extension $E$, and a sub-$H$-extension $D$, we have
\begin{eqnarray*}
	\nu(E, E)&=& 1\\
	\nu(D, E) &=& -\sum_{\substack{D' \in \calE^E_H \\ D'\ne D}} \nu(D', E) \quad \text{if } D\neq E.
\end{eqnarray*}

\begin{lemma}\label{lem:count_surj}
	Let $n$ be a positive integer, $\calC$ a finite set of finite $|\Gamma|'$-$\Gamma$-groups, and $H$ a finite $\Gamma$-group of level $\calC$. Let $(\calF_n)^{\calC} \overset{\rho}{\to} H$ be an $H$-extension structure on $(\calF_n)^{\calC}$. Let $(E,\pi)$ be an $H$-extension. We have
	\[
		|\Sur_{\Gamma,H}(\rho, \pi)|=\begin{cases}
			\sum_{D\in \calE^E_H} \nu(D,E) \left(\frac{|Y(D)|}{|Y(H)|}\right)^n & \text{ if $E$ is level $\calC$ and admissible}\\
			0 & \text{otherwise}.
		\end{cases}
	\]
\end{lemma}

\begin{proof}
	By Lemma~\ref{L:admiss}, $(\calF_n)^\calC$ is admissible, so $|\Sur_{\Gamma,H}(\rho, \pi)|=0$ if $E$ is not of level $\calC$ or not admissible. If $E$ is level $\calC$ and $(D, \psi) \leq (E, \pi)$, then $\Gamma$-equivariant homomorphisms $(\calF_n)^\calC \to D$ exactly correspond to $\Gamma$-equivariant homomorphisms $\calF_n \to D$.  By Corollary~\ref{C:FHoms}, these correspond to elements of $Y(D)^n$. For a homomorphism $\calF_n \to D$ to be compatible with the maps to $H$, we need the element of $Y(D)^n$ to map to the element of $Y(H)^n$ that corresponds to $\rho$ via Corollary~\ref{C:FHoms}.
	 By Lemma~\ref{lem:XYproperties} \eqref{item:4}, we have that the number of homomorphisms $\calF_n \to D$ compatible with the maps to $H$ is $(|Y(D)|/|Y(H)|)^n$. 
So we have
	\[
		\left(\frac{|Y(D)|}{|Y(H)|}\right)^n = \sum_{\substack{(E', \phi)\in \calE^{(D, \psi)}_H 
		}} |\Sur_{\Gamma,H}(\rho, \phi)|.
	\]
	We then use M\"obius inversion to obtain the proposition.
\end{proof}

Next, for an irreducible $F\rtimes \Gamma$-group $G$, we will count $|\Sur_{F\rtimes \Gamma}(R,G)|$ in two ways.  First we will count it in the following proposition in terms of $H$-extensions with kernel $G$, and then in the proof of Corollary~\ref{cor:multiplicities} we will use a count of $|\Sur_{F\rtimes \Gamma}(R,G)|$ in terms of multiplicities $m(\calC, n, \varpi, G)$ to determine the multiplicities in terms of an expression involving counting $H$-extensions.
We now give some notation for the set of $H$-extensions that will arise.  For a $\Gamma$-group $H$, 
and an irreducible abelian $H\rtimes \Gamma$-group $G$, let 
  $\mathcal{E}(G,H)$ denote the set of isomorphism classes of admissible $H$-extensions $(E,\pi)$ such that $\ker \pi \isom G$ as $(H\rtimes \Gamma)$-groups.
  For a $\Gamma$-group $H$, 
and a finite non-abelian $\Gamma$-group $G$, let 
  $\mathcal{E}(G,H)$ denote the set of isomorphism classes of admissible $H$-extensions $(E,\pi)$ such that $\ker \pi \isom G$ as $\Gamma$-groups
and $\ker \pi$ is also an irreducible $E\rtimes \Gamma$-group.  In either case, let   $\mathcal{E}_\mathcal{C}(G,H)$ be the subset of $(E,\pi)\in\mathcal{E}(G,H)$
  such that $E$ is level $\mathcal{C}$.

\begin{proposition}\label{prop:count_surj_RtoG}
	Let $1\to R\to F\to H\to 1$ be a fundamental short exact sequence associated to a finite set $\calC$ of finite $|\Gamma|'$-$\Gamma$-groups, a positive integer $n$, an admissible $\Gamma$-group $H$, and a $\Gamma$-equivariant surjection $\varpi\colon  (\calF_n)^\calC \to H$.
	\begin{enumerate}
		\item\label{item:count_surj_ab} Let $G$ be an irreducible abelian $F\rtimes \Gamma$-group. Then
			\[
				|\Sur_{F\rtimes \Gamma}(R,G)| = |\Aut_{F\rtimes \Gamma}(G)| \sum_{(E, \pi)\in \mathcal{E}_\mathcal{C}(G,H)}
				\frac{\sum_{D\in \calE^E_H} \nu(D,E)\left(\frac{|Y(D)|}{|Y(H)|}\right)^n}{|\Aut_{\Gamma,H}(E,\pi)| }
			\]
			if the action of $F\rtimes \Gamma$ on $G$ factors through $F\rtimes \Gamma \to H\rtimes \Gamma$ and $|\Sur_{F\rtimes\Gamma}(R,G)|=0$ otherwise. 
		\item\label{item:count_surj_nonab} Let $G$ be a finite non-abelian $\Gamma$-group in $\calN$. Let $G_i$ be the pairwise non-isomorphic irreducible $F\rtimes\Gamma$-group structures on $G$, that are compatible with the $\Gamma$-action on $G$, for $1\leq i \leq k$ ($k$ may be 0). Then
			\[
				\sum_{i=1}^k \frac{|\Sur_{F\rtimes \Gamma}(R, G_i)|}{|\Aut_{F\rtimes \Gamma}(G_i)|}=\sum_{(E, \pi)\in \mathcal{E}_\mathcal{C}(G,H)}
				 \frac{\sum_{D\in \calE^E_H} \nu(D,E) \left(\frac{|Y(D)|}{|Y(H)|}\right)^n}{|\Aut_{\Gamma,H}(E, \pi)|}.
			\]
	\end{enumerate}
\end{proposition}

\begin{proof}
We have that 	$|\Sur_{F\rtimes \Gamma}(R,G)|$ is $|\Aut_{F\rtimes \Gamma}(G)|$ times the number of $F\rtimes \Gamma$-subgroups $M$ of $R$ such that $R/M$  is isomorphic to $G$ as an $F\rtimes \Gamma$-group. If $R/M$ is abelian, then the action of $F\rtimes \Gamma$ on $R/M$ factors through $H\rtimes\Gamma$ (because conjugation by elements from $R$ is trivial in $R/M$ as $R/M$ is abelian). 
We then have the number of $F\rtimes \Gamma$-subgroups $M$ of $R$ such that $R/M$ is isomorphic to $G$ as an $F\rtimes \Gamma$-group is
\begin{eqnarray*}
  && \sum_{(E, \pi)\in \mathcal{E}(G,H)}
   \#\left\{ F\rtimes\Gamma\text{-subgroups $M$ of $R$}\,\Bigg|\, 
  \begin{aligned}
  	&(F/M \ra H) \simeq (E,\pi)\\
	& \text{ as $H$-extensions}, 
  \end{aligned}\right\} \\
\end{eqnarray*}
where we recall that $H$ is a $\Gamma$-group, and ``$H$-extensions'' above refers to extensions of $\Gamma$-groups.
By the definition of $F$ and the irreducibility of $G$ as an $F\rtimes \Gamma$-group, we have that 
any surjection of $H$-extensions $(\calF_n)^{\calC}\ra H$ to $E\ra H$ factors through $F\ra H$
(as in \cite[Prop~8.7]{LiuWood2017}).
Thus the above sum is equal to
  \begin{align*}
\sum_{(E, \pi)\in \mathcal{E}(G,H)}
\frac{|\Sur_{\Gamma}((\calF_n)^{\calC}\ra H,\pi)|}{|\Aut_{\Gamma,H}(E,\pi)|}.
\end{align*}
 Then \eqref{item:count_surj_ab} now follows from applying Lemma~\ref{lem:count_surj} above.
 The proof of \eqref{item:count_surj_nonab} is similar, except we  put together all $H$-extensions $(E,\pi)$ whose kernel has the same $\Gamma$-group structure.
\end{proof}

\begin{corollary}\label{cor:multiplicities}
	Let $H$, $F$ and $R$ be as in Proposition~\ref{prop:count_surj_RtoG}. If $G\in \calA_{H}$, then
	\[
		\frac{h_{H\rtimes \Gamma}(G)^{m(\calC, n, \varpi, G)}-1}{h_{H\rtimes \Gamma}(G)-1}
		=\sum_{(E, \pi)\in \mathcal{E}_\mathcal{C}(G,H)}
		 \frac{\sum_{D\in \calE^E_H} \nu(D,E) \left(\frac{|Y(D)|}{|Y(H)|}\right)^n}{|\Aut_{\Gamma,H}(E,\pi)|}.
	\]
	If $G\in \calN$, then
	\[
		m(\calC, n, \varpi, G)
		=\sum_{(E, \pi)\in \mathcal{E}_\mathcal{C}(G,H)}
		\frac{\sum_{D\in \calE^E_H} \nu(D,E) \left(\frac{|Y(D)|}{|Y(H)|}\right)^n}{|\Aut_{\Gamma,H}(E,\pi)|}.
	\]
\end{corollary}

\begin{remark}\label{rmk:R_not_depend_surj}
	This corollary shows that $m(\calC, n, \varpi, G)$ does not depend on the choice of the $\Gamma$-equivariant surjection $\varpi\colon  (\calF_n)^{\calC} \to H$. Thus we will denote $m(\calC, n, \varpi, G)$ by $m(\calC, n, H, G)$.
\end{remark}

\begin{proof}
	We have that $R$ as an $F\rtimes \Gamma$-group can be written as $\prod_{i=1}^k G_i^{m_i}$ where $G_i$ are pairwise non-isomorphic irreducible $F\rtimes \Gamma$-groups. By \cite[Thm.~7.1]{LiuWood2017}, 
	\[
		|\Sur_{F\rtimes \Gamma}(R, G_i)|=
		\begin{cases}
			h_{F\rtimes \Gamma}(G_i)^{m_i} -1 & \text{ if $G_i$ is abelian}\\
			m_i |\Aut_{F\rtimes \Gamma}(G_i)| & \text{ if $G_i$ is non-abelian.}
		\end{cases}
	\]
	Then the corollary follows by Proposition~\ref{prop:count_surj_RtoG} and the definition of $m(\calC, n, \varpi, G)$ for $G\in \calA_{H}$ and $G\in \calN$.
\end{proof}

Now we are ready to put together the above results to give the explicit measures of basic opens.
Recall $\mathcal{E}_\mathcal{C}(G,H)$ was defined just before Proposition~\ref{prop:count_surj_RtoG} as a certain set of minimal extensions of $H$ by $G$.
We define for $G\in \calA_H$
\begin{equation}\label{E:lam1}
	\lambda(\calC, H, G):=
		 (h_{H\rtimes \Gamma}(G)-1) \sum_{(E, \pi)\in \mathcal{E}_\mathcal{C}(G,H)}
		 \frac{1}{|\Aut_{\Gamma,H}(E,\pi)|},
\end{equation}
and for $G\in \calN$
\begin{equation}\label{E:lam2}
	\lambda(\calC, H, G):=\sum_{(E, \pi)\in \mathcal{E}_\mathcal{C}(G,H)}
	\frac{1}{|\Aut_{\Gamma,H}(E,\pi)|}.
\end{equation}

\begin{theorem}\label{thm:probability_u,n}
	Let $\calC$ be a finite set of finite $|\Gamma|'$-$\Gamma$-groups and $H$ a finite level $\calC$ admissible $\Gamma$-group. Let $n\geq 1$ and $u>-n$ be integers. Then 
	\begin{eqnarray}
		&& \Prob((X_{u,n})^\calC \isom H \text{ as $\Gamma$-groups}) \label{eq:probability_u,n} \label{eq:formula_u,n} \\
		&=& \frac{|\Sur_{\Gamma}(\calF_n, H)|}{|\Aut_{\Gamma}(H)| |Y(H)|^{n+u}}\prod_{G\in \calA_H} \prod_{k=0}^{m(\calC, n, H, G)-1} (1-\frac{h_{H\rtimes \Gamma}(G)^k}{|Y(G)|^{n+u}}) \prod_{G\in \calN} (1-|Y(G)|^{-n-u})^{m(\calC, n, H, G)} \nonumber
	\end{eqnarray}
	and we have
	\begin{eqnarray}
		&& \lim_{n\to \infty} \Prob((X_{u,n})^\calC\isom H \text{ as $\Gamma$-groups})\label{eq:lim_formula_u,n}\\
		&=& \frac{1}{|\Aut_{\Gamma}(H)| |Y(H)|^u} \prod_{G\in \calA_H} \prod_{i=1}^{\infty}(1-\lambda(\calC, H, G) \frac{h_{H\rtimes \Gamma}(G)^{-i}}{|Y(G)|^u}) \prod_{G\in \calN} e^{-|Y(G)|^{-u} \lambda(\calC, H,G)}.\nonumber
	\end{eqnarray}
\end{theorem}
\begin{proof}
We have
	\begin{eqnarray}
		&& \Prob((X_{u,n})^{\calC} \isom H \text{ as $\Gamma$-groups}) \label{eq:sum_N}\\
		&=& \sum_{\substack{N \unlhd (\calF_n)^{\calC}, \text{ s.t. $N$ is closed} 
		\\ (\calF_n)^{\calC}/N \isom H \text{ as $\Gamma$-groups}}} \Prob([ Y(\{r_1, \dots, r_{n+u}\})]_{(\calF_n)^{\calC}} = N),\nonumber
	\end{eqnarray}
	where $r_1, \dots, r_{n+u}$ are independent random with respect to the Haar measure on $(\calF_n)^{\calC}$.
	For each closed normal $\Gamma$-subgroup $N$ of $(\calF_n)^{\calC}$,
by 
Lemma~\ref{L:rightcoset},
Lemma~\ref{lem:XYproperties} \eqref{item:3}
and Lemma~\ref{L:YYs}, 
it follows that
	\begin{equation}\label{eq:Prob-Y}
		\Prob(Y(\{r_1, \dots, r_{n+u}\}) \subset N^d) = |Y(H)|^{-n-u}.
	\end{equation}
	Let $1\to R\to F \to H \to 1$ be the fundamental short exact sequence induced by $(\calF_n)^\calC \to (\calF_n)^\calC/N\isom H$.  We claim
	\begin{eqnarray*}
		&& \Prob([ Y(\{r_1, \dots, r_{n+u}\})]_{(\calF_n)^{\calC}}=N \mid Y(\{r_1, \dots, r_{n+u}\}) \subset N^d)\\
		&=& \Prob([ Y(\{r'_1, \dots, r'_{n+u}\})]_F  =R),
	\end{eqnarray*}
	where $r'_i$ are independent uniform random elements of $R$. 
In the case we are considering $[ Y(\{r_1, \dots, r_{n+u}\})]_{(\calF_n)^{\calC}}$ is a 	$(\calF_n)^{\calC}$-normal $\Gamma$-subgroup of $N$, and thus is $N$ if and only if its quotient mod $M$ is $R$ (by the definition of the fundamental exact sequence \eqref{eq:fundshort}).  We can  use 
	Lemma~\ref{L:rightcoset} 
and Lemma~\ref{L:YYs} to see that for $r$ Haar random from $(\calF_n)^{\calC}$, the distribution of $Y(r)$ mod $M$ conditional on $Y(r)\sub N^d$ is the same as the distribution of $Y(r')$ for $r'$ uniform random from $R$. 
	Then by Proposition~\ref{prop:formula_probability} and Remark~\ref{rmk:R_not_depend_surj}, Equality~\eqref{eq:formula_u,n} follows from the fact that the number of subgroups $N$  that the sum \eqref{eq:sum_N} is taken over equals  $|\Sur_{\Gamma}(\calF_n, H)| |\Aut_{\Gamma}(H)|^{-1}$.
	
	We claim that there are only finitely many $G\in \calA_H \cup \calN$ such that $m(\calC, n, H, G)$ is nonzero. First, since any irreducible $F\rtimes \Gamma$-group factor of $R$	is characteristically simple, as a group it is isomorphic to a power of a simple group.  The factor, as a $\Gamma$-group, is in $\overline{\calC}$
and
 \cite[Cor.~6.12]{LiuWood2017} shows the underlying simple group is in the closure of the set of groups in $\mathcal{C}$ (forgetting the $\Gamma$-structure) under taking subgroups and quotients (and this is a finite set of groups).
	Then \cite[Lem.~6.1]{LiuWood2017}, applied to the quotient of \eqref{eq:ses2} by all \emph{other} 
	irreducible $F\rtimes \Gamma$-group factors of $R$, shows that the power of the simple group is
	 bounded by $|H\rtimes \Gamma|$, proving the claim. Therefore, to establish Equality~\eqref{eq:lim_formula_u,n}, it suffices to take the limit of a factor in \eqref{eq:formula_u,n} corresponding to each $G\in \calA_H \cup \calN$. 
	 For each $E,D$ in the sum in Corollary~\ref{cor:multiplicities}, since $E$ is admissible, we have that $Y(D)=Y(E)$ implies $E=D$.  Thus we have $|Y(D)|/|Y(H)|\leq |Y(E)|/|Y(H)| = |Y(G)|$ where the equality holds exactly when $D=E$.
	 Then by the definition of $\lambda(\calC, H, G)$ and Corollary~\ref{cor:multiplicities}, we see that $\lambda(\calC, H, G)$ is related to $m(\calC, n, H, G)$ as follows
	\begin{equation}\label{eq:lambda_m}
		\lambda(\calC, H, G) = \begin{cases}
			\lim\limits_{n\to \infty} h_{H\rtimes \Gamma}(G)^{m(\calC, n, H, G)}|Y(G)|^{-n} & \text{if }G\in \calA_H\\
			\lim\limits_{n\to \infty} m(\calC, n, H, G)|Y(G)|^{-n} & \text{if }G\in \calN.
		\end{cases}
	\end{equation}
	
	If there are no extensions $E$ in the sum in Corollary~\ref{cor:multiplicities}, then $m(\calC, n, H, G)=0$ for any $n$ and $\lambda(\calC, H, G)=0$. Otherwise $\lambda(\calC, H, G)>0$, and it follows by \eqref{eq:lambda_m} that $m(\calC, n, H, G)\to\infty$ as $n\to \infty$. So using \eqref{eq:lambda_m} and \cite[Lem.~8.11]{LiuWood2017}, we obtain the limit in Equality~\eqref{eq:lim_formula_u,n} for the factor corresponding to each $G\in \calA_H \cup \calN$.  Then \eqref{eq:lim_formula_u,n} will follow from the claim that 
	\begin{equation}\label{eq:lim-Y}
		\lim_{n\ra\infty} |\Sur_{\Gamma}(\calF_n, H)| |Y(H)|^{-n}=1.
	\end{equation}
	By Corollary~\ref{C:FHoms},
$|\Hom_{\Gamma}(\calF_n, H)|=|Y(H)|^n$, and the number of non-surjective homomorphisms is bounded by the sum of $|Y(H')|^n$ over all proper $\Gamma$-subgroups $H'$ of $H$.  Since $H$ is admissible, as above we have that $Y(H')=Y(H)$ implies $H'=H$, and thus the claim and theorem follow.
\end{proof}

\section{Countable additivity of $\mu_u$}\label{S:countadd}

In the last section we proved that the limit defining $\mu_u$ in \eqref{E:defmualgebra} exists.  In this section we will show that $\mu_u$ extends to a Borel probability measure, and in particular that there is no escape of mass.  They key ingredient will be to control the size of $\mu_{u,n}$ on the basic opens, and make sure the measures are not decreasing too fast in $n$.

\subsection{Chief factor pairs}
In this section, we will introduce a notion of $\Gamma$-chief factor pairs, which will be used in the main technical result to control the size of $\mu_{u,n}$. 
We first recall the definition of chief factor pairs from \cite{LiuWood2017}. A \emph{chief factor pair} is a pair $(M,A)$ such that $M$ is an irreducible $A$-group and the $A$-action on $M$ is faithful. Two chief factor pairs $(M_1, A_1)$ and $(M_2, A_2)$ are isomorphic if there is an isomorphism $\omega\colon  M_1 \to M_2$ such that the induced isomorphism $\omega^*\colon \Aut(M_1) \to \Aut(M_2)$ maps $A_1$ to $A_2$. If $M$ is a minimal normal subgroup of a finite group $G$, then we denote by $\rho_M$ the homomorphism 
\begin{eqnarray*}
	\rho_M\colon  G & \to & \Aut(M)\\
	g & \mapsto & (x\mapsto gxg^{-1})_{x\in M}.
\end{eqnarray*}
Let $G$ be a finite group and 
\[
	1=G_0 \lhd G_1 \lhd \cdots \lhd G_r=G
\]
a chief series of $G$ (a chain of normal subgroups such that $G_{i+1}/G_i$ is a minimal normal subgroup of $G/G_i$ for each $0\leq i\leq r-1$). Then $\CF(G)$ is defined to be the set consisting of all isomorphism classes of chief factor pairs $\left\{\left(G_{i+1}/G_i, \rho_{G_{i+1}/G_i}(G/G_i)\right)_{i=0}^{r-1} \right\}$. By \cite[Lem.~6.3]{LiuWood2017}, this set does not depend on the choice of the chief series of $G$. For a set $T$ of finite groups, we define $\CF(T):=\cup_{G\in T} \CF(G)$. 

\begin{definition}
	For a finite $|\Gamma|'$-$\Gamma$-group $G$, we define $\CF_{\Gamma}(G)$  to be the set of isomorphism classes of chief factor pairs as follows
	\[
		\CF_{\Gamma}(G):= \left\{ (M,A)\in \CF(G\rtimes \Gamma) \,\mid\, \gcd(|M|,|\Gamma|)=1\right\}.
	\]
	In other words, the set $\CF_{\Gamma}(G)$ is the set of chief factor pairs in $\CF(G\rtimes \Gamma)$ in which the first term comes from the normal subgroup $G$ of $G\rtimes \Gamma$.
	Chief factor pairs belonging to $\CF_{\Gamma}(G)$ are called \emph{$\Gamma$-chief factor pairs}.
	Let $\calC$ be a set of finite $|\Gamma|'$-$\Gamma$-groups. We define
	\[
		\CF_{\Gamma}(\calC):=\bigcup_{G\in \calC} \CF_{\Gamma}(G).
	\]
\end{definition}

\begin{lemma}\label{lem:CF_barC}
	If $\calC$ is a set of finite $|\Gamma|'$-$\Gamma$-groups and $\calC$ is closed under taking $\Gamma$-subgroups and $\Gamma$-quotients, then $\CF_{\Gamma}(\overline{\calC})=\CF_{\Gamma}(\calC)$.
\end{lemma}

\begin{proof}
	Since $\overline{\calC}$ is the closure of $\calC$ under taking finite direct products, $\Gamma$-subgroups and $\Gamma$-quotients, it suffices to show that none of these three actions creates new $\Gamma$-chief factor pairs not belonging to $\CF_{\Gamma}(\calC)$. First, it is clear that taking $\Gamma$-quotients and taking finite direct products do not create new chief factor pairs.
	
	Assume $J$ is a $\Gamma$-subgroup of $G$ for $G\in \overline{\calC}$ such that $\CF_{\Gamma}(G) \subseteq \CF_{\Gamma}(\calC)$. We want to prove $\CF_{\Gamma}(J) \subseteq \CF_{\Gamma}(\calC)$. Let
	\[
		1=G_0 \lhd G_1 \lhd \cdots \lhd G_r =G \lhd \cdots  \lhd G \rtimes \Gamma
	\]
	be a chief series of $G\rtimes \Gamma$ passing through the normal subgroup $G$. We can construct a chief series of $J\rtimes \Gamma$ that passes through $G_i\cap J$ for every $i=0, \dots, r$. The chief factor pairs achieved from the elements between $G_i\cap J$ and $G_{i+1}\cap J$ are achieved from the group $J/(G_i \cap J) \simeq (J G_i)/G_i$, which is a $\Gamma$-subgroup of $G/G_i$. Thus it is enough to consider the chief series at the positions between $1$ and $G_1 \cap J$.
Since $(G_1, \rho_{G_1}(G\rtimes \Gamma))\in \CF_{\Gamma}(G)\subseteq \CF_{\Gamma}(\calC)$, there is a $\Gamma$-group $G' \in \calC$ and a minimal subgroup $G_1'$ of $G'$ such that the chief factor pairs $(G_1, \rho_{G_1}(G\rtimes \Gamma))$ and $(G_1', \rho_{G_1'} (G'\rtimes \Gamma))$ are isomorphic, i.e., there exists $  \alpha\colon  G_1 \overset{\sim}{\to} G_1'$ such that $\alpha^*\colon  \Aut(G_1) \overset{\sim}{\to} \Aut(G_1')$ maps $\rho_{G_1}(G\rtimes \Gamma)$ to $\rho_{G_1'}(G'\rtimes\Gamma)$. Define $A:=\rho_{G_1}(J\rtimes \Gamma) = ((J\rtimes \Gamma) \cdot \Cen_{G\rtimes \Gamma}(G_1))/\Cen_{G\rtimes \Gamma}(G_1)$, 
where the notation $\Cen_{G\rtimes \Gamma}(G_1)$ represents the centralizer of the subgroup $G_1$ in $G\rtimes \Gamma$. 
By $\gcd(|G|,|\Gamma|)=\gcd(|G'|, |\Gamma|)=1$, we have $\alpha^*(\rho_{G_1}(G))=\rho_{G'_1}(G')$ because normal Hall subgroups are order-unique. 
 Then it follows by the Schur--Zassenhaus Theorem that $\alpha^*(\rho_{G_1}(\Gamma))$ is conjugate to $\rho_{G'_1}(\Gamma)$ in $\rho_{G'_1}(G'\rtimes \Gamma)$, and so by taking the preimages under $\rho_{G'_1}$, we have $\Gamma$ is conjugate to a subgroup $\Gamma'$ of $\rho_{G'_1}^{-1}(\alpha^*(\rho_{G_1}(\Gamma)))$ such that $\rho_{G'_1}(\Gamma')=\alpha^*(\rho_{G_1}(\Gamma))$.
We let $J':= \rho_{G_1'}^{-1} (\alpha^*(A)) \cap G'$,
and since $\alpha^*(\rho_{G_1}(J))\sub \rho_{G'_1}(G')$, we have
$\rho_{G'_1}^{-1}( \alpha^*(\rho_{G_1}(J)) )\sub J'$ and thus
  $\Gamma'$ is a complement of the normal subgroup $J'$ of $\rho_{G'_1}^{-1}(\alpha^*(A))$.
We then have $\rho_{G'_1}^{-1}(\alpha^*(A))=J'\rtimes \Gamma'$ and the  short exact sequence
\[
	1 \to \Cen_{G'\rtimes \Gamma}(G_1') \to J'\rtimes \Gamma' \to \alpha^*(A) \to 1.
\]
So the action of $J'\rtimes \Gamma'$ via conjugation on  $G_1'$, and hence on $G_1' \cap J'$, factors through $\alpha^*(A)$. We define $J'_1:=\alpha(G_1 \cap J)$. Because the conjugation action of $J\rtimes \Gamma$ on $G\rtimes \Gamma$ stabilizes $J$, we have the action of $A$ on $G_1$ stabilizes $G_1 \cap J$. So $\alpha$ maps $G_1\cap J$ with the action of $A$ isomorphically to $J'_1$ with the action of $\alpha^*(A)$. Then by $\rho_{G_1}(G_1\cap J) \subset A$, we have $\rho_{G'_1}(J'_1) \subset \alpha^*(A)$ and therefore $J'_1$ is a subgroup of $J'$. Every chief factor pair of $J\rtimes \Gamma$ achieved from positions between $1$ and $G_1 \cap J$ is also a chief factor pair of $J'\rtimes \Gamma'$ achieved via a series passing through $J_1'$. 
The group $G'$ has a second action of $\Gamma$ defined by the conjugation morphism $c_g$ (by an element $g\in G'\rtimes \Gamma$) taking $\Gamma \ra \Gamma'$ and then the conjugation action of $\Gamma'$ in $G'\rtimes \Gamma$.  
Further the inner automorphism $c_g$ on $G'$ gives a group automorphism of $G'$ taking the first $\Gamma$-action to the second $\Gamma$-action, and so $G'$ with the second $\Gamma$ action is in $\calC$.
Thus
 $J'$ as a $\Gamma'$-subgroup of $G'$ belongs to $\calC$, so $\CF_{\Gamma}(J) \subseteq \CF_{\Gamma}(\calC)$ and we prove the lemma.
\end{proof}

Lemma~\ref{lem:CF_barC} establishes the analog of \cite[Lem.~6.11]{LiuWood2017} in the category of $\Gamma$-groups. By applying the proof of \cite[Lem.~6.13]{LiuWood2017}, we have the following corollary.

\begin{corollary}\label{cor:bound_A}
Let $\calC$ be a  set of  $|\Gamma|'$-$\Gamma$-groups, each of order at most $\ell$ for some integer $\ell$.
Then, for each pair $(M,A)\in \CF_{\Gamma}(\overline{\calC})$, the maximal divisor of $|A/\Inn(M)|$ that is relatively prime to $|\Gamma|$ is strictly less than $\ell$. Furthermore, $\CF_{\Gamma}(\overline{\calC})$ is a finite set.
\end{corollary}

\subsection{Countably additivity of $\mu_u$}
In this section, we will use Corollary~\ref{cor:bound_A} to control the size of $\mu_{u,n}$ and deduce countable additivity as a result.
Let $\calC_{\ell}$ be the set of (isomorphism classes of) all $\Gamma$-groups of order $\leq \ell$.
The multiplicities appearing in the formula for $\mu_{u,n}$ are the key part to be bounded, and those will be controlled in terms of $\Gamma$-chief factor pairs by the result below.  The key aspect of the bound is that the bound for multiplicities involving $H$ only depends on the quotient of $H$ at a smaller level.

\begin{lemma}\label{lem:bound_by_surj}
	Let $H$ be a $\Gamma$-group of level $\calC_{\ell}$ 
	and let $\widetilde{H}:=H^{\calC_{\ell-1}}$. 
	If $M$ is a direct product of isomorphic abelian simple groups, then 
	\begin{equation}\label{eq:abelian}
		\#\left\{G\in \calA_H \,\Bigg|\, 
		\begin{aligned}
			&G\simeq M \text{ as groups, and}\\
			&m(\calC_\ell, n, H, G) \neq 0 \text{ for some }n
		\end{aligned}\right\} \leq \sum_{(M,A)\in \CF_{\Gamma}(\overline{\calC_\ell})} |\Sur(\widetilde{H}\rtimes \Gamma, A)|.
	\end{equation}
	If $M$ is a direct product of isomorphic non-abelian simple groups, then
	\begin{eqnarray}
		&& \#\left\{\text{isom. classes of $H$-extensions } (E,\pi) \,\Bigg|\, \begin{aligned}
			&\ker \pi \simeq M \text{ is irred. $E$-group} \label{eq:non-abelian}\\
			&E \text{ is level }\calC_\ell
		\end{aligned} \right\}\\
		&\leq& \sum_{(M,A)\in \CF_{\Gamma}(\overline{\calC_\ell})} |\Sur(\widetilde{H}\rtimes \Gamma, A/\Inn(M))|. \nonumber
	\end{eqnarray}
\end{lemma}

\begin{proof}
	Let $G\in \calA_H$ and $\phi: G\simeq M$ an isomorphism of groups such that $m(\calC_\ell, n, H, G)\neq 0$ for some $n$. Let 
	\[
		1\to R \to F \rtimes \Gamma \to H\rtimes \Gamma \to 1
	\]
	be the short exact sequence associated to $\calC_\ell$, $n$ and $H$. Then $G$ appears as an irreducible factor of $R$ and $(G, \rho_G(F\rtimes \Gamma)) \in \CF_{\Gamma}(\overline{\calC_\ell})$. After applying $\phi: G\simeq M$, the quotient $\rho_G(F\rtimes \Gamma)$ of $H\rtimes \Gamma$ acts on $M$, and hence $(M, \rho_G(F\rtimes \Gamma))\in \CF_{\Gamma}(\overline{\calC_{\ell}})$. We let $\rho$ denote the quotient map from $H\rtimes \Gamma$ to $\rho_G(F\rtimes \Gamma)$, so we obtain the following map between sets
	\begin{equation}\label{eq:map_abelian}
		\left\{ G\in \calA_H, \phi \,\Bigg|\, \begin{aligned}
			&\phi: G\simeq M \text{ and}\\
			& m(\calC_\ell, n, H, G)\neq 0 \text{ for some }n
		\end{aligned} \right\}
		\to \{(M,A)\in \CF(\overline{\calC_\ell}), \rho \in \Sur (H\rtimes \Gamma, A)\}.
	\end{equation}
If $(G_1, \phi_1)$ and $(G_2, \phi_2)$ have the same image on the right-hand side, then $\phi_2^{-1}\circ \phi_1$ gives an isomorphism of the $H\rtimes \Gamma$-groups $G_1$ and $G_2$ taking $\phi_1$ to $\phi_2$.
	  So it follows that the map \eqref{eq:map_abelian} is injective. 
		
	To establish \eqref{eq:non-abelian}, we will similarly construct a map 
	\begin{eqnarray}
		&& \left\{ \text{isom. classes of $H$-extensions }(E,\pi) \,\Bigg|\, \begin{aligned}
			&\ker \pi \simeq M\text{ is irred. $E$-group}\\
			& E \text{ is level }\calC_\ell
		\end{aligned} \right\} \label{eq:map_non-abelian}\\
		&\to& \{(M, A) \in \CF_{\Gamma}(\overline{\calC_\ell}), \phi \in \Sur(H\rtimes \Gamma, A/\Inn(M))\}. \nonumber
	\end{eqnarray}
	Let $(E, \pi)$ be an $H$-extension described on the left hand side in \eqref{eq:map_non-abelian}. Then we define the image of $(E,\pi)$ to be $(M,A), \phi$, where $M=\ker \pi$, and $A=\rho_{\ker \pi} (E\rtimes \Gamma)$, and $\phi\in \Sur(H\rtimes \Gamma, A/\Inn(M))$ is a surjection on quotient groups induced by $\rho_{\ker\pi}\in \Sur(E\rtimes \Gamma, A)$ (we have that $\rho_{\ker\pi}$ maps $\ker\pi$ isomorphically to $\Inn(M)$ since $M$ is a direct product of isomorphic non-abelian simple groups). If $(E_1, \pi_1)$ and $(E_2, \pi_2)$ have the same image $(M,A)$, $\phi$ under the map \eqref{eq:map_non-abelian}, then they are isomorphic as $H$-extensions since $E_1 \rtimes \Gamma$ and $E_2 \rtimes \Gamma$ are both the fiber product of $\phi\colon  H\rtimes \Gamma \to A/\Inn(M)$ and $A\to A/\Inn(M)$.
	(This can be checked because they inject into the fiber product and have the same order.)
	
	Finally, by Corollary~\ref{cor:bound_A}, in both abelian case and non-abelian case, $\Sur(H\rtimes \Gamma, A/\Inn(M))$ is in one-to-one correspondence with $\Sur(\widetilde{H} \rtimes \Gamma, A/\Inn(M))$, and the lemma follows.
\end{proof}

Given a finite $\Gamma$-group $H$ of level $\calC_\ell$, we let $P_{u,n}(U_{\calC_{\ell}, H})$ denote the product in \eqref{eq:probability_u,n}, i.e.,
\[
	P_{u,n}(U_{\calC_{\ell}, H}) = \prod_{G\in \calA_H} \prod_{k=0}^{m(\calC_{\ell}, n, H, G)-1} (1-\frac{h_{H\rtimes \Gamma}(G)^k}{|Y(G)|^{n+u}}) \prod_{G\in \calN} (1-|Y(G)|^{-n-u})^{m(\calC_{\ell}, n, H, G)}.
\]
We have now give the key consequence of the work above, which is that the $P_{u,n}(U_{C_\ell,H})$ can be controlled in terms of  $H^{\calC_{\ell-1}}$.
\begin{corollary}\label{cor:lower_bound_of_product}
	Suppose $\ell>1$, $n\geq 1$ and $u>-n$ are integers and $\widetilde{H}$ is a finite $\Gamma$-group of level $\calC_{\ell-1}$. Then there exists a nonzero constant $c(u, \ell, \widetilde{H})$ depending on $u$, $\ell$ and $\widetilde{H}$ such that, for each finite $\Gamma$-group $H$ of level $\calC_\ell$ with $H^{\calC_{\ell-1}}\simeq \widetilde{H}$ as $\Gamma$-groups, the probability $P_{u,n}(U_{\calC_{\ell},H})$ is either zero or not less than $c(u, \ell, \widetilde{H})$.
\end{corollary}

Corollary~\ref{cor:lower_bound_of_product} can be proven by the same argument as in the proof of \cite[Lem.~9.5]{LiuWood2017}.  For the abelian factors in Theorem~\ref{thm:probability_u,n}, Remark~\ref{rmk:5.2} shows that it suffices to give an upper bound on the number of non-trivial factors depending only on $u$, $\ell$ and $\widetilde{H}$, which is given by Lemma~\ref{lem:bound_by_surj}.  For the non-abelian factors,  
we need an upper bound on the sum of  the multiplicities each weighted by $|Y(G)|^{-n}$.  Using the expression of Corollary~\ref{cor:multiplicities} for the multiplicities, and combining with Lemma~\ref{lem:count_surj}, Corollary~\ref{C:FHoms}, and Lemma~\ref{lem:bound_by_surj}, we obtain an upper bound depending only on $u$, $\ell$ and $\widetilde{H}$.

The following is a technical lemma that allows us to use our control of $P_{u,n}(U_{\calC_\ell,H})$ in terms of $H^{\calC_{\ell-1}}$ to obtain desired analytic results involving interchanging infinite sums and limits.

\begin{lemma}\label{lem:swaplim&sum}
	Suppose $f_n(H, \ell)$ is a nonnegative-valued function defined for  positive integers $n$, $\ell$ and every finite $\Gamma$-group $H$ of level $\calC_\ell$, such that 
	\begin{enumerate}
		\item \label{item:swap-2} There are functions $g_n(H,\ell)$ satisfying
		\[
			f_n(H, \ell)=g_n(H,\ell) P_{u,n}(U_{\calC_{\ell}, H}), \quad \forall n, \ell, H,
		\]
	
		\item \label{item:swap-3} For every $H$, the limit function $g(H,\ell)$ of $g_n(H,\ell)$ as $n\to\infty$ exists, and $g_n(H,\ell) \leq g(H, \ell)$ for any $n$.

		\item \label{item:swap-4} For any $n$, $\ell$ and any finite $\Gamma$-group $\widetilde{H}$ of level $\calC_{\ell-1}$, the following equality holds.
		\[
			\sum_{\substack{H :\text{ finite $\Gamma$-group of level $\calC_\ell$} \\ \text{s.t. $\widetilde{H} \simeq H^{\calC_{\ell-1}}$ as $\Gamma$-groups}}} f_n(H, \ell) = f_n(\widetilde{H}, \ell-1).
		\]
	\end{enumerate}
	Then for any integer $\ell\geq 1$, we have that $\lim_{n\to \infty} f_n(H, \ell)$ exists for any $\ell$ and any finite $\Gamma$-group $H$ of level $\calC_\ell$, and
	\[
		\sum_{\text{$H$: finite $\Gamma$-group of level $\calC_\ell$}} \lim_{n\to \infty}  f_n(H, \ell) = \lim_{n\to \infty} f_n(\text{trivial group},1).
	\]	
\end{lemma}

\begin{proof}
	The existence of the limit $\lim_{n\to \infty} f_n(H, \ell)$  follows from
	\eqref{item:swap-2}, \eqref{item:swap-3}, and the proof of Theorem~\ref{thm:probability_u,n}, which shows $\lim_{n\ra\infty} P_{u,n}(U_{\calC_{\ell}, H})$ exists.
	We will prove the lemma by induction on $\ell$. Since the trivial group (with trivial $\Gamma$-action) is the only finite $\Gamma$-group of level 1, the statement in the lemma obvious holds for $\ell=1$. Now we assume 
	\[
		\sum_{\text{$\widetilde{H}$: finite $\Gamma$-group of level $\calC_{\ell-1}$}} \lim_{n\to\infty} f_n(\widetilde{H}, \ell-1) = \lim_{n\to\infty} f_n(\text{trivial group}, 1).
	\]
	Then using the condition \eqref{item:swap-4}, the inductive step is equivalent to the following identity for each $\widetilde{H}$ of level $\ell-1$.
	\begin{equation}\label{eq:inductive-swap}
		\sum_{\substack{\text{$H$: finite $\Gamma$-group of level $\calC_\ell$} \\ \text{s.t. $\widetilde{H}\simeq H^{\calC_{\ell-1}}$ as $\Gamma$-groups}}} \lim_{n\to\infty} f_n(H, \ell) = \lim_{n\to \infty} \sum_{\substack{\text{$H$: finite $\Gamma$-group of level $\calC_\ell$} \\ \text{s.t. $\widetilde{H}\simeq H^{\calC_{\ell-1}}$ as $\Gamma$-groups}}} f_n(H, \ell).
	\end{equation}
	We fix $\ell$ and $\widetilde{H}$. By  condition \eqref{item:swap-3}, for every $H$, there are positive numbers $i(H)$ such that $g_n(H,\ell) \geq \frac{1}{2} g(H, \ell)$ for $n\geq i(H)$. So by Corollary~\ref{cor:lower_bound_of_product}, when $n\geq i(H)$ we have either $f_n(H,\ell)=0$ or 
	\[
		f_n(H, \ell) \geq \frac{1}{2} c(u, \ell, \widetilde{H}) g(H, \ell).
	\]
	We define $\Lambda$ to be the set of all isomorphism classes of finite $\Gamma$-groups $H$ of level $\ell$ such that $H^{\calC_{\ell-1}}\simeq \widetilde{H}$ and $f_n(H, \ell)>0$ for some $n$. So for a $\Gamma$-group $H$ of level $\calC_\ell$ 
	  whose pro-$\calC_{\ell-1}$ completion is $\widetilde{H}$, the function $f_n(H, \ell)$ is dominated by $g(H, \ell)$ when $H\in \Lambda$ and 0 otherwise. The sum of dominated functions
	\begin{eqnarray*}
		\sum_{H\in \Lambda} g(H, \ell) &=& \lim_{n\to \infty} \sum_{H \in \Lambda \text{ s.t. } i(H) \leq n} g(H, \ell)\\
		&\leq & \lim_{n\to \infty} \sum_{H \in \Lambda \text{ s.t. } i(H) \leq n} \frac{2}{c(u, \ell, \widetilde{H})} f_n(H, \ell)\\
		&\leq & \frac{2}{c(u, \ell, \widetilde{H})}\lim_{n\to \infty} \sum_{\substack{\text{$H$: finite $\Gamma$-group of level $\calC_\ell$} \\ \text{s.t. $\widetilde{H}\simeq H^{\calC_{\ell-1}}$ as $\Gamma$-groups}}} f_n(H, \ell)\\
		&= & \frac{2}{c(u, \ell, \widetilde{H})} \lim_{n\to \infty} f_n(\widetilde{H}, \ell-1).
	\end{eqnarray*}
	is finite. So by Lebesgue's Dominated Convergence Theorem, we obtain \eqref{eq:inductive-swap}
\end{proof}

Finally, we are able to deduce countable additivity of $\mu_u$.
\begin{theorem}\label{thm:countably_additivity}
	Let $u\geq 0$ be an integer. Then $\mu_u$ is countably additive on the algebra $\calA$ generated by the basic open sets $U_{\calC, H}$ for $H$ a finite $\Gamma$-group and $\calC$ a finite set of finite $|\Gamma|'$-$\Gamma$-groups.
\end{theorem}

\begin{proof}
	By the same argument as in the proof of \cite[Cor.~9.7 and Thm.~9.1]{LiuWood2017}, it suffices to prove
	\[
		\sum_{\text{$H$: finite $\Gamma$-group of level $\ell$}} \mu_u(U_{\calC_{\ell}, H})=1.
	\]
	We will apply Lemma~\ref{lem:swaplim&sum} by setting 
	\[
		f_n(H, \ell)=\mu_{u,n}(U_{\calC_\ell, H}) \quad  \text{and} \quad g_n(H)=\frac{|\Sur_\Gamma(\calF_n, H)|}{|\Aut_{\Gamma}(H)| |Y(H)|^{n+u}}.
	\]
	The limit of $g_n$ as $n\to \infty$ is $g(H)=|\Aut_{\Gamma}(H)|^{-1} |Y(H)|^{-u}$ by Corollary~\ref{C:FHoms}. One can easily check that conditions \eqref{item:swap-2}, \eqref{item:swap-3}, and \eqref{item:swap-4} in Lemma~\ref{lem:swaplim&sum} hold, and hence we have
	\[
		\sum_{\text{$H$: finite $\Gamma$-group of level $\ell$}} \mu_u(U_{\calC_\ell, H}) = \mu_{u}(U_{\calC_1, \text{trivial group}}) =1.
	\]
\end{proof}

\subsection{Proof of Theorem~\ref{T:MainGroup}}
Applying Theorem~\ref{thm:countably_additivity}, it follows
 from Carath\'{e}odory's extension
theorem that $\mu_u$ can be uniquely extended to a measure on the Borel sets of $\mathcal{P}$. 
We can then let $\mu_\Gamma=\mu_1$.
Any open set in our topology is a disjoint union of basic opens. So, for any open $U$, by \eqref{E:defmualgebra}
(with Theorem~\ref{thm:probability_u,n} to prove the limit exists)
 and Fatou's lemma, we have
$\mu_\Gamma(U) \leq \lim_{n\ra\infty} \mu_{1,n}(U)$.  This proves
Theorem~\ref{T:MainGroup}, with a measure $\mu_\Gamma$ such that
\begin{equation}\label{E:measonbasic}
\mu_\Gamma(U_{\mathcal{C},H})=
\frac{1}{|\Aut_{\Gamma}(H)| [H:H^\Gamma]} \prod_{G\in \calA_H} \prod_{i=1}^{\infty}(1-\lambda(\calC, H, G) \frac{h_{H\rtimes \Gamma}(G)^{-i}}{[G:G^\Gamma]}) \prod_{G\in \calN} e^{- \lambda(\calC, H,G)/[G:G^\Gamma]},
\end{equation}
where recall that $\mathcal{A}_H$ and $\mathcal{N}$ are defined just before Section~\ref{SS:41},  $\lambda(\mathcal{C},H,G)$ is defined in \eqref{E:lam1}
and \eqref{E:lam2}, and  $h_G(H):=|\Hom_G(H,H)|.$

\subsection{The Borel probability measure $\mu_u$}
In this section, we let $\calC$ be an arbitrary (not necessarily finite) set of finite  $|\Gamma|'$-$\Gamma$-groups and give the value of $\mu_u$ on the specific type of Borel sets
\[
	V_{\calC, H}:=\{X\in \mathcal{P} \,\mid\, X^{\calC}\simeq H \text{ as $\Gamma$-groups}\}
\]
for a finite $\Gamma$-group $H$ of level $\calC$. The set $V_{\calC, H}$ is not
necessarily a basic open set, but is the intersection of a sequence of basic open sets. 

\begin{definition}\label{def:infiniteC}
	Let $\calC$ be a set of finite $|\Gamma|'$-$\Gamma$-groups, $n$ a positive integer, and $H$ a finite $\Gamma$-group of level $\calC$. Let $\calD_1 \subset \calD_2 \subset \cdots$ be finite sets of finite $\Gamma$-groups such that $\cup_{i\geq 1} \calD_i=\calC$. For any $G\in \calA_H \cup \calN$, we define $m(\calC, n, H, G)=\lim_{i\to \infty}m(\calD_i, n, H, G)$ and  $\lambda(\calC, H, G) =\lim_{i\to \infty} \lambda(\calD_i, H, G)$. One can see that these definitions do not depend on the choice of the increasing sequence $\calD_i$.
\end{definition}

\begin{theorem}\label{thm:11.3}
	Let $\calC$ be an arbitrary set of finite $|\Gamma|'$-$\Gamma$-groups.
	 Let $H$ be a finite admissible $\Gamma$-group of level $\calC$. Let $u$ be an integer. Then 
	\[
		\mu_u(V_{\calC, H})=\frac{1}{|\Aut_{\Gamma}(H)| |Y(H)|^u} \prod_{G\in \calA_H} \prod_{i=1}^{\infty} (1-\lambda(\calC, H, G) \frac{h_{H\rtimes \Gamma}(G)^{-i}}{|Y(G)|^u}) \prod_{G\in\calN} e^{-|Y(G)|^{-u} \lambda(\calC, H,G)}.
	\]
\end{theorem}

Theorem~\ref{thm:11.3} can be proven with the same argument as in the proof of \cite[Lemma~11.3]{LiuWood2017}.  Briefly,
$\mu_u(V_{\calC, H})$ is a limit in $i$ of $\mu_u(U_{\calD_i, H})$, each of which are a limit in $n$ by definition.  We use that each factor in $\mu_u(U_{\calD_i, H})$ is non-increasing in $i$ and 
\cite[Lemma~8.11]{LiuWood2017} to obtain Theorem~\ref{thm:11.3}.

\section{Moments}\label{S:moments}

In this section, we find the moments of our distributions $\mu_{u}$ on random $\Gamma$-groups.
It is these averages over the distribution that will we prove in Part II actually match the corresponding averages of $\Gal(K^\#/K)$ for function fields as $q\ra\infty$.  Given our definitions, it is simple to find the limit in $n$ of the moments of $\mu_{u,n},$ but it takes some work to prove that these agree with the moments of $\mu_u$.

\begin{lemma}\label{lem: lim-moments}
	Let  $A$ be a finite admissible $|\Gamma|'$-$\Gamma$-group. Then for any integer $u$ we have
	\[
		\lim_{n\to \infty} \E (\Sur_\Gamma (X_{u,n}, A)) = \frac{1}{[A:A^{\Gamma}]^{u}}.
	\]
\end{lemma}

\begin{proof}
We have
	\begin{eqnarray*}
		\E(\Sur_\Gamma (X_{u,n}, A))
		&=& \sum_{\varphi \in \Sur_{\Gamma}(\calF_n, A)} \Prob(Y(\Rr) \in (\ker \varphi)^d ) \\
		&=& \frac{|\Sur_{\Gamma}( \calF_n, A)|}{|Y(A)|^{n+u}} \\
		&\to& \frac{1}{|Y(A)|^{u}}, \quad \text{as }n \to \infty.
	\end{eqnarray*}
	The the second equality and the limit above follow respectively by \eqref{eq:Prob-Y} and \eqref{eq:lim-Y} in the proof Theorem~\ref{thm:probability_u,n}.  Then Lemma~\ref{L:rightcoset} implies the lemma.
\end{proof}

For the following we will again use our multiplicity bounds coming from chief factor pairs as captured by Lemma~\ref{lem:swaplim&sum}.  This will give us the analytic power to interchange the limit with the implicit infinite sum in the expected value. Using a  delicate choice of function in Lemma~\ref{lem:swaplim&sum}, we are able to obtain the moments of $\mu_u$. 

\begin{theorem}\label{thm:moments}
Let $X_u$ be a random  $\Gamma$-group with distribution $\mu_u$.
	For any integer $u$ and any finite admissible $\Gamma$-group  $A$ we have
	\[
		\E(\Sur_{\Gamma}(X_u, A))=\frac{1}{[A:A^{\Gamma}]^u}.
	\]
\end{theorem}

\begin{proof}
	We consider the following function defined for any positive integer $\ell$ and any finite $\Gamma$-group $H$ of level $\calC_\ell$, 
	\[
		f_n(H, \ell) = \E( |\Sur_{\Gamma}(X_{u,n}, A)| \times \mathbbm{1}_{X_{u,n}^{\calC_\ell} \simeq H}),
	\]
	where $\mathbbm{1}_{X_{u,n}^{\calC_\ell} \simeq H}$ is the indicator function of
 $X_{u,n}^{\calC_{\ell}} \simeq H$ as $\Gamma$-groups. 
 We let $\pi_{\calF_n}\colon  \calF_n \to (\calF_n)^{\calC_\ell}$ and $\pi_A: A \to A^{\calC_\ell}$ be the pro-$\calC_\ell$ completion maps. Each $\phi\in \Sur_{\Gamma}(\calF_n, A)$ induces a map $\overline{\phi} \in \Sur_{\Gamma}((\calF_n)^{\calC_\ell}, A^{\calC_\ell})$.
 By the definition of random group $X_{u,n}$, we have 
	\begin{eqnarray}
		&& \E(|\Sur_{\Gamma}(X_{u,n}, A)| \times \mathbbm{1}_{X_{u,n}^{\calC_\ell}\simeq H})\label{eq:long-1}\\
		&=& \sum_{\phi \in \Sur_{\Gamma}(\calF_n, A)} \Prob( Y(\Rr) \subseteq (\ker \phi)^d \text{ and } (\calF_n)^{\calC_\ell}/[Y(\overline{\Rr})]_{(\calF_n)^{\calC_\ell} \rtimes \Gamma} \simeq H)\nonumber 
	\end{eqnarray}
	where $\Rr$ is the set of $n+u$ independent Haar random elements of $\calF_n$, and $\overline{\Rr}$ is their set of images in $(\calF_n)^{\calC_\ell}$.
 Given $\phi\in \Sur_{\Gamma}(\calF_n, A)$, we have
	\[
		\Prob(Y(\Rr) \subseteq (\ker \phi)^d\mid 
Y(\overline{\Rr})		\subseteq (\ker \overline{\phi})^d)= \frac{|Y(A^{\calC_\ell})|^{n+u}}{|Y(A)|^{n+u}},
	\]
	by Lemma~\ref{lem:XYproperties}.  Further, we can use Lemma~\ref{lem:XYproperties} and basic properties of probability to see that for any set $y$ of $n+u$ elements of $(\ker \overline{\phi})^d$, we have that
		\[
		\Prob(Y(\Rr) \subseteq (\ker \phi)^d\mid 
Y(\overline{\Rr})=y)= \frac{|Y(A^{\calC_\ell})|^{n+u}}{|Y(A)|^{n+u}}.
	\]
	We can then consider all $y$ that give quotient $H$ to see  that 
\begin{align*}
&\Prob\left(
\begin{aligned}
&Y(\Rr) \subseteq (\ker \phi)^d \text{ and } \\
&(\calF_n)^{\calC_\ell}/[Y(\overline{\Rr})]_{(\calF_n)^{\calC_\ell} \rtimes \Gamma} \simeq H
\end{aligned}
 \,\Bigg|\,
  \begin{aligned}
&Y(\overline{\Rr}) \subseteq (\ker \overline{\phi})^d \text{ and } \\
&(\calF_n)^{\calC_\ell}/[Y(\overline{\Rr})]_{(\calF_n)^{\calC_\ell} \rtimes \Gamma} \simeq H
\end{aligned}
\right)\\
=&\Prob(Y(\Rr) \subseteq (\ker \phi)^d\mid 
Y(\overline{\Rr}) \subseteq (\ker \overline{\phi})^d \text{ and } 
(\calF_n)^{\calC_\ell}/[Y(\overline{\Rr})]_{(\calF_n)^{\calC_\ell} \rtimes \Gamma} \simeq H
)\\=&\frac{|Y(A^{\calC_\ell})|^{n+u}}{|Y(A)|^{n+u}}.
\end{align*}
Thus
	\eqref{eq:long-1} is equal to 
	\begin{eqnarray}
		&& \frac{|Y(A^{\calC_\ell})|^{n+u}}{|Y(A)|^{n+u}} \sum_{\phi \in \Sur_{\Gamma}(\calF_n, A)} \Prob\left( \begin{aligned}
		& Y(\overline{\Rr}) \subseteq (\ker \overline{\phi})^d \text{ and } \\
		&(\calF_n)^{\calC_\ell}/[Y(\overline{\Rr})]_{(\calF_n)^{\calC_\ell} \rtimes \Gamma} \simeq H
		\end{aligned}\right) \nonumber \\
		&=& \frac{|Y(A^{\calC_\ell})|^{n+u}}{|Y(A)|^{n+u}}  \sum_{\phi \in \Sur_{\Gamma}(\calF_n, A)} \frac{\#\left\{(\tau, \pi) \,\Bigg|\, \begin{aligned} 
		&\tau\in \Sur_{\Gamma}((\calF_n)^{\calC_\ell}, H)\\
		&\pi \in \Sur_{\Gamma}(H, A^{\calC_\ell}) \\
		& \text{ and } \pi \circ \tau = \overline{\phi}
		\end{aligned}\right\}}{|\Aut_{\Gamma}(H)| |Y(H)|^{n+u}} P_{u,n}(U_{\calC_\ell, H}). \label{eq:long-2}
	\end{eqnarray}
	On the other hand, let $\overline\phi \in \Sur_{\Gamma}((\calF_n)^{\calC_\ell}, A^{\calC_\ell})$. Then the composition map $\rho:= \overline{\phi} \circ \pi_{\calF_n}$ is a  surjection from $\calF_n\to A^{\calC_\ell}$. The number of $\phi\in \Sur_{\Gamma}(\calF_n, A)$ such that $\phi$ induces  $\overline{\phi}$ is the size of $\Sur_{\Gamma}(\rho, \pi_A)$ which is computed in Lemma~\ref{lem:count_surj}. Therefore,
	\begin{eqnarray*}
		&&\sum_{\phi \in \Sur_{\Gamma}(\calF_n, A)} \#\left\{(\tau, \pi) \,\Bigg|\, \begin{aligned} 
		&\tau\in \Sur_{\Gamma}((\calF_n)^{\calC_\ell}, H)\\
		&\pi \in \Sur_{\Gamma}(H, A^{\calC_\ell}) \\
		& \text{ and } \pi \circ \tau = \overline{\phi}
		\end{aligned}\right\}
		= |\Sur_{\Gamma}(\rho, \pi_A)| |\Sur_{\Gamma}((\calF_n)^{\calC_\ell}, H)| |\Sur_{\Gamma} (H, A^{\calC_\ell})|.
		\end{eqnarray*}
From Corollary~\ref{C:FHoms}, we have that 
		\[
			\frac{|\Sur_{\Gamma}((\calF_n)^{\calC_\ell}, H)|}{|Y(H)|^n}\leq1,\quad 
			\lim_{n\to \infty} \frac{|\Sur_{\Gamma}((\calF_n)^{\calC_\ell}, H)|}{|Y(H)|^n}=1. 
		\]
			Also using Corollary~\ref{C:FHoms} and Lemma~\ref{lem:count_surj}, we have
		\[
		\text{and} \quad \frac{|\Sur_{\Gamma}(\rho, \pi_A)|}{|Y(A)|^n |Y(A^{\calC_\ell})|^{-n}}\leq 1, \quad \lim_{n \to \infty} \frac{|\Sur_{\Gamma}(\rho, \pi_A)|}{|Y(A)|^n |Y(A^{\calC_\ell})|^{-n}}=1.
		\]
		Then by \eqref{eq:long-2}, we obtain that $f_n(H, \ell)=g_n(H, \ell) P_{u,n}(U_{\calC_\ell,H})$ with
		\[
		g_n(H, \ell)=\frac{|Y(A^{\calC_\ell})|^{n+u} |\Sur_{\Gamma}(\rho, \pi_A)| |\Sur_{\Gamma}((\calF_n)^{\calC_\ell}, H)| |\Sur_{\Gamma}(H, A^{\calC_\ell})|}{|Y(A)|^{n+u} |Y(H)|^{n+u} |\Aut_{\Gamma}(H)|}
		\]
		\[
			\text{and}\quad g(H,\ell):=\lim_{n\to \infty} g_n(H, \ell) =\frac{|Y(A^{\calC_\ell})|^u |\Sur_{\Gamma}(H, A^{\calC_\ell})|}{|Y(A)|^u |Y(H)|^u |\Aut_{\Gamma}(H)|},
		\]
		for which conditions \eqref{item:swap-2} and \eqref{item:swap-3} in Lemma~\ref{lem:swaplim&sum} clearly hold. Also, the condition Lemma~\ref{lem:swaplim&sum} \eqref{item:swap-4} holds by the definition of functions $f_n$. Thus, by Lemma~\ref{lem:swaplim&sum} and Lemma~\ref{lem: lim-moments} we have
		\begin{equation}\label{eq:long-3}
			\sum_{\substack{\text{$H$: finite $\Gamma$-group}\\ \text{of level $\ell$}}} \lim_{n\to\infty} f_n(H, \ell) =\lim_{n\to \infty} f_n(\text{trivial group}, 1) = \lim_{n\to \infty} \E(|\Sur_{\Gamma}(X_{u,n}, A)|) = \frac{1}{[A:A^{\Gamma}]^u}.
		\end{equation}
		When $\ell$ is sufficiently large such that $A$ is of level $\calC_\ell$, 
		\[
			\lim_{n\to \infty} f_n(H, \ell) = \lim_{n\to \infty} |\Sur_{\Gamma}(H, A)| \Prob((X_{u,n})^{\calC_\ell} \simeq H) = |\Sur_{\Gamma}(H, A)| \Prob((X_u)^{\calC_\ell} \simeq H),
		\]
		and hence \eqref{eq:long-3} gives the desired result in the theorem.
\end{proof}

\section{Comparison to other distributions}\label{S:other}

\subsection{Cohen--Lenstra--Martinet distribution}\label{S:CM}
First, let $\calC$ be any finite set of finite abelian $|\Gamma|'$-$\Gamma$-groups.  Then in a fundamental short exact sequence $1\ra R \ra F\ra H \ra 1$ for an admissible $\Gamma$-group $H$ of level $\calC$, we see that the $H$-action on $R$ is trivial, and thus $m(\calC, n, H,G)=0$ unless $G\in \calA_H$ with trivial $H$-action.  Moreover, by Corollary~\ref{cor:abelianizationCalF}, we see that $(\calF_n^{\textrm{ab}})^{\Gamma}=1$, and thus $F^{\Gamma}=R^{\Gamma}=H^{\Gamma}=1$ (since $F$ is a quotient of $\calF_n^{\textrm{ab}}$). 
On the other hand, assume $H$ is an abelian $\Gamma$-group of level $\calC$ such that $H^{\Gamma}=1$. The elements $h^{-1} \gamma(h)$ for all $h \in H$ and $\gamma \in \Gamma$ generate a $\Gamma$-subgroup $S$ of $H$ such that $\Gamma$ acts trivially on $H/S$.  Since $H^{\Gamma}=1$, Lemma~\ref{lem:XYproperties}\eqref{item:1} gives $(H/S)^\Gamma=1$, and hence $H=S$ and $H$ is admissible.  
So we see that a finite abelian $|\Gamma|'$-$\Gamma$-group $H$ is admissible if and only if $H^{\Gamma}=1$.

Now consider the set $\calC$ of all abelian $|\Gamma|'$-$\Gamma$-groups.  
Let $N$ be a closed  $\Gamma$-subgroup of $(\calF_n)^{\calC}$  such that the quotient $H:= (\calF_n)^{\calC}/N$ is finite. We then claim $N\simeq (\calF_n)^{\calC}$.  The claim can be checked on the maximal pro-$p$ quotients $N_p$ and $(\calF_n)^{\calC}_p$, in which case it follows (e.g.~see \cite[Sect.~16, Thm.~34, Cor.~2]{Serre1977}) because 
$N_p$ is a projective $\Z_p[\Gamma]$-module (since $\Z_p[\Gamma]$ is hereditary)  and $N_p\otimes_{\Z_p} \Q_p \simeq (\calF_n)^{\calC}_p \otimes_{\Z_p} \Q_p$ (since $N_p$ is a submodule of $(\calF_n)_p^{\calC}\isom (\Z_p[\Gamma])^n$ of finite index and is torsion-free).

Recall that by Remark~\ref{rmk:5.2}, for each $G\in \calA_H$, we have $h_{\Gamma}(G)^{\mathfrak{m}_G}=|G|$, where $\mathfrak{m}_G$ is the maximal integer such that $G^{\mathfrak{m}_G}$ can be $(H\rtimes \Gamma)$-generated by coordinates of $Y(g)$ of one element $g\in G^{\mathfrak{m}_G}$. 
Then we can use $N\simeq (\calF_n)^{\calC}$ to  see that for any $G\in \calA_H$ with trivial $H$-action, we have $m(\calC, n, H, G)=n\mathfrak{m}_G$, and hence by \eqref{eq:lambda_m} we have $\lambda(\calC, H, G)=1$ for any $G\in \calA_H$ with trivial $H$-action and $G^\Gamma=1$.

By Theorem~\ref{thm:11.3}, we then have, for $H$ a finite abelian $|\Gamma'|$-$\Gamma$-group $H$ with $H^{\Gamma}=1,$
\[
	\mu_u(V_{\calC,H}) =\frac{1}{|\Aut_{\Gamma}(H)| |H|^u}
\prod_{p\nmid |\Gamma|}	\,\,\,
	 \prod_{
	\substack{G \text{ irred. non-trivial}\\ \text{rep. of $\Gamma$ over $\F_p$}}} \,\prod_{i=1}^{\infty} ( 1- \frac{h_{\Gamma}(G)^{-i}}{|G|^u}).
\]
We obtain a similar result if we restrict $\calC$ to all abelian groups whose orders are only divisible by primes in some fixed set of primes not dividing $|\Gamma|.$  When $\Gamma=\Z/2\Z$ and $u=0,1$, 
the measures above are the distributions that Cohen and Lenstra \cite{Cohen1984} conjectured for the  (odd parts of) class groups of imaginary and real, respectively, quadratic fields.  For any $\Gamma$, when $u=1$,
the measures above are the conjectured distributions for the (prime-to-$|\Gamma|$ parts of) class groups of totally real $\Gamma$-number fields, by conjectures of Cohen and Lenstra \cite{Cohen1984} for abelian $\Gamma$ and Cohen and Martinet \cite{Cohen1990} for general $\Gamma$. (See \cite[Proposition 6.5]{WW}
 for more explanation of the fact that these conjectures agree with the measures above.)

\subsection{Boston--Bush--Hajir distribution}

Let $p$ be an odd prime. Let $\Gamma=\Z/2\Z$  and denote the non-trivial element of $\Gamma$ by $\sigma$.  

\subsubsection{For pro-$p$ completions}\label{sect:pro-p}
 Recall that we denote the pro-$p$ completion of $\calF_n$ by $(\calF_n)_p$. 
Let $H$ be an admissible $p$-$\Gamma$-group and $\varpi\colon  (\calF_n)_p \to H$ a $\Gamma$-equivariant surjection. Then we have the following short exact sequence
\[
	1 \to N \to (\calF_n)_p \overset{\varpi}{\to} H \to 1.
\]
We first consider this as a short exact sequence of pro-$p$ groups (forgetting the $\Gamma$-actions), and define $N^*:=N^p [N, (\calF_n)_p]$, and 
 $R':=N/N^*$, and $F':=(\calF_n)_p/N^*$. 
The group $N^*$ is a $\Gamma$-subgroup of $N$.
 It follows that $R'$ and $F'$ naturally obtain $\Gamma$-group structures from $N$ and $(\calF_n)_p$ respectively, and then we have a short exact sequence of $p$-$\Gamma$-groups
\begin{equation}\label{eq:fundshort}
	1\to R' \to F' \to H \to 1.
\end{equation}
By definition, $F'$ is the largest quotient of $(\calF_n)_p$ such that $\ker(F'\to H)$ is an elementary abelian $p$-group with the trivial $F'$-action.

\begin{lemma}\label{L:newfses}
Let $\Gamma=\Z/2\Z$ and $p$ be an odd prime.
Let $H$ be an admissible $p$-$\Gamma$-group. 
Assume $\varpi\colon  (\calF_n)_p \to H$ is a $\Gamma$-equivariant surjection whose kernel can be generated, as a closed normal $\Gamma$-subgroup, by the coordinates of $Y(S)$, for some finite set $S\subset (\calF_n)_p$. If $\calC$ is a finite set of $p$-$\Gamma$-groups such that $H^{\calC}=H$ then the fundamental short exact sequence associated to $n$, $H$ and $\calC$ is given by the surjection $(F')^{\calC} \to H$ induced by \eqref{eq:fundshort}. In particular, when the finite set $\calC$ is sufficiently large, the fundamental short exact sequence is \eqref{eq:fundshort}.
\end{lemma}

\begin{proof}
	Let $1\to R \to F \to H \to 1$ denote the fundamental short exact sequence associated to $n$, $H$ and $\calC$. By the condition on the generation of $\ker \varpi$, we have that $\Gamma$ acts by $-1$ on any irreducible factor $G$ of $R$ and hence $G$ is an irreducible $F$-group.  Since $\calC$ contains only $p$-$\Gamma$-groups, $G$ has to be abelian, so is an irreducible $H$-group, which implies that $G=\Z/p\Z$ with the trivial $H$-action.  
Therefore by the construction of \eqref{eq:fundshort}, $F'\to H$ has to factor through $F$. On the other hand, $F\to H$ factors through the largest quotient of $F'$ of level $\calC$, so we see $F=(F')^\calC$. The last statement in the lemma follows immediately if $F'$ is of level $\calC$.
\end{proof} 

\begin{remark}\label{rem:fses-inf}
	Although the definition of short exact sequence given in Section~\ref{sect:defofmu} does not work for an infinite set $\calC$, under conditions in Lemma~\ref{L:newfses} we can define the short exact sequence for any set $\calC$ of $p$-$\Gamma$-groups to be the one given by $(F')^{\calC}\to H$. It is clear that the multiplicity of factors in $R=\ker((F')^{\calC} \to H)$ agrees with the one obtained via Definition~\ref{def:infiniteC}.
\end{remark}

When $n$ is equal to the generator rank $d(H)$ of $H$, then $F'$ is called the $p$-covering group of $H$ and $R'$ is called the $p$-multiplicator rank of $H$, and $\dim_{\F_p}(R')$ is the relation rank $r(H)$ of $H$. 
Then we can conclude the following.

\begin{theorem}\label{thm:measure_pro-p}
	Let $\Gamma=\Z/2\Z$ be generated by $\sigma$ and $p$ an odd prime.   Let $u$ be a fixed integer.
Let $\calC$ be the set of all $p$-$\Gamma$-groups.	
	Given a finite $p$-$\Gamma$-group $H$, 
	if $H$ is of the form $\calF_n^{\calC}/[Y(\Rr)]$ for some subset $\Rr\sub \calF_n^{\calC}$ of size $n+u$
	for some $n$, then
	\[
		\mu_u(V_{\calC, H})=\frac{1}{|\Aut_{\Gamma}(H)| |Y(H)|^u} \prod_{i=1+u+d(H)-r(H)}^{\infty} (1- p^{-i}).
	\]
	If $H$ is not of this form for any $n$, then $\mu_u(V_{\calC, H})=0$.
\end{theorem}

\begin{proof}
Theorem~\ref{thm:11.3} gives the formula for $\mu_u(V_{\calC, H})$, and by the proof of Lemma~\ref{L:newfses}, we see that $\mathcal{A}_H$ has only one element $G$ with $\lambda(\mathcal{C},H,G)\ne 0$,
 and that $|Y(G)|=p$ and $h_{H\rtimes \Gamma}(G)=p$.  From \cite[Lem.~12.5]{LiuWood2017}, we know that $m(\calC, n, H, G)$ is equal to $r(H)-d(H)+n$, and thus from Equation~\eqref{eq:lambda_m} we know that 
$\lambda(\calC,H,G)=p^{r(H)-d(H)}$. Combining we obtain the theorem.
\end{proof}

\subsubsection{For pro-$p$ class $c$ completions}\label{sect:pro-p-class-c}

For a $p$-group $P$, we let $(P_i(P))_{i\geq 0}$ denote the lower $p$-central series of $P$, and $Q_i(P):=P/P_i(P)$. For each $i$, the $P_i(P)$ is a characteristic subgroup of $P$, so if $P$ is a given $\Gamma$-group structure, then $Q_i(P)$ automatically obtains a $\Gamma$-group structure from $P$.

Given a positive integer $c$, we define $\calC$ to be the set of all finite $p$-$\Gamma$-groups of $p$-class at most $c$.  The set $\calC$ is closed under taking finite direct products, $\Gamma$-subgroups and $\Gamma$-quotients. Therefore, the free admissible pro-$\calC$ $\Gamma$-group $(\calF_n)^{\calC}$ is exactly $Q_c((\calF_n)_p)$. 

Let $H$ be a finite admissible $p$-$\Gamma$-group of $p$-class at most $c$. Let $\varpi^{\calC}\colon  (\calF_n)^{\calC} \to H$ be the surjection induced by a $\Gamma$-equivariant surjection $\varpi: (\calF_n)_p \to H$.
Recall we have defined $F'$ and $R'$ in \eqref{eq:fundshort} for $\varpi$.
 Then for each $i$, we have the following commutative diagram of surjections
\[\begin{tikzcd}
	& F' \arrow[two heads]{r} \arrow[two heads]{d} & H \arrow[two heads]{d} \\
	& Q_i(F') \arrow[two heads]{r} &Q_i(H).
\end{tikzcd}\]
The kernel of the upper row, which is $R'$, is an elementary abelian $p$-group, and thus so is the kernel of the lower row.

Let $cl_p(H)$ denote the $p$-class of $H$. If $cl_p(H)<c$, i.e.,~$Q_{c-1}(H)=H$, then $P_{c-1}(F')$ is contained in $\ker(F' \to H)$. Therefore $P_c(F')$, which is defined to be $[P_{c-1}(F'), F']P_{c-1}(F')^p$, is trivial since $\ker(F' \to H)$ (and hence $P_{c-1}(F')$) is an elementary abelian $p$-group on which $F'$ acts trivially by conjugation. It follows that $F'$ is of $p$-class at most $c$, and 
for $H$ as in Lemma~\ref{L:newfses},
we have by Remark~\ref{rem:fses-inf} that $1\to R' \to F' \to H \to 1$ is the fundamental short exact sequence defined by $\varpi$, 
and hence $m(\calC, n ,H, G)=r(H)-d(H)+n$, where $G$ is the irreducible factor defined in 
the proof of Lemma~\ref{L:newfses}.

If $cl_p(H)=c$, then $P_c((\calF_n)_p)\leq N$. The nucleus of $H$ is defined to be $P_c(F')$, which is the image of $P_c((\calF_n)_p)$ under the quotient map $(\calF_n)_p \to F'$, and hence  $P_c(F')=P_c((\calF_n)_p) N^* / N^*$ and is contained in $N/N^*=R'$.
The nuclear rank of $H$ is defined to be $\nu(H):=\dim_{\F_p} (P_c(F'))$. 
Let $1\to R_\calC \to F_{\calC} \to H \to 1$ denote the fundamental short exact sequence induced by $\varpi^{\calC}$, for which we have $F_{\calC} = F'/ P_c(F')$ and $R_{\calC}=R'/P_c(F')$ by Lemma~\ref{L:newfses} and Remark~\ref{rem:fses-inf}.
Then by applying Theorem~\ref{thm:measure_pro-p}, we have $m(\calC, n, H, G)=r(H)-\nu(H)-d(H)+n$. Therefore the theorem below immediately follows.

\begin{theorem}\label{thm:measure_class_c}
	Let $p$, $H$, and $G$ be the ones defined in Theorem~\ref{thm:measure_pro-p}. Let $\calC$ be the set of all $p$-$\Gamma$-groups of $p$-class at most $c$. If $H$ is  of the form $\calF_n^{\calC}/[Y(\Rr)]$ for some subset $\Rr\sub \calF_n^{\calC}$ of size $n+u$
	for some $n$	, then 
	\begin{enumerate}
		\item when $cl_p(H)<c$ then $m(\calC, n, H, G)=r(H)+n-d(H)$ and 
		\[
			\mu_u(V_{\calC, H})=\frac{1}{|\Aut_{\Gamma}(H)| |Y(H)|^u} \prod_{i=1+u+d(H)-r(H)}^{\infty} (1-p^{-i});
		\]
		\item when $cl_p(H)=c$ then $m(\calC, n ,H, G)=r(H)-\nu(H)+n-d(H)$ and
		\[
			\mu_u(V_{\calC, H})=\frac{1}{|\Aut_{\Gamma}(H)| |Y(H)|^u} \prod_{i=1+u+d(H)+\nu(H)-r(H)}^{\infty} (1-p^{-i}).
		\]
	\end{enumerate}
	If $H$ is not of this form for any $n$, then $\mu_u(V_{\calC, H})=0$.
\end{theorem}

So, when $\Gamma=\Z/2\Z$, we see explicitly our measure above on $p$-$\Gamma$-groups of
$p$-class at most $c$.  In the case $u=0$, our measure agrees with the measure of Boston, Bush, and Hajir
\cite[Lemma 4.8]{Boston2017} predicted to be the distribution of the $p$-class $c$ quotient of $p$-class field tower groups of imaginary quadratic fields.  In the case $u=1$, our measure agrees with the measure of Boston, Bush, and Hajir
\cite[Theorem 2.22]{Boston2018} predicted to be the distribution of the $p$-class $c$ quotient of $p$-class field tower groups of real quadratic fields.  (See \cite[Lemma 2.7]{Boston2017a} and \cite[Lemma 2.8]{Boston2018} for the fact that the support of their measure is contained in the support of our measure, and then it follows from the fact that the formulas match for each $H$ that the measures have exactly the same support.)
In the work of Boston, Bush, and Hajir, they start with the distribution on generator ranks predicted by Cohen--Lenstra, and then for a given generator rank take all relations in the Frattini subgroup.  This way, they avoid taking a limit in $n$, which is the most difficult aspect of our model.  Beyond the pro-nilpotent case we cannot see how to avoid taking a limit in $n$.

\begin{center}
{\bf Part II: Counting unramified extensions of function fields}\end{center}

\section{Notation for Part II}

\begin{notation}
 For a scheme $Z$ and a ring $R$, we write $Z_R$ for $Z\times_{\Spec \Z} \Spec R$. 

For a scheme $X$ with a morphism $X\ra \Spec \F_q$, we write $F_{X,q}\colon X\ra X$ for the Frobenius map that on open affine comes from the ring endomorphism $a\mapsto a^q$.  We write $\overline{X}$ for the base change $X\times_{\Spec \F_q}\Spec \overline{\F}_q$, and we write $\Frob_X$ for the Frobenius map
$F_{X,q} \times \id_{\Spec \overline{\F}_q} \colon\overline{X} \ra \overline{X}.$

When we consider functions $f,g$ of several variables $x_1,\dots,y_1,\dots$, we write
$$f(x_1,\dots,y_1,\dots)= g(x_1,\dots,y_1,\dots)+
O_{x_1,\dots}(h(x_1,\dots,y_1,\dots))$$ to mean there exists a function $C$ of the variables $x_1,\dots$, such that for every value of the $x_1,\dots,y_1,\dots$ we consider, we have
$$|f(x_1,\dots,y_1,\dots)- g(x_1,\dots,y_1,\dots)|\leq C(x_1,\dots) h(x_1,\dots,y_1,\dots).$$

\end{notation}

\section{Relation between objects counted}\label{S:relation}

In this section we explain how the objects that our moments count, i.e.,~unramified  $H$-extensions of $\Gamma$-extensions, correspond to certain $H\rtimes \Gamma$-extensions.  That there is some kind of correspondence is clear.  However, since we wish to obtain exact constant averages, we must precisely determine the factor introduced by this correspondence, and we do that in this section.

\begin{lemma}\label{L:firstbij}
Let $\Gamma$ be a finite group, $H$ a $\Gamma$-group, and $G$ a profinite group.
We have a bijection from
$$
\Sur(G,H\rtimes \Gamma)
$$
to
\begin{equation}\label{E:all4}
\left\{(\rho,N,s,\phi) \,\Bigg|\, 
	\begin{aligned} & \rho\in \Sur(G,\Gamma), N \textrm{ a closed normal subgroup of $G$ with }N\sub \ker \rho, \\
	& s\in \Hom(\Gamma, G/N) \textrm{ with }\rho \circ s=\operatorname{id}, \phi\in \operatorname{Isom}_\Gamma(\ker \rho/N, H) 
	\end{aligned}\right\}.
\end{equation}
(The action of $\Gamma$ on $\ker \rho/N$ is given by $s$ and then conjugation.
Since $N\sub \ker \rho$,  by a slight abuse of notation we use $\rho$ to also denote the map $G/N\ra \Gamma.$)
 In the bijection, an element  $\psi\in \Sur(G,H\rtimes \Gamma)$ is mapped to $(\rho,N,s,\phi)$, where
\begin{itemize}
\item $\rho$ is the projection of $\psi$ to $\Gamma$
\item $N=\ker\psi$
\item $s=\psi^{-1} \colon \Gamma\ra G/N$, and
\item $\phi=\psi|_{\ker \rho/N}$.
\end{itemize}
\end{lemma}

\begin{proof}
We construct an inverse as follows.  Let $(\rho,N,s,\phi)$ be given.  From $\rho$ and $N\leq \ker\rho$ we obtain an exact sequence
$$
1\ra \ker \rho/N \ra G/N \stackrel{\rho}{\ra} \Gamma \ra 1.
$$
The map $s$ defines a section of $\rho:G/N \ra \Gamma$, giving an isomorphism 
$G/N\isom  \ker \rho/N \rtimes \Gamma.$
As a $\Gamma$-isomorphism, the map $\phi$ extends to $\ker \rho/N \rtimes \Gamma \isom H\rtimes \Gamma$.  
Composing the projection $G\ra G/N$ with these two maps gives a surjection $\psi: G \ra H\rtimes \Gamma$.
It is straightforward to check that this indeed defines an inverse.
\end{proof}

The following will apply Lemma~\ref{L:firstbij} to exactly the extensions we want to count with the required parameters.  Recall our notation $E_\Gamma(D,Q)$ for the set of isomorphism classes of totally real $\Gamma$-extensions of $Q$ with $\rD K=D$.
Let $K^{\operatorname{un},\infty}$ be the maximal extension of $K$ that is unramified everywhere and split completely at all places of $K$ above $\infty$.

\begin{lemma}\label{L:gettoN}
Let $\Gamma$ be a finite group, and $H$ a finite $|\Gamma|'$-$\Gamma$-group.
Let $N(H,\Gamma,D,Q)$ be the number of surjections $\Gal(\overline{Q}/Q)\ra H\rtimes \Gamma$ such that the corresponding $H\rtimes \Gamma$-extension $K$ has $\rD K=D$, the associated $H$-extension $K/K^H$ is unramified everywhere and split completely at all places above $\infty$, and the associated $\Gamma$-extension $K^{H}/Q$ is split completely at $\infty$.
Then
$$
\sum_{K\in E_\Gamma(D,Q)} |\Sur_\Gamma(\Gal(K^{\operatorname{un},\infty}/K),H)|=\frac{1}{[H:H^{\Gamma}]}N(H,\Gamma,D,Q).
$$
\end{lemma}

\begin{proof}
We consider the bijection of Lemma~\ref{L:firstbij} with $G=\Gal(\overline{Q}/Q)$.
The elements of \eqref{E:all4} fiber over $\rho$.  For each $\rho$, we will now pick out certain distinguished elements of \eqref{E:all4} and relate them to the left-hand side of the lemma.
Given  $\rho \in \Sur(G, \Gamma)$, let $M$ be the normal subgroup of $\ker\rho$ such that $\ker\rho/M$ is the maximal pro-$|\Gamma|'$ quotient of $\ker \rho$, and by Definition~\ref{D:defG} we have a distinguished section $s_{\rho}: \Gamma \to G/M$ of $\rho$ (the subgroup $M$ is characteristic in $\ker\rho$, and hence is normal in $G$). Then the section $s_{\rho}$ defines a $\Gamma$-action on $\ker\rho/M$, and maps to a section $\overline{s_{\rho}}: \Gamma \to G/N$ for every normal subgroup $N$ of $G$ with $M\subset N \subset \ker \rho$. Each element $(\rho, N, \overline{s_{\rho}}, \phi)$ in \eqref{E:all4} with the fixed $\rho$ and $s_{\rho}$ defines a surjection $\varphi \in \Sur_{\Gamma}(\ker \rho/M ,H)$. On the other hand, given $\varphi \in \Sur_{\Gamma}(\ker\rho/M, H)$, we let $N:=\ker \varphi$ and $\phi: \ker\rho /N\to H$ be the isomorphism induced by $\varphi$, which gives an element $(\rho, N, \overline{s_{\rho}}, \phi)$ in \eqref{E:all4}. So, we have a bijection between the elements 
in \eqref{E:all4} with $s= \overline{s_{\rho}}$ and
\begin{equation}\label{E:renorm}
\{(\rho,\phi) \,|\, \rho\in \Sur(G,\Gamma),   \phi\in \operatorname{Sur}_{\Gamma,s_\rho} (\ker \rho/M, H)
\}.
\end{equation}

The set \eqref{E:renorm} is close to the left-hand side in the lemma,  but to obtain the left-hand side precisely, 
 we additionally must require that $\rho$ is split completely at $\infty$, and $\phi$ is unramified everywhere and split completely at all places above $\infty$.  These are conditions that can be detected from $\rho$ and $N$, and we can see in the bijection of Lemma~\ref{L:firstbij} that they correspond to the condition on $\psi\in \Sur(G,H\rtimes \Gamma)$ that the associated $\Gamma$-extension is split completely at $\infty$, and the associated $H$-extension is unramified everywhere and split completely everywhere above $\infty$.  Similarly, with this ramification restriction, we can see that the condition that $\rD K=D$ corresponds in the bijection to the condition that $\psi\in \Sur(G,H\rtimes \Gamma)$ corresponds to a field with $\rD=D$.  So
the subset of \eqref{E:renorm}
which satisfy our local and discriminant conditions  is the set counted by the left-hand side expression in the lemma.  

Now, we consider how the distinguished elements in \eqref{E:all4} relate to the other elements in \eqref{E:all4}.
Given $\rho$ and $N$ that occur in \eqref{E:all4} when $(|H|,|\Gamma|)=1$, 
the Schur--Zassenhaus theorem says that all sections 
$s\colon\Gamma \ra G/N$ of $\rho\colon G/N\ra\Gamma$ differ by conjugation of elements  
of $\ker \rho/N$.  Since the stabilizer of the action of $\ker \rho/N$ on these sections is $(\ker \rho/N)^\Gamma$, we have that the number of sections
is $[\ker \rho/N:(\ker \rho/N)^\Gamma]=[H:H^\Gamma].$  Further, each choice of section leads to an equal number of choices for $\phi$.
This observation combined with the bijection above and the bijection of Lemma~\ref{L:firstbij} shows that the left-hand side of the lemma is $[H:H^\Gamma]^{-1}$ times the size of  $N(H,\Gamma,D,Q)$.
\end{proof}

\section{Function field counting}\label{S:FFC}

In this section, we will show how the existence and certain properties of Hurwitz schemes (to be proven later) allow us to prove Theorem~\ref{T:MainFF}.

\begin{notation}\label{N:Gc}
Let $G$ be a finite group and $c$ a subset of $G\setminus \{1\}$ closed under conjugation by elements of $G$ and taking invertible powers (e.g.~if $g\in c$ and $\gcd(m,\operatorname{ord}(g))=1$, then $g^m\in c$), such that $c$ generates $G$. Let $n$ be a non-negative integer. Let $R:=\Z[|G|^{-1}]$. 
\end{notation}

\subsection{Properties of $\cHur_{ G,c}^n$} \label{S:cons}

 In Section~\ref{S:Construction}, we will define a scheme $\cHur_{ G,c}^n$ over $\Spec R$ for each non-negative integer $n$.
We will show the following three results about $\cHur_{ G,c}^n$.

\begin{lemma}\label{L:what}
Let $\Gamma$ be a finite group, and $H$ a finite $|\Gamma|'$-$\Gamma$-group.
Let $G:=H\rtimes \Gamma$ and $c$ be the set of elements of $G$ that have the same order as their image in $\Gamma$. For all prime powers $q$ relatively prime to $|G|$ and $n\geq 0$, we have 
$$|\cHur_{ G,c}^n(\F_q)|=N(H,\Gamma,q^n,\F_q(t)),$$
where $ N(H,\Gamma,q^n,\F_q(t))$ is defined in Lemma~\ref{L:gettoN}.
\end{lemma}

\begin{lemma}\label{L:C}
Let $G,c,n$ be as in Notation~\ref{N:Gc}.  Let $Y$ be a connected component of $\cHur_{G,c}^n$.
For all integers $i\geq 0$, prime powers $q$ relatively prime to $|G|$, and primes $\ell \nmid q$ with $\ell>n$, we have
that $Y_{\overline{\F}_q}$ is smooth of dimension $n$, and
$$
\dim_{\Q_{\ell}} H^i_{\textup{c,\'et}} (Y_{\overline{\F}_q},\Q_\ell )= \dim_{\Q} H^{2n-i} (Y_{\C},\Q ),
$$
where on the left we have compactly supported \'{e}tale cohomology, and on the right we have singular cohomology.
\end{lemma}
Lemmas~\ref{L:what} and \ref{L:C} will be proven in Section~\ref{S:Construction}.

We define $\pi_{G,c}(q,n)$ to be the number of $\Frob_{(\cHur_{ G,c}^n)_{\F_q}}$ fixed 
components of $(\cHur_{ G,c}^n)_{\overline{\F}_q}$.
We define $d_{G,c}(q)$ to be the number of orbits of $q$th powering on the set of non-trivial $G$-conjugacy classes in $c$.
If we drop $c$ from either subscript, $c$ is taken to be $G\setminus 1$.

\begin{theorem}\label{T:comp}
Let $\Gamma$ be a finite group and $H$ be a finite admissible $\Gamma$-group. 
 Let $G:=H\rtimes \Gamma$ and $c$ be the set of elements of $G$ that have the same order as their image in $\Gamma$.  
Then, there is a real $C_G>0$, and for every $a\in (\Z/|G|^2\Z)^\times$, there is a positive integer $M_a$ and a non-empty set $E_a$ of residues mod $M_a$ such that the following holds:

For every
  positive integer $n$,  and
every prime power $q$ such that $q\equiv a \pmod{|G|^2}$ and $\gcd(q(q-1),|H|)=1$,
we have $$
 \pi_{G,c}(q,n) =\pi_{\Gamma}(q,n)+O_G( n^{d_{\Gamma}(q)-2});
$$
moreover,
\begin{itemize}
\item if $n\pmod{M_a} \not\in E_a$, then  $\pi_{G,c}(q,n)=\pi_{{\Gamma}}(q,n)=0$, and 
\item if $n\pmod{M_a} \in E_a$
and $n$ is sufficiently large (given $G$),  then
$\pi_{\Gamma}(q,n) \geq C_G n^{d_{\Gamma}(q)-1}.$
\end{itemize}
\end{theorem}

\begin{remark}
The complicated dependency we see in Theorem~\ref{T:comp} on the residue classes of $q$ and $n$ is simply the truth.  Even for quadratic extensions, we know that the discriminant is always an even power of $q$ (i.e.,~$n$ is even).  Some geometric components of the Hurwitz scheme are not defined over $\F_q,$ but then necessarily are defined over $\F_{q^r}$ for some $r$.
\end{remark}

\subsection{Proof of Theorem~\ref{T:MainFF}}

Now we show how Theorem~\ref{T:MainFF} follows from Lemmas~\ref{L:what} and \ref{L:C} and Theorem~\ref{T:comp}.

\begin{proof}[Proof of Theorem~\ref{T:MainFF}]
We first consider $q$ in a single residue class mod $|G|^2$.
From Lemma~\ref{L:what} we have
$$
\sum_{n=0}^b N(H,\Gamma,q^n,\F_q(t)) =\sum_{n=0}^{b}  \#\cHur_{ G,c}^n(\F_q).
$$
By taking $H$ the trivial group in Lemma~\ref{L:what}, we have
$$
\sum_{n=0}^b |E_\Gamma(q^n,\F_q(t))|=\sum_{n=0}^{b}  \#\cHur_{ \Gamma,\Gamma\setminus \{1\}}^n(\F_q).
$$
By Theorem~\ref{T:comp},  for $n$ such that $n\pmod{M_a}\not\in E_a$,
 the Hurwitz schemes above have
no Frobenius fixed components and thus no $\F_q$-points.
Thus we can round down $b$ to the nearest integer 
 whose residue mod $M_a$ is in $E_a$, without changing the sums above.

For a scheme $Z$ over $\F_q$, by the Grothendieck--Lefschetz trace formula (see, e.g.~\cite[Theorem 29.8]{MilneECNotes}) we have
$$
\#Z({\F}_q) =\sum_{j\geq 0} (-1)^j\Tr(\Frob| { H^j_{\textrm{c,\'et}}(Z_{\overline{\F}_q},\Q_\ell)})
$$
(for any prime $\ell\nmid q$).
Also, if $Z$ is $n$-dimensional, then $\Tr(\Frob| { H^{2n}_{\textrm{c,\'et}}(Z_{\overline{\F}_q},\Q_\ell)})$
is $q^n$ times the number of Frobenius fixed $n$-dimensional components of $Z_{\overline{\F}_q}$.
Let 
$$B(b)=\max_{0\leq j \leq 2b,1\leq n\leq b} \max(\dim H^{j}( (\cHur_{ G,c}^n)_{\C},\Q),\dim H^{j}( (\cHur_{\Gamma,\Gamma\setminus\{1\}}^n)_{\C},\Q)).
$$
We have that the absolute value of any eigenvalue of $\Frob$ on $H^j_{\textrm{c,\'et}}((\cHur_{ G,c}^n)_{\overline{\F}_q},\Q_\ell)$ is at most
$q^{j/2}$ \cite[Th\'eor\`eme 1]{Deligne1980}, and similarly for $\cHur_{\Gamma,\Gamma\setminus\{1\}}^n$.

Thus from the trace formula above and the comparison in Lemma~\ref{L:C}, 
\begin{align*}
|\#\cHur_{ G,c}^b(\F_q) - \pi_{G,c}(q,b) q^b|  \leq B(b) \sum_{0\leq j< 2b} q^{j/2}.
\end{align*}
Similarly,
\begin{align*}
|\#\cHur_{ \Gamma,\Gamma\setminus\{1\}}^b(\F_q) - \pi_\Gamma(q,b) q^b|  \leq B(b) \sum_{0\leq j< 2b} q^{j/2}.
\end{align*}
For $1\leq n<b$, we have
\begin{align*}
\#\cHur_{ G,c}^n(\F_q), \#\cHur_{ \Gamma,\Gamma\setminus\{1\}}^n(\F_q)   \leq B(b) \sum_{0\leq j \leq 2n} q^{j/2}.
\end{align*}
So we have
\begin{align*}
\frac{\sum_{0\leq n \leq b} N(E,\Gamma,q^n,\F_q(t))}{\sum_{0\leq n \leq b} |E_\Gamma(q^n,\F_q(t))|}=\frac{\pi_{G,c}(q,b) q^b+E_1(b,q,G,c)q^{b-1/2}}{\pi_\Gamma(q,b) q^b+E_2(b,q,\Gamma)q^{b-1/2}},
\end{align*}
where $E_1(b,q,G,c), E_2(b,q,\Gamma)=O_{b,G}(1)$. 
By the first statement of Theorem~\ref{T:comp}, we conclude that 
\begin{align*}
\frac{\sum_{0\leq n \leq b} N(E,\Gamma,q^n,\F_q(t))}{\sum_{0\leq n \leq b} |E_\Gamma(q^n,\F_q(t))|}
&=1+\frac{E_3(b,q,G,c,\Gamma)b^{d_\Gamma(q)-2}
+E_1(b,q,G,c)q^{-1/2}}{\pi_\Gamma(q,b)+E_2(b,q,\Gamma)q^{-1/2}},
\end{align*}
where $E_3(b,q,G,c,\Gamma)=O_G(1).$ 
The second statement of Theorem~\ref{T:comp} says that $\pi_\Gamma(q,b) \geq C_G b^{d_\Gamma(q)-1} $ (recall our rounding down $b$ above).
Thus we conclude that
$$
\lim_{\substack{q\ra\infty\\q\equiv a\pmod{|G|^2}}}\frac{\sum_{0\leq n \leq b} N(E,\Gamma,q^n,\F_q(t))}{\sum_{0\leq n \leq b} |E_\Gamma(q^n,\F_q(t))|} =
1+O_G(b^{-1}).
$$
Since the above is true for each choice of $a$, it is also true without the restriction to $q$ in a single residue class.
With Lemma~\ref{L:gettoN} and Theorem~\ref{thm:moments}, this proves Theorem~\ref{T:MainFF}.
\end{proof}

\section{Construction of $\cHur_{ G,c}^n$}\label{S:Construction}
  In this section, we give a construction of the scheme $\cHur_{ G,c}^n$ for integers $n\geq 0$, and prove its basic properties including Lemmas~\ref{L:what} and \ref{L:C}. We refer to \cite{Bertin2011} for general facts about Hurwitz spaces.

\subsection{Algebraic Hurwitz spaces}

Let $S$ be a scheme. A \defi{curve} over $S$ is a smooth and proper map $X \to S$ whose geometric fibers are connected and 1-dimensional. A \defi{cover} of a curve $X$ over $S$ is a finite, flat, and surjective morphism $f\colon Y \to X$ of $S$-schemes, where $Y$ is also a curve over $S$. A cover is \defi{tame} if the ramification index at every point is prime to the characteristic of that point. We define $\underline{\Aut} f$ to be the automorphism group scheme of $S$ and $\Aut f := (\underline{\Aut} f)(S)$ to be the group of automorphisms of $f$. A cover $f\colon Y \to X$ is \defi{Galois} if $f$ is separable and if $\Aut f$ acts transitively on fibers of geometric points of $X$. Associated to a tame Galois cover of curves $Y \to X$ is its \defi{branch locus} $D \subset X$, which has the properties that $D \to S$ is \'etale \cite[Proposition 3.1.1]{Bertin2011}, the restriction of $f$ to $X - D$ is \'etale, and $X-D$ is maximal with respect to this property. If there exists a constant $n$ such that the degree of each geometric fiber of $D \to S$ is equal to $n$ (which is automatic if $S$ is connected), then we say that $f$ has $n$ \defi{branch points}. 

Let $G$ be a finite group. We define the \defi{Hurwitz stack} $\Hur_{G}^n$ to be the fibered category of tame Galois covers of $\P^1$ with $n$ branch points, together with a choice of identification of $G$ with the automorphism group of the cover. More precisely, an object of $\Hur_{G}^n$ in the fiber $\Hur_{G}^n(S)$ over a scheme $S$ is a pair $(f,\iota)$, where 
$
f \colon X \to \P^1_S
$
is a tame Galois cover and
$
\iota\colon G \to \Aut f
$
is an isomorphism. Morphisms are commutative diagrams which commute with $\iota$; i.e., given a morphism $S' \to S$ of schemes and objects $(f,\iota) \in \Hur_{G}^n(S)$ and $(f',\iota') \in \Hur_{G}^n(S')$, a morphism is a commutative diagram
\[
\xymatrix{
 X' \ar[d]_{f'} \ar[r]& X\ar[d]_{f} \\
\P^1_{S'} \ar[r]& \P^1_{S}
}
\]
equivariant for the $G$-actions on $X'$ and $X$.
Let $\Conf^n(\P^1)$ be the quotient $\left(\left(\P^1_{\Z}\right)^n-\boldsymbol\Delta\right) / S_n$, where $\boldsymbol \Delta$ is the big diagonal (i.e., the subscheme of tuples such that some pair of entries agree). 

\begin{theorem}[{\cite[Theorem 4 and Theorem 1.4.3]{Wewers1998}}]\label{T:Wewers}
  The fibered category $\Hur_{G}^n$ is a separated Deligne--Mumford stack, smooth and finite type over $\Spec \Z$; its coarse space exists and is a scheme. 
The map
\[
\psi^n_{G}\colon \Hur_{G}^n \to \Conf^n(\P^1),
\]
which sends $(f,\iota)$ to the branch locus of $f$, is \'etale 
and is proper  over $\Conf^n(\P^1)_{\Z[|G|^{-1}]}$. 
\end{theorem}

We define $\Hur_{G,1}^n$ to be the fibered category of tame Galois $G$-covers with $n$ branch points together with a choice of point over $\infty$. In other words, an object of $\Hur_{G,1}^n$ in the fiber $\Hur_{G,1}^n(S)$ over a scheme $S$ is a pair $(f,\iota; P)$, where $(f,\iota) \in \Hur_{G}^n(S)$, and $P \in X(S)$ is a point lying over $\infty$, i.e., a map $P\colon S \to X$ making the diagram 
\[
\label{eqn:lying-over}
\xymatrix{
 S \ar[d]_{\id} \ar[r]^P& X\ar[d]_{f} \\
S\ar[r]^{\infty}\ar[dr]_{\id}& \P^1_S\ar[d]\\
& S
}
\]
commute. 

\begin{lemma}
  The stack $\Hur_{G,1}^n$ is algebraic and separated, and it admits a coarse space which is a scheme. 
\end{lemma}

\begin{proof}
First, the forgetful map $\phi \colon \Hur_{G,1}^n \to \Hur_{G}^n$ is representable, quasi-finite, and proper. Indeed, given an $S$-point $S \to \Hur_{G}^n$ corresponding to the data $(f,\iota)$, the fiber product 
\[
S \times_{\Hur_{G}^n} \Hur_{G,1}^n
\]
is isomorphic to $f^{-1}(\infty) := S \times_{\P^1_S} X$. In particular, since the fibered category $\Hur_{G,1}^n$ admits a representable map to an algebraic stack, it is also algebraic \cite[05UM]{stacks-project}. Since $S \times_{\P^1_S} X$ is finite if $S$ is finite, $\phi$ is also quasi-finite. Since properness is local on the target and the map $S \times_{\P^1_S} X \to S$ is proper, $\phi$ is proper.

Next, since $ \Hur_{G,1}^n \to \Hur_{G}^n$ is proper and  $\Hur_{G}^n$ is separated,  $\Hur_{G,1}^n$ is separated. By \cite[Corollary 1.3 (1)]{KeelM:Quotients}, $\Hur_{G,1}^n$  thus admits a coarse (algebraic) space. 

Finally, since $\phi$ is proper and quasi-finite, the same is true of the induced map $\phi'$ of coarse spaces. Indeed,  $\phi'$ is quasi-finite since, by \cite[Theorem 11.1.2]{olsson:stacks-book}, it is a bijection on geometric points. For properness, by \cite[Theorem 6.12]{rydh:quotients} the map $q$ from $\Hur_{G}^n$ to its coarse space is proper, so the map $q \circ \phi$ is proper, and properness of $\phi'$ now also follows from \cite[Theorem 6.12]{rydh:quotients} (noting that $\Hur_{G,1}^n$ is locally noetherian since $\Hur_{G}^n$ is locally noetherian and $\phi$ is proper). By \cite[082J]{stacks-project}, $\phi'$ is representable; since the coarse space of $\Hur_{G}^n$ is a scheme, we conclude that the coarse space of $\Hur_{G,1}^n$ is a scheme.
\end{proof}

\begin{remark}
  The stack $\Hur_{G,1}^n$ contains a dense open subscheme (which we will see next), but is in general not a scheme (e.g., it is not a scheme if $G \cong \Z/2\Z$ and the fiber of $f$ over $\infty$ contains a double point). 
\end{remark}

We define $\Hur_{G,*}^n$ to be the open substack of $\Hur_{G,1}^n$ such that the marked point is unramified. In other words, $\Hur_{G,*}^n$ is the subcategory of $(f,\iota; P) \in \Hur_{G,1}^n$ such that the image $P(S) \subset f^{-1}(\infty) := S \times_{\P^1_S} X$ of $S$ under $P$ is contained in the open subscheme $(f^{-1}(\infty))^{\sm}$. The branch locus map $\psi^n_{G,*}\colon \Hur_{G,*}^n \to \Conf^n(\P^1)$  has image in $\Conf^n(\A^1)$.

 \begin{proposition}\label{prop:etale}
     The stack $\Hur_{G,*}^n$ is representable by a scheme, and the branch locus map $\psi^n_{G,*}\colon \Hur_{G,*}^n \to \Conf^n(\A^1)$ is finite \'etale over $\Conf^n(\A^1)_{\Z[|G|^{-1}]}$.
   \end{proposition}
   
\begin{proof}
The square
\[
\xymatrix{
\Hur_{G,*}^n \ar[r]\ar[d]_{\psi^n_{G,*}} &  \Hur_{G,1}^n \ar[d]_{\psi^n_{G}} \\
\Conf^n(\A^1) \ar[r] & \Conf^n(\P^1)
}
\]
is cartesian (by definition of $\Hur_{G,*}^n$ and $\Conf^n(\A^1)$); 
in particular, the map  $\Hur_{G,*}^n \to \Hur_{G,1}^n$ is representable and open, and $\psi^n_{G,*}$ is proper and quasi-finite (since the same was true of $\Hur_{G,1}^n \to \Conf^n(\P^1)$).

Let $(f,\iota; P) \in \Hur_{G,*}^n$. Since $P$ is unramified and $f$ is Galois, $G$ acts simply transitively on $f^{-1}(\infty)$, and any element of $G$ fixing $P$ thus fixes every point of $f^{-1}(\infty)$. The tuple $(f,\iota; P)$ thus has no automorphisms; by \cite[04SZ]{stacks-project} $\Hur_{G,*}^n$ is thus an algebraic space. Since formation of coarse spaces commutes with Zariski opens \cite[Theorem 11.1.2]{olsson:stacks-book}, representability follows since the coarse space of $\Hur_{G,1}^n$ is a scheme.  Since $\Hur_{G,*}^n$ is a scheme, $\psi^n_{G,*}$ is representable, and since it is also proper and quasi-finite, it is finite.

Finally, to check that $\psi^n_{G,*}$ is \'etale, it is enough to check that  $\Hur_{G,*}^n \to \Hur_{G}^n$ is \'etale (since $\psi^n_{G}$ is \'etale by Theorem \ref{T:Wewers}); indeed, given an $S$-point $S \to \Hur_{G}^n$ corresponding to the data $(f,\iota)$, the fiber product 
\[
S \times_{\Hur_{G}^n} \Hur_{G,*}^n
\]
is isomorphic to $f^{-1}(\infty)^{\sm}$, where  $f^{-1}(\infty) := S \times_{\P^1_S} X$, which a scheme and is \'etale over $S$.
\end{proof}

\subsection{Analytic Hurwitz spaces}

 All stacks in this section are Deligne--Mumford. See \cite[Section 2]{hall:relative-gaga-principle} for background on analytic Deligne--Mumford stacks.
 For a stack  $X$ over $\mathbb{C}$, denote by $X_{\an}$ its complex analytification; we will reserve superscripts for analytic stacks that were not constructed as the analytification of an algebraic stack (so for example, $(\Hur_{G}^n)_{\an}$ is the complex analytification of the algebraic stack $\Hur_{G}^n$, while $\Hur_{G}^{n, \text{an}}$ is the ``analytic'' Hurwitz stack).

\begin{proposition}[{\cite[Lemma 7.1]{hall:relative-gaga-principle}}]
\label{P:stacky-relative-gaga}
  Let $f \colon X \to X'$ be a morphism of analytic stacks. Suppose that for every local Artinian $\mathbb{C}$-scheme $S$, the functor 
\[
X(S_{\an}) \to X'(S_{\an})
\]
is an equivalence of categories. Then $f$ is an equivalence. 
\end{proposition}

\begin{remark}
\label{R:stacky-gaga}
By stacky GAGA \cite[Theorem C]{hall:relative-gaga-principle}, for any proper $\mathbb{C}$-scheme $S$, the analytification functor 
\[
  X(S) \to   X_{\an}(S_{\an})
\]
is an equivalence.
\end{remark}

We define $\Hur_{G}^{n, \text{an}}$ to be the analytic Hurwitz stack: an object of $\Hur_{G}^{n, \text{an}}$ in the fiber $\Hur_{G}^{n, \text{an}} (S)$ over an analytic space $S$ is a pair 
$
(f\colon X \to \P^1_S, \iota\colon G \cong \Aut f)
$
as above, but where $X$ is now an analytic space and $f$ is an analytic morphism. We define $\Hur_{G,*}^{n,\text{an}}$ similarly. These are indeed stacks: \'etale descent follows from descent in the complex topology (since analytic \'etale maps are locally isomorphisms in the complex topology). That these are analytic stacks follows incidentally from the following theorem.

\begin{theorem}
  The natural maps $\left(\Hur_{G}^n\right)_{\an} \to \Hur_{G}^{n, \text{an}}$ and $\left(\Hur_{G,*}^n\right)_{\an} \to \Hur_{G,*}^{n,\text{an}}$ are equivalences of analytic stacks.
\end{theorem}

\begin{proof}
  By by Proposition \ref{P:stacky-relative-gaga} and Remark \ref{R:stacky-gaga}, it suffices to check that for every local Artinian $\mathbb{C}$-scheme $S$, the functors
\[
\Hur_{G}^n(S) \to \Hur_{G}^{n, \text{an}}(S_{\an})
\]
and 
\[
\Hur_{G,*}^n(S) \to \Hur_{G,*}^{n,\text{an}}(S_{\an})
\]
are equivalences of categories.

First, for a smooth proper analytic curve $C^{\an} \to S_{\an}$, there is a smooth proper curve $C \to S$ and an isomorphism $C_{\an} \to C^{\an}$ over $S_{\an}$. Indeed, $S \to \Spec \mathbb{C}$ is proper (since $S$ is local and Artinian), so a smooth proper analytic curve $C^{\an} \to S_{\an}$ is also proper over $\left(\Spec \C\right)_{\an}$; by \cite[Theorem 7.3]{artin:algebraization-of-formal-moduil-II}, there exists an algebraic space $C$ and an isomorphism $C_{\an} \cong C^{\an}$. By stacky GAGA, the composition $C_{\an} \to C^{\an} \to S_{\an}$ algebrizes to a map $f\colon C \to S$, which is smooth and proper (since the same is true of $f_{\an}$). Finally, a proper 1-dimensional algebraic space over $\mathbb{C}$ is a curve \cite[Theorem V.4.9 and Corollary III.3.6]{Knutson:algebraicSpaces}.   

To finish, we invoke stacky GAGA  (Remark \ref {R:stacky-gaga}): noting again that $S \to \Spec \mathbb{C}$ is proper, an analytic section of $X_{\an} \to S_{\an}$ algebrizes, and an analytic map $Y_{\an} \to X_{\an}$ of curves over $S_{\an}$ is proper and algebrizes (and in particular automorphisms algebrize).
\end{proof}

\subsection{Topological Hurwitz spaces}\label{S:top}

Next, we show that $\Hur_{G,*}^{n, \text{an}}$ agrees with its purely ``topological'' analogue; we model our arguments on \cite[Section 3]{Romagny2006}, with minor simplifications since $\Hur_{G,*}^{n, \text{an}}$ is a scheme.
\\

 We define a topological space $\Hur_{G,*}^{n, \text{top}}$ as follows. As a set, $\Hur_{G,*}^{n, \text{top}} := \Hur_{G,*}^{n, \text{an}}(\mathbb{C})$. We define a basis of neighborhoods around a given $(f,\iota; P) \in \Hur_{G,*}^{n, \text{top}}$ with branch locus $D =\{t_1, \ldots, t_n\} \subset \A^1(\C)$ as follows. The choice of $\iota$ and $P$ determines a homomorphism $\rho_f\colon \pi_1(\P^1(\C) \setminus D,\infty) \to G$ (rather than conjugacy class of homomorphisms). Let $\underline{C} = \{C_1, \ldots, C_n\}$ be disjoint disk-like neighborhoods of the points $t_1,\ldots,t_n$ and let $\Conf^n(\A^1) (\underline{C})$ denote the subset of $\Conf^n(\A^1) (\C)$ consisting of $D' = \{t_1',\ldots,t_n'\}$ such that $t_i' \in C_i$. We define $\mathcal{H}((f,\iota; P), \underline{C})$ to be the subset of $\Hur_{G,*}^{n, \text{top}}$ of $(f', \iota'; P')$ whose associated $D'$ is an element of $\Conf^n(\A^1) (\underline{C})$ and such that $\rho_{f'} = \rho_f$ under the natural identification
\[
\pi_1(\P^1(\C)\setminus D, \infty) \cong \pi_1(\P^1(\C)\setminus\cup C_i, \infty) \cong \pi_1(\P^1(\C)\setminus D', \infty).
\]
We take $\mathcal{H}((f,\iota; P), \underline{C})$ to be a basis of open neighborhoods of the point $(f,\iota; P)$.

The natural map $\Hur_{G,*}^{n, \text{top}} \to \Conf^n(\A^1) (\mathbb{C})$ is a covering space morphism; indeed, for $\Conf^n(\A^1) (\underline{C}) \subset \Conf^n(\A^1) (\mathbb{C})$ a neighborhood of $D$,
\[
\Psi^{-1}(\Conf^n(\A^1) (\underline{C})) = \coprod_{(f,\iota; P)} \mathcal{H}((f, \iota; P), \underline{C})
\] 
where the coproduct ranges over all isomorphism classes of $(f,\iota; P)$ with branch locus $D$. In particular, this gives $\Hur_{G,*}^{n, \text{top}}$ the structure of an analytic space.

There is a natural map $\Phi \colon \Hur_{G,*}^{n, \text{an}} \to \Hur_{G,*}^{n, \text{top}}$ of analytic spaces, which we describe via Yoneda's lemma. Let $S$ be an analytic space and let $(f,\iota;P) \in \Hur_{G,*}^{n, \text{an}}(S)$. Let $\phi_f\colon S \to \Hur_{G,*}^{n, \text{top}}$ be the map which sends $s \in S$ to the fiber of $(f,\iota;P)$ over $s$; the local isomorphism $\Hur_{G,*}^{n, \text{top}} \to \Conf^n(\A^1)$ endows $\phi_f$ with an analytic structure.
The map $\Phi$ is thus a bijective map of smooth analytic spaces, and by \cite[p.~19]{griffiths-and-Harris} is an isomorphism of analytic spaces.

\subsection{Definition of $\cHur_{ G,c}^n$}

Let $G,c,R$ be as in Notation~\ref{N:Gc} throughout this section.
Given an algebraically closed field $k$, for a $k$-point of $\Hur_{G,*}^n$, we can consider the $n$ branch points in $\P^1_k$, and at each branch point, the inertia groups are a conjugacy class of cyclic subgroups.  Thus from the $k$-point of $\Hur_{G,*}^n$, we obtain a multiset of conjugacy classes of non-trivial cyclic subgroups of $G$.  By \cite[Thm 3.1]{Components} and \cite[Remark 5.3]{Components}, we see that this multiset is constant on components of $\Hur_{G,*}^n$.   We define $\cHur_{ G,c}^n$ to be 
the union of components of $\left(\Hur_{G,*}^n\right)_R$ for which all of the inertia groups are generated by elements of $c$. 

For the rest of the paper, we work over $R = \Z[|G|^{-1}]$. As we add the subscript $c$ to our notation, we are removing the $*$; nonetheless, we really do mean to continue working with a choice of unramified point over $\infty$ throughout the rest of the paper.

Now we prove Lemma~\ref{L:what}, to show that the ${\F}_q$-points of $\cHur_{ G,c}^n$ are our objects of interest.
\begin{proof}[Proof of Lemma~\ref{L:what}]
This lemma is about translating between the subfields of $\overline{\F_q(t)}$ counted by $N(H,\Gamma,q^n,\F_q(t))$ and the abstract extensions of
$\F_q(t)$ that correspond to $\F_q$-points of $\cHur_{ G,c}^n$.

Recall that $N(H,\Gamma,q^n,\F_q(t))$ counts certain surjections $\Gal(\overline{\F_q(t)}/\F_q(t))\ra H\rtimes \Gamma$.
A surjection 
$\Gal(\overline{\F_q(t)}/\F_q(t))\ra H\rtimes \Gamma$
is equivalent to a subfield $L$ of $\overline{\F_q(t)}$ such that $\F_q(t)\sub L$, and an isomorphism $\Gal(L/\F_q(t))\isom H \rtimes \Gamma$.   Using the definition of $c$, we can see that the associated $H$-extension is unramified everywhere if and only if all inertia groups are generated by elements of $c$.  Also for any extension $K/\F_q(t)$, we have

$\rD K=q^n$, where $n$ is degree of the ramification divisor (i.e.,~branch locus of the cover or curves).

By definition of $\cHur_{ G,c}^n$ and the correspondence between smooth projective curves over a field and their function fields, we have that ${\F}_q$-points of $\cHur_{ G,c}^n$ correspond to
 isomorphism classes of $(M,\pi,m)$, where $M/\F_q(t)$ is a Galois  extension   (split completely over $\infty$ and with 
 $ \rD M =q^n$),
 $\pi\colon \Gal(M/\F_q(t)) \ra H \rtimes \Gamma$  is an isomorphism 
 (taking all inertia groups to subgroups generated by elements of $c$), and $m$ is a choice of place of $M$ over $\infty$ (for the marked point).
   We choose an embedding $\overline{\F_q(t)} \ra\overline{\F_q(t)}_\infty$.  Then we see that each isomorphism class of $(M,\pi,m)$ has a distinguished element  where
$M$ is a subfield of $\overline{\F_q(t)}$, Galois over $\F_q(t)$, and $m$ chooses the place whose valuation pulls back from the valuation on
$\overline{\F_q(t)}_\infty$. (In an isomorphism class, we have a unique subfield of $\overline{\F_q(t)}$  isomorphic to a given Galois extension, and we have a number of choices for an isomorphism from a field in the isomorphism class to that subfield, given by elements of the Galois group, and exactly one of those choices takes the chosen place over $\infty$ to the one that pulls back from $\overline{\F_q(t)}_\infty$, and this determines what $\pi$ must be.)  
These Galois subfields $L$ of $\overline{\F_q(t)}$  with choices of isomorphism $\Gal(L/\F_q(t))\isom H \rtimes \Gamma$ (i.e., the distinguished elements above)
are exactly what is counted by $N(H,\Gamma,n,\F_q(t))$, as described in the first paragraph of the proof.
\end{proof}

Now we prove Lemma~\ref{L:C}.
\begin{proof}[Proof of Lemma~\ref{L:C}]
\
By Proposition~\ref{prop:etale}, we have that $Y$ is smooth of relative dimension $n$ over $\Spec R$.  Thus we have $\dim H^i_{\textrm{c,\'et}}(Y_{\overline{\F}_q},\Q_\ell)=
 \dim H^{2n-i}_{\textrm{\'et}}(Y_{\overline{\F}_q},\Q_\ell)$ by Poincar\'{e} Duality.
Next, we will relate $\dim H^{j}_{\textrm{\'et}}(Y_{\overline{\F}_q},\Q_\ell)$ to 
 $\dim H^{j}( Y_\C,\Q_\ell)$.  
 The result \cite[Proposition 7.7]{Ellenberg2016} gives an isomorphism between \'{e}tale cohomology in characteristic $0$ and in positive characteristic in the case of a finite cover of a complement of a reduced normal crossing divisor in a smooth proper scheme.  Though \cite[Proposition 7.7]{Ellenberg2016} is only stated	 for \'{e}tale cohomology with coefficients in $\Z/\ell\Z$, the argument goes through identically for coefficients in $\Z/\ell^k\Z$, and then we can take the inverse limit and tensor with $\Q_{\ell}$ to obtain the result of \cite[Proposition 7.7]{Ellenberg2016} with $\Z/\ell\Z$ coefficients replaced by $\Q_\ell$ coefficients. 
 We apply this strengthened comparison to $ Y \times_{\Conf^n(\A^1)} \operatorname{PConf}^n(\A^1),$
 where $\operatorname{PConf}^n(\A^1)$ is the moduli space of $n$ labelled points on $\A^1$, and is the complement of a relative normal crossings divisor in a smooth proper scheme \cite[Lemma 7.6]{Ellenberg2016}, as well as an $S_n$ cover of $\Conf^n(\A^1)$.
 Then we take $S_n$-invariants to conclude that 
 $ \dim H^{j}_{\textrm{\'et}}(Y_{\overline{\F}_q},\Q_\ell)=\dim H^{j}_{\textrm{\'et}}(Y_{\C},\Q_\ell)$.
By the comparison of \'{e}tale and analytic cohomology \cite[Expos\'e XI, Theorem 4.4]{1973}, we have $\dim H^{j}( Y_{\C},\Q_\ell)=\dim H^{j}_{\textrm{\'et}}(Y_{\C},\Q_\ell)$. 
\end{proof}

The following lemma spells out explicitly how Frobenius acts on the points of our Hurwitz schemes over $\F_q$, which will be necessary for the proof of Theorem~\ref{T:comp}.

\begin{lemma}\label{L:HowFrob}
For any $q$ relatively prime to $|G|$, the $\overline{\F}_q$-points of $(\cHur_{ G,c}^n)_{\overline{\F}_q}$
correspond to certain $(f,\iota,P)$, where $f\colon X \ra \P^1_{\overline{\F}_q}$ is a map of smooth proper curves over $\overline{\F}_q$.  The action of the Frobenius $\Frob_{(\cHur_{ G,c}^n)_{{\F}_q}}$ on these 
$\overline{\F}_q$-points takes $(f,\iota,P)$ to $(f^F,\iota^F,P^F)$, where $F\colon\Spec \overline{\F}_q \ra \Spec  \overline{\F}_q$ is the morphism corresponding to the map $x\mapsto x^q$  on fields, and 
$X^F$ is the base change $X\times_{\Spec \overline{\F}_q}\Spec \overline{\F}_q$ along $F$, and similarly for  $f^F,\iota^F,P^F$.

\end{lemma}
\begin{proof}
For any scheme $Z$ over $\F_q$, and any $\overline{\F}_q$-point (i.e.,~morphism of $\Spec \overline{\F}_q$-schemes) $s\colon\Spec \overline{\F}_q\ra Z_{\overline{\F}_q}$, we have that $\Frob_{Z} \circ s=s\circ F.$
Using this fact, the lemma follows from the definition of $\cHur_{ G,c}^n$.
\end{proof}

\section{Proof of Theorem~\ref{T:comp}}\label{S:countcomp}
We start with the notation from Notation~\ref{N:Gc}, and not yet with the specific $G$ and $c$ of Theorem~\ref{T:comp}.
Theorem~\ref{T:comp} is about counting Frobenius fixed components in $\cHur_{ G,c}^n$.
  First, we will give some definitions necessary to describe a component invariant of $\cHur_{ G,c}^n$.
The invariant is constructed in \cite{Components}, and more details about the following definitions and statements can be found there.

For an algebraically closed field $k$, we define $\widehat{\Z}(1)_k=\varprojlim \mu_m(k)$ and $\widehat{\Z}_k=\varprojlim \Z/m\Z$, where $m$ ranges over all positive integers if $\chr k=0$ and over all positive integers relatively prime to the characteristic of $k$ when $\chr(k)>0$.  The subset of topological generators of $\widehat{\Z}(1)_k$ will be denoted $\widehat{\Z}(1)^{\times}_k$.  It is a torsor for the units $(\widehat{\Z}_k)^\times$ of $\widehat{\Z}_{k}.$  For a set $X$ with an action of $(\widehat{\Z}_k)^\times$, we define
$$
X\langle -1 \rangle_k :=\operatorname{Map}_{(\widehat{\Z}_{k})^\times }(\widehat{\Z}(1)_{k}^{\times},X )
$$
to be the set of functions $\widehat{\Z}(1)_k^{\times}\ra X$ equivariant for the $(\widehat{\Z}_k)^\times$-actions.  If we choose an element $\mu \in \widehat{\Z}(1)_k^{\times}$, then elements of 
$X\langle -1 \rangle_k$ are specified  by their values on $\mu$
by the $(\widehat{\Z}_{k})^\times$  action.  We omit the subscripts $k$ from the notation when we are working over a fixed field.

We write $c/G$ for the set of non-trivial conjugacy classes in $c$ (the quotient of the set $c$ under the conjugation action of $G$).  If $\chr k=0$, or $\chr k>0$ and $\chr k\nmid |G|$, then  $(\widehat{\Z}_k)^\times$ has a \emph{powering} action on the set of elements of $G$,  where $\{\alpha_m\}\in (\widehat{\Z}_{k})^\times$ takes $g$ to $g^{\alpha_{\operatorname{ord}(g)}}$.  This gives an induced action of $(\widehat{\Z}_k)^\times$ on $c/G$ and thus $\Z^{c/G}$ (by permuting the coordinates via the action on $c/G$).  
Let $\mathfrak{S}^{c,G}$ be the set of $(\widehat{\Z}_\C)^\times$-orbits of elements of $\Z^{c/G}$.

For every $x,y\in G$ that commute, we have a homomorphism $\Z^2\ra G$ that sends $(i,j)$ to $x^iy^j$,
and an induced homomorphism $\Z=H_2(\Z^2,\Z)\ra H_2(G,\Z)$.  Let $\langle x, y\rangle$ be the image of $1$ under this homomorphism (which is $[x|y]-[y|x]\in H_2(G,\Z)$ in the non-homogeneous chain notation for group homology).
Let $Q_c$ be the subgroup of $H_2(G,\Z)$ generated by $\langle x,y\rangle$ for $x\in c$ and $y\in G$ that commute with $x$.  We define $H_2(G, c):=H_2(G,\Z)/Q_c$.

We now  will define groups $K(G,c)$ and $U(G,c)$.
Every central extension $\widetilde{G}\ra G$ gives rise to a natural map $\tau_{\widetilde{G}}\colon H_2(G,\Z)\ra \ker (\widetilde{G} \to G)$ 
(given by the Universal Coefficients Theorem; of course $\tau_{\widetilde{G}}$ depends on the map $\widetilde{G}\ra G$ but we leave that dependence implicit).
Let $S\ra G$ be a Schur covering group of $G$, which is a central extension such that  $\tau_S \colon
H_2(G,\Z)\ra \ker(S\ra G)$ is an isomorphism.  
Let $\overline{S}:=S/\tau_S(Q_c)$, which is  a central extension of $G$, and we can check that 
$\tau_{\overline{S}}\colon H_2(G,\Z)\ra \ker(\overline{S}\ra G)$ factors through $H_2(G, c)$ and gives an isomorphism 
$\overline{\tau}_{\overline{S}}\colon H_2(G, c)\isom \ker(\overline{S}\ra G)$.
Then we define $U(G,c):= \overline{S} \times_{G^{\operatorname{ab}}}  \Z^{c/G}$, where the map $\Z^{c/G}\ra G^{\operatorname{ab}}$ sends the generator corresponding to a conjugacy class to the image of an element in that class.
We define $K(G,c):=\ker(U(G,c)\ra G)$, where the map $U(G,c) \to G$ is the composition of surjections $U(G,c) \to \overline{S}$ and $\overline{S}\to G$.
See \cite[Section 2]{Components} for more detailed definitions of $U(G,c)$ and $K(G,c)$ and proof that they depend only on $G,c$ and not the choice of Schur cover.  

We will now describe an action of $(\widehat{\Z}_k)^\times$ on the \emph{set} of elements of $K(G,c)$, when $\chr k=0$, or $\chr k>0$ and $\chr k\nmid |G|$.
(Even though $K(G,c)$ has a group structure, this action does not respect the group structure.)
In each conjugacy class $\gamma$ in $c/G$, we pick one element $x_\gamma$ in the class $\gamma$, and then one preimage $\widehat{x}_\gamma$ of $x_\gamma$ in $\overline{S}$. 
Then if $y=gx_\gamma g^{-1}$ for some $g\in G$, we can easily check the element $\widehat{y}:= \widetilde{g} \widehat{x}_\gamma \widetilde{g}^{-1}$ does not depend on the choice of $g$ or preimage $\widetilde{g}$ of $g$ in $\overline{S}$.
For $x\in c$, we define $[x]=(\widehat{x},e_x)\in U(G,c)$, where $\widehat{x}$ is as above, and $e_x$ is the generator $\Z^{c/G}$ corresponding to the conjugacy class of $x$.
For $\alpha\in\widehat{\Z}_k^\times$ and $\gamma\in c/G$, we define
$
w_\alpha(\gamma)=\widehat{x_\gamma}^{-\alpha} \widehat{x_\gamma^{\alpha}}\in \ker(\overline{S}\ra G).
$
We have a group homomorphism
$$
W_{\alpha}\colon \Z^{c/G} \ra \ker(\overline{S}\ra G)
$$
sending the generator for the conjugacy class $\gamma\in c/G$ to 
 $w_\alpha(\gamma).$
We have an action of $\widehat{\Z}_k^\times$ on the \emph{set} of elements of $K(G,c)$ given by
$$
\alpha * (g,\underline{m}  )=(g^{\alpha}W_\alpha({\underline{m})},\underline{m}^\alpha)
$$
for $\alpha\in \widehat{\Z}_k^\times$, where we write $(g, \underline{m}) \in \ker(\overline{S}\to G) \times_{G^{\operatorname{ab}}} \Z^{c/G}=K(G,c)$.

For a positive integer $M$, we write $\Z^{c/G}_{\geq M} \one$ for the elements of $\Z^{c/G} \one$ that take any element of $\widehat{\Z}(1)_k^{\times}$ to an element all of whose coordinates are at least $M$.
We write $K(G,c)_{\geq M} \one$ for the preimage of $\Z^{c/G}_{\geq M} \one$ in $K(G,c) \one$ and $\mathfrak{S}^{c,G}_{\geq M}$ for the image of $\Z^{c/G}_{\geq M} \one$ in $\mathfrak{S}^{c,G}$. We write $\Z^{c/G}_{n,\geq M} \one$ for the elements of $\Z^{c/G}_{\geq M} \one$ that take any element of $\widehat{\Z}(1)_k^{\times}$ to an element whose coordinates sum to $n$, and the same subscripts for $K(G,c)$ and $\mathfrak{S}^{c,G}$ analogously.

We can now explain the existence of a component invariant of $\cHur_{ G,c}^n$ and the properties
of this invariant that we will apply.  The idea for this  invariant is due to Ellenberg, Venkatesh, and Westerland, and the below theorem can be found in \cite[Theorem 5.2, Theorem 2.5,  Theorem 6.1, and Section 6.1]{Components}.

\begin{theorem}\label{T:LI}
Let $G,c,n$ be as in Notation~\ref{N:Gc}.
For every algebraically closed field $k$ of characteristic relatively prime to $|G|$, and
every $k$-point $\overline{s}$ of $\cHur_{ G,c}^n$, we can associate 
an element $\mathfrak{z}_{\overline{s}}\in K(G,c)_{n,\geq 0} \one_k $ called the \emph{lifting invariant} such that the following properties hold.

\begin{enumerate}
\item \emph{(Comparison of the lifting invariant in families)}  Let $S$ be a scheme, and $\phi\colon S \ra \cHur_{ G,c}^n$ a morphism.
Let $\overline{s}_1$ and $\overline{s}_2$ be geometric points of $S$ such that the image of $\overline{s}_2$ is in the closure of the image of $\overline{s}_1$. 
\begin{enumerate}
\item \emph{(Across different fields)}
For every $m$ relatively prime to the characteristic of the residue field $k(\overline{s}_2),$
there is a group isomorphism $\mu_m(k(\overline{s}_1))\ra \mu_m(k(\overline{s}_2))$ such that these isomorphisms are compatible with the powering maps $\mu_{m_1m_2}(k)\ra \mu_{m_1}(k)$ that take $\zeta\mapsto \zeta^{m_2},$ giving a map $\sigma: \widehat{\Z}(1)_{k(\overline{s}_1)}\ra 
\widehat{\Z}(1)_{k(\overline{s}_2)}$ taking topological generators to topological generators.
Composing with this map, we have that $\mathfrak{z}_{{\overline{s}_2}}\mapsto \mathfrak{z}_{{\overline{s}_1}}.$
 \item \emph{(Over the same field)}
 If 
 $S$ is a $k$-scheme (and the points $\overline{s}_i$ above respect the $k$-scheme structure of $S$), then we can 
naturally identify $\widehat{\Z}(1)_{k(\overline{s}_i)}$ with $\widehat{\Z}(1)_{k}$, and choose $\sigma$ above so that 
it corresponds to the identity on   $\widehat{\Z}(1)_{k}$ via these identifications.
\end{enumerate}
\item \emph{(Action of  Frobenius on the lifting invariant)} If $\overline{s}$ is a $\overline{\F}_q$-point of $(\cHur_{ G,c}^n)_{\overline{\F}_q}$ (respecting the 
$\overline{\F}_q$-scheme structure), and $\Frob:=\Frob_{(\cHur_{ G,c}^n)_{{\F}_q}}$  is the Frobenius map, then for a topological generator
$\zeta$ of $\widehat{\Z}(1)_{\overline{\F}_q}$ with $\mathfrak{z}_{\overline{s}}(\zeta)=g$ , we have
$\mathfrak{z}_{\Frob(\overline{s})}(\zeta)=q^{-1} *g.$ 
\end{enumerate}
\end{theorem}

\begin{corollary}\label{C:comp}
For $G,c,n,k$ as in Theorem~\ref{T:LI}, we have that $k$-points of $(\cHur_{ G,c}^n)_{k}$ in the same component have the same lifting invariant (which we now call the lifting invariant of the component).  Geometric points of the same component of $\cHur_{ G,c}^n$ have lifting invariants with the same image in $\mathfrak{S}^{c,G}$ (which we now call the $\mathfrak{S}^{c,G}$-invariant of the component).
\end{corollary}

We now give another description of the components over $\cHur_{ G,c}^n$ over $\C$ coming from topology.
We recall the definition of the braid group 
$$B_n:=\langle \sigma_1, \dots, \sigma_{n-1} | 
\sigma_i\sigma_{i+1}\sigma_i=\sigma_{i+1}\sigma_i\sigma_{i+1} (1\leq i\leq n-2), \sigma_i\sigma_j=\sigma_j\sigma_i (i-j\geq 2)\rangle.$$
The braid group has a natural action on $G^n$, where for $\sigma\in B_n$, we have
\begin{equation}\label{E:braidact}
\sigma_i(g_1,\dots,g_n)\mapsto (g_1,\dots,g_{i-1},g_ig_{i+1}g_i^{-1},g_i,g_{i+2},\dots,g_n),
\end{equation}
where only coordinates $i$ and $i+1$ are changed.

\begin{theorem}\label{T:CcompLong}
Let $G,c,n$ be as in Notation~\ref{N:Gc}.
Consider a point $\underline{P}\in \Conf^n \A^1(\C)$ corresponding to $\C$-points $P_1,\dots,P_n$ of $\A^1$, and choices  $\gamma_1,\dots \gamma_n\in \pi_1(\P^1(\C) \setminus \{P_1,\dots,P_n \},\infty)$, such that $\gamma_i$ generates an inertia group of $P_i$ for all $i$, and $\pi_1(\P^1(\C) \setminus \{P_1,\dots,P_n \},\infty)$ is the free group generated by the $\gamma_i$ for $1\leq i\leq n-1$, and $\gamma_1\cdots\gamma_n=1$. 

Using the $\gamma_i's$, the fiber of $\cHur_{ G,c}^n$ over $\underline{P}$ is  identified with the set of tuples $(g_1,\dots,g_n)\in c^n$ such that $g_1\cdots g_n=1$ and the $g_i$ generate $G$.

Moreover, there is a $\zeta\in \widehat{\Z}(1)_\C^\times$, such that, for every point in the  fiber of $\cHur_{ G,c}^n$ over $\underline{P}$, if the point is identified with $(g_1,\dots,g_n)$, then it has lifting invariant that sends $\zeta$ to $[g_1]\cdots [g_n]$ (where $[g_i]$ is defined above).

 Also, $\pi_1(\Conf^n \A^1(\C),\underline{P})$ is naturally identified with $B_n$ 
  in such that way that the action of $B_n$ on the fiber above $\underline{P}$ in the covering space $\cHur_{ G,c}^n(\C)\ra \Conf^n \A^1(\C)$ 
  is given by \eqref{E:braidact}.

Via these correspondences, the components of $\cHur_{ G,c}^n(\C)$ are in bijection with $B_n$-orbits of $(g_1,\dots,g_n)\in c^n$ with  $g_1\cdots g_n=1$ and with the $g_i$ generating $G$.
\end{theorem}

\begin{proof}
Given a Galois cover $X\ra \P^1$ over $\C$, with automorphism group identified with $G$ and
choice of point above $\infty$,  and with branch locus $\{P_1,\dots,P_n\}$, 
 if we let $U$ be the complement in $X$ of the preimages of the $P_i$, then $U \ra \P^1(\C) \setminus \{P_1,\dots,P_n \}$ is a normal covering space map with automorphism group identified with $G$ and choice of point above $\infty$, which is equivalent to a surjective group homomorphism $\pi_1(\P^1(\C) \setminus \{P_1,\dots,P_n \},\infty)\ra G.$
Given a  normal covering space map $U \ra \P^1(\C) \setminus \{P_1,\dots,P_n \}$, there is a unique cover of curves $X\ra \P^1$ from which it arises by deleting the branch locus, and the covering space map and cover of curves have the same automorphism group.  
 By taking the images of the $\gamma_i$, surjective group homomorphisms $\pi_1(\P^1(\C) \setminus \{P_1,\dots,P_n \},\infty)\ra G$ are equivalent to tuples $(g_1,\dots,g_n)$ with  $g_i\in G$, and $g_1\cdots g_n=1$, and with the $g_i$ generating $G$.  In this correspondence, our points of $\cHur_{ G,c}^n$ correspond to those tuples with all $g_i\in c$.  
The statement on the lifting invariant follows directly from the definition of the lifting invariant \cite[Theorem 5.2]{Components} (see also \cite[Theorem 2.5]{Components} for an isomorphism taking our $[g_i]$ to those in \cite{Components}). 

Finally, the action of $\pi_1(\Conf^n \A^1(\C),\underline{P})$ on the covering space $\cHur_{ G,c}^n(\C)$ follows from two inputs: (1) our proof in Section~\ref{S:top} of a homeomorphism between 
$\cHur_{ G,c}^n(\C)$ with the analytic topology and a topologically defined Hurwitz space, (2)
Fried and V\"{o}lklein's computation of the action of the braid group on this topologically defined Hurwitz space \cite[Section 1.4]{Fried}.  In \cite{Fried} the authors work with a quotient of $\Hur_{G,*}^{n, \text{top}}$, but their proof works equally well on $\Hur_{G,*}^{n, \text{top}}$ (and in fact their computation naturally occurs on $\Hur_{G,*}^{n, \text{top}}$ and they just pass to a quotient).
\end{proof}

We let $c^{n,G}_{1,\geq M}$ be the subset of $c^n$ of tuples $(g_1,\dots,g_n)$ such that $g_1\cdots g_n=1$, each conjugacy class of $c$ has at least $M$ indices $i$ with $g_i$ in that class, and the $g_1,\dots,g_n$ generate $G$.  
The following theorem and corollary show that, at least for components including every inertia type at least $M$ times, the component invariant described in Theorem~\ref{T:LI} exactly cuts out the components over $\C$ we have described topologically in Theorem~\ref{T:CcompLong} as corresponding to braid group orbits, and that further these also exactly correspond to the components over $\overline{\F}_q$.

\begin{theorem}[{\cite[Theorem 3.1]{Components}}]\label{T:Ccomp}
Let $G,c$ be as in Notation~\ref{N:Gc}, such that $c$ generates $G$. There is $M$ sufficiently large such that the following holds for all $n\geq M$:
The map of sets $\Phi\colon c^{n,G}_{1,\geq M} \ra K(G,c)_{n,\geq M}$ given by $(g_1,\dots,g_n)\mapsto [g_1]\cdots [g_n]$ is surjective, and its fibers are exactly the braid group $B_n$-orbits on 
$c^n_{1,\geq M}$.
\end{theorem}

\begin{corollary}\label{C:LIworks}
Let $G,c$ be as in Notation~\ref{N:Gc}, $M$ sufficiently large for Theorem~\ref{T:Ccomp}, and  $q$ a prime power relatively prime to $G$.
Let $n\geq 0$.  For $s\in\mathfrak{S}^{c,G}_{n,\geq M},$
let $Y$ be the union of components of $\cHur_{ G,c}^n$ with $\mathfrak{S}^{c,G}$-invariant $s$.  
Then the lifting invariant gives a bijection between components of $Y_\C$ and elements of $K(G,c)\one_\C$ with image $s$, and a bijection 
between components of $Y_{\overline{\F}_q}$ and elements of $K(G,c)\one_{\overline{\F}_q}$ with image $s$.
\end{corollary}
\begin{proof}
The statement for components over $\C$ follows directly from Theorems~\ref{T:LI}, \ref{T:CcompLong}, and \ref{T:Ccomp}.
For $\overline{\F}_q$ components, we first claim that all  elements $K(G,c)\one_{\overline{\F}_q}$ with image $s\in\mathfrak{S}^{c,G}$ arise as lifting invariants.  
From \cite[Remark 5.4]{Components},
we see that there is some $\zeta\in \widehat{\Z}(1)_{\overline{\F}_q}^\times$ such that  for
every element $h$ in the image of the map $c^n_{1,\geq 0} \ra K(G,c)$, we have that $\zeta\mapsto h$ occurs as a lifting invariant over
$\overline{\F}_q$.  From Theorem~\ref{T:Ccomp}, we see this includes all lifting invariants in $K(G,c)_{n,\geq M} \one_{\overline{\F}_q}$.  
Thus each possible lifting invariant occurs from at least one component over $\overline{\F}_q$.  Then Lemma~\ref{L:C} with $i=2n$ tells us that the number of 
components of $Y_{\overline{\F}_q}$ equals the number of components of $Y_\C$, which we already saw was the number of lifting invariants with image $s$.
\end{proof}

Next, we are going to prove a result explicitly expressing a count of components in terms of the group theory of $K(G,c)$.  Proposition~\ref{P:countcomp} will take into account components described by Corollary~\ref{C:LIworks} as well as those not described by Corollary~\ref{C:LIworks} because they do not involve each inertia type $M$ times (and thus the result will include an error term).
We write $\Z^{c/G}_{\equiv q}$ for the set of  elements of $\Z^{c/G}$ whose coordinates are constant on each set of conjugacy classes of ${c/G}$ that can be obtained from one another by taking $q$th powers, and
$\Z^{c/G}_{\equiv q,n,\geq M}$ when we also add the conditions that all coordinates are at least $M$ and the coordinates sum to $n$. For an element $a$ in an abelian group $A$, we write $\nr_{k}(a)$ for the number of $x\in A$ such that $x^{k}=a$.

\begin{proposition}\label{P:countcomp}
Let $G,c,n$ be as in Notation~\ref{N:Gc}.
Let $q$ be a prime power such that $(q,|G|)=1$.
Let 
$$b(G,c,q,n)
=\sum_{\underline{m}\in \ker\left(\Z^{c/G}_{\equiv q,n\geq 0}\ra G^{\operatorname{ab}}\right) } \nr_{q-1}( W_{q^{-1}}(\underline{m}) ).
$$ 
(For applying $\nr_{q-1}$, we are considering $W_{q^{-1}}(\underline{m})$ as an element of $\ker(\overline{S}\ra G)$.)
Then
 $$\pi_{G,c}(q,n)=
b(G,c,q,n)
 +O_G\left(n^{d_{G,c}(q)-2}\right)
 $$
 and if $b(G,c,q,n)=0$ then $\pi_{G,c}(q,n)=0$.
\end{proposition}
Of course if $(|G|,q-1)=1$, then $\nr_{q-1}( W_{q^{-1}}(\underline{m})^q )=1$ for all $\underline{m}$.
However, we will be interested, even for understanding unramified extensions of quadratic fields, in the case when $|G|$ and $(q-1)$ may share factors.

\begin{proof}

By the last statement of Theorem~\ref{T:LI}, we have that 
 for a component of $(\cHur_{ G,c}^n)_{\overline{\F}_q}$ to be fixed by $\Frob:=\Frob_{(\cHur_{ G,c}^n)_{\F_q}}$, its lifting invariant  must take some topological generator $\zeta$ to $g$ where $g=q^{-1} * g$.  
  In particular, we have that the image of $g$ in $\Z^{c/G}$ lies in $\Z^{c/G}_{\equiv q}$.
 
If an element of $K(G,c)\one$ sends a $\zeta\in (\widehat{\Z} (1))^\times$ to an element with image in $\Z^{c/G}_{\equiv q}$, then it sends any other element 
 $\zeta'\in (\widehat{\Z} (1))^\times$ to an element with image in $\Z^{c/G}_{\equiv q}$.   Thus by Theorem~\ref{T:LI} there is a union $Z_q$ of components of $\cHur_{ G,c}^n$ all of whose geometric points have lifting invariants that send a topological generator to elements with image in $\Z^{c/G}_{\equiv q}$, and such that all geometric points of $\cHur_{ G,c}^n$ with this property are in $Z_q$.
 We will take an integer $M$ that is sufficiently large  for Theorem~\ref{T:Ccomp}. If an element of $K(G,c)\one$ sends a $\zeta\in (\widehat{\Z} (1))^\times$ to an element with image in $\Z^{c/G}$ that has some coordinate $<M$, then it sends any other element 
 $\zeta'\in (\widehat{\Z} (1))^\times$ to an element with the same property.   Thus by Theorem~\ref{T:LI} there is a union $Z'_q$ of components of $Z_q$ all of whose geometric points have lifting invariants that send a topological generator to an element with image in $\Z^{c/G}$ that has some coordinate $<M$, and such that all geometric points of $Z_q$ with this property are in $Z'_q$.

First, we consider components of $(\cHur_{ G,c}^n)_{\overline{\F}_q}$ in $Z'_q$.
Note there are 
$O_G(n^{d_G(q)-2})$  choices of $\underline{m}\in \Z^{c/G}_{\equiv q}$ with some component $<M$. 
 By Theorem~\ref{T:CcompLong} and \cite[Lemma 3.3]{Ellenberg2005}  there are $O_G(1)$ components of $(\cHur_{ G,c}^n)_{\C}$ corresponding to each $\underline{m}$.  
Thus there are $O_G(n^{d_G(q)-2})$ components of $(Z'_q)_{\C}$, and thus by Lemma~\ref{L:C}, $O_G(n^{d_G(q)-2})$ components of $(Z'_q)_{\overline{\F}_q}$.

Now, we consider $Z_q\setminus Z'_q$.
We fix an element of $\zeta \in (\widehat{\Z}(1)_{\overline{\F}_q})^\times$, and thereby associate any $(\widehat{\Z}_{\overline{\F}_q})^\times$-set  $X$ with $X\one_{\overline{\F}_q}$.
By Corollary~\ref{C:LIworks}, components of $(Z_q\setminus Z'_q)_{\overline{\F}_q}$ exactly correspond to elements of $K(G,c)_{n,\geq M}$
whose image in $\Z^{c/G}$ lies in $\Z^{c/G}_{\equiv q}$. 
Since components correspond exactly to lifting invariants in this case, a component is $\Frob$-fixed if and only if its lifting invariant is, which by Theorem~\ref{T:LI}
are exactly those  $g\in K(G,c)$ such that $g=q^{-1}*g$.  
We write  $g=(h,\underline{m})\in \ker(\overline{S}\ra G )\times_{G^{\operatorname{ab}}}\Z^{c/G}=K(G,c)$.
We have $g=q^{-1}*g$ if and only if
$$
(h,\underline{m})=(h^{q^{-1}}W_{q^{-1}}({\underline{m}}),\underline{m}^{q^{-1}}).
$$
Since $\underline{m}\in \Z^{c/G}_{\equiv q}$, we have $\underline{m}=\underline{m}^{q^{-1}}$.
We have that
$h=h^{q^{-1}}W_{q^{-1}}({\underline{m}})$ if and only if $h^{q-1}=W_{q^{-1}}({\underline{m}})^q$.
So  $\underline{m}\in \Z^{c/G}_{\equiv q}$ with trivial image in $G^{\operatorname{ab}}$ has 
$\nr_{q-1}( W_{q^{-1}}(\underline{m})^q )$
 elements $g\in K(G,c)$ mapping to $\underline{m}$  such that $g=q^{-1}*g$.  
Note since $\ker(\overline{S}\ra G)$ is isomorphic to $H_2(G,c)$ and $q$ is relatively prime to $G$ and hence
$H_2(G,c)$, we have that 
$\nr_{q-1}( W_{q^{-1}}(\underline{m})^q )=\nr_{q-1}( W_{q^{-1}}(\underline{m})).$
So we see that $b(G,c,q,n)$ is exactly the number of $\Frob$-fixed component invariants in 
$K(G,c)_{n,\geq 0}\one_{\overline{\F}_q}.$  We immediately conclude the final statement of the proposition because if there are no $\Frob$-fixed component invariants then there are no $\Frob$-fixed components. 
Since in this case we are only counting $\underline{m}$ all of whose coordinates are $\geq M$, there are
 $$
b(G,c,q,n)
 +O_G(n^{d_G(q)-2})
 $$
$\Frob$ fixed components in this case.
Combining the two cases, we prove the proposition.
\end{proof}

Now we turn to the question of estimating the $b(G,c,q,n)$.  For this we will use the following lemma on counting lattice points, which is an application of Davenport's theorem
\cite{Davenport1951} after projecting everything onto a codimension $1$ space by forgetting the last coordinate. 
\begin{lemma}\label{L:lattice}

Let $L\sub \Z^d$ be a translate of a full lattice in $\R^d$ by an integer vector. Let $a_i$ be non-negative integers for $i=1,\dots,d$, not all $0$. Then there exists a positive integer $k$, a residue $s$ mod $k$, and a real $r>0$ such that the following holds.
 For every positive integer $n$, let $L'(n)=L\cap \{\underline{x}\in\R^d \,\mid\, \sum_{i=1}^d a_ix_i =n  \}$, and $L''(n)=L'(n)\cap \{\underline{x}\in\R^d\, \mid\, \forall i, x_i\geq 0 \}$.   If $n \not\equiv s \pmod k$, then
  $L'(n)=\emptyset$.  If $n \equiv s \pmod k$, then
  $$
  \#L''(n)= rn^{d-1}+O_{d,L,\underline{a}}(n^{d-2}).
  $$
\end{lemma}

\begin{corollary}\label{C:countB}
Given $G,c$ as in Notation~\ref{N:Gc}, and an element $a\in (\Z/|G|^2\Z)^\times$, there is a positive integer $M_a$,  a non-empty set $E_a$ of residues mod $M_a$, and positive reals $r_{a,b}$ for $b\in E_a$,
such that for any prime power $q\equiv a \pmod{|G|^2}$ and any non-negative integer $n$ the following holds:
\begin{itemize}
\item If $n\pmod{M_a} \not\in E_a$ then  
$\pi_{G,c}(q,n)=0$, and 
\item if $n\pmod{M_a} =b\in E_a$ then
$
\pi_{G,c}(q,n)=r_{a,b}n^{d_{G,c}(q)-1}+O_{G}(n^{d_{G,c}(q)-2}).
$
\end{itemize}
\end{corollary}
\begin{proof}
The set $\Z^{c/G}_{\equiv q}$, as a subset of $\Z^{c/G}$, only depends on $q \pmod{|G|}$.  Further, since every element of $H_2(G,\Z)$ has order dividing $|G|$, we also see that $\nr_{q-1}$ on $\ker(\overline{S}\ra G)$ only depends on $q \pmod{|G|}$.  Now since $W_{q^{-1}}$ involves elements in $\overline{S}$, and elements in $\overline{S}$ all have order dividing $|G|^2$, we conclude 
 $b(G,C,q,n)$ only depends on $q \pmod{|G|^2}$.  We consider a single residue, $a \pmod{|G|^2}$, and now consider $q\equiv a \pmod{|G|^2}$.  We will write $a$ instead of $q$ in all the places where we have $q$-dependence that only depends on $a$.
 
 We have
$$
b(G,c,a,n)=\sum_{h\in \ker(\overline{S}\ra G)} \nr_{a-1}(h) \#\{ \underline{m}\in \ker(\Z^{c/G}_{\equiv a,n\geq 0}\ra G^{\operatorname{ab}}) \,\mid\, W_{a^{-1}}(m)=h  \}.
$$
The elements of $\Z^{c/G}_{\equiv a}$ with trivial image in $G^{\operatorname{ab}}$ form a sublattice.  The elements of $\Z^{c/G}_{\equiv a}$ with trivial image under $W_{a^{-1}}$ and in $G^{\operatorname{ab}}$ form a sublattice, and for any $h\in \ker(\overline{S}\ra G)$, the elements of 
$\Z^{c/G}_{\equiv a}$ with image $h$ under $W_{a^{-1}}$ and trivial image in $G^{\operatorname{ab}}$ are a translate of a sublattice.  When $\nr_{a-1}(h)=0$, the lattice translate obtained is not relevant
for the calculation of $b(G,c,a,n)$.  However, there are certainly some $h$ that are $(a-1)$th powers, and we now consider their corresponding lattice translates.

We consider the natural isomorphism $\Z^{c/G}_{\equiv a}\isom \Z^{d_{G,c}(a)}$ where we take one coordinate from each $q$-powering class.  All of the lattices above, viewed in $\Z^{d_{G,c}(a)}$, are full lattices.  
The condition that the coordinates are positive remains the same under $\Z^{c/G}_{\equiv a}\isom \Z^{d_{G,c}(a)}$. 
The condition that the coordinates sum to $n$ in $\Z^{c/G}_{\equiv a}$ becomes that, for certain positive integers $a_i$ (the sizes of the $q$th powering orbits),
the linear combination $\phi(\underline{x}):=\sum_i a_ix_i$ of the coordinates $(x_1,\dots, x_{d_{G,c}(a)})\in \Z^{d_{G,c}(a)}$ is equal to $n$.  

We then apply Lemma~\ref{L:lattice} to each lattice translate under consideration.  For each lattice translate (and $\phi$), we obtain the values $k$ and $s$ from Lemma~\ref{L:lattice}.  Taking the least common multiple of the $k$'s to be $M_a$ and let $E_a$ be the set of residues mod $M$ such that for $n \pmod{M_a} \in E_a$, at least one lattice translate has points where $\phi$ takes value $n$.  
Then the corollary with $b(G,c,q,n)$ in place of $\pi_{G,c}(q,n)$ follows from the application of Lemma~\ref{L:lattice}.  The corollary as written then follows from Proposition~\ref{P:countcomp}.
\end{proof}

Now we have counted to components for each $G,c$, somewhat inexplicitly.  In order to prove 
 Theorem~\ref{T:comp}, we will need to compare these counts for $G,c$ and $\Gamma,\Gamma\setminus 1$.
 First, this requires a a group theory lemma on the compatibility of various maps related to $H\rtimes \Gamma$ and $\Gamma$.

	\begin{lemma}\label{L:Schur}
		Let $\Gamma$ be a finite group, $H$ an admissible $|\Gamma|'$-$\Gamma$-group, and $G:=H\rtimes \Gamma$. 
Then there are Schur covering groups $S,S_\Gamma$, of $G, \Gamma$ respectively, with a commutative diagram
		$$
\xymatrix{
S \ar[d] \ar[r]& G \ar[d] \\
S_\Gamma \ar[r]& \Gamma.
}
$$
		Also, the quotient map $G\to \Gamma$ induces $\rho\colon H_2(G, \Z) \to H_2(\Gamma,\Z)$, and
		 $\gcd(|\ker \rho|,|\Gamma|)=1$. 
	\end{lemma}
	
	\begin{proof}
	We choose a Schur covering group of $G$, and denote it by
	\begin{equation}\label{eq:StemCover}
		1\to M \to S \overset{\tau}{\to} G \to 1.
	\end{equation}
	Then as a group $M$ is isomorphic to $H_2(G,\Z)$. As $M$ is an abelian group, we can write $M$ as $M_H \times  M_\Gamma$, where $M_H$ is the product of Sylow $p$-subgroups of $M$ for all $p \mid {H}$ and $M_{\Gamma}$ is the complement of $M_H$.

	The group $S$ has a normal subgroup $S':=\tau^{-1}(H)$.  If we take the quotient $S'/M_H$, we obtain an exact sequence
	$$
	1\ra M_\Gamma \ra S'/M_H \ra H \ra 1.
	$$
	Since $\gcd(|M_\Gamma|,|H|)=1$, we have that this sequence splits, and since $M_\Gamma$ is central, 
	there is a unique splitting $H\ra S'/M_H$ and projection $\phi: S'/M_H \ra M_\Gamma$ (that is the identity on $M_\Gamma$).
Then we let $D$ be the preimage of $\ker(\phi)$ in $S'$.  We have that $|D|=|H||M_H|=|S'|/|M_\Gamma|$
and $\gcd(|D|,|M_{\Gamma}|)=1$.  Since $M_H$ is central, it follows that $S'\isom M_\Gamma \times D$.	
	 The group $D$ is a normal Hall  subgroup of $S'$ and hence characteristic, which implies that $D$ is a normal subgroup of $S$, so we have the following commutative diagram:
	\[\begin{tikzcd}
		1 \arrow{r} & M_\Gamma \times M_H \arrow{r} \arrow{d}{/M_H}& S \arrow{r}\arrow{d}{/D}  & H\rtimes \Gamma \arrow{r} \arrow{d}{/H} &1  \\
		1 \arrow{r} & M_\Gamma \arrow{r} & {S_\Gamma}  \arrow{r} & \Gamma \arrow{r} &1.
	\end{tikzcd}\]
	The bottom row is a central extension of $\Gamma$ by $M_\Gamma$.  
	 For any commutative diagram of central extensions, the induced maps from $H_2$ commute with the maps between the kernels (one can check from definitions).  In this case this gives 
a commutative diagram
$$
\xymatrix{
{H_2(G,\Z)} \ar[d]_{\rho} \ar[r]_{\pi_{S}}^{\isom}& M_{\Gamma}\times M_H \ar[d] \\
{H_2(\Gamma,\Z)} \ar[r]_{\pi_{S_\Gamma}}& M_\Gamma,
}
$$
which gives an isomorphism $\ker \rho \isom M_H$.
Since $|H_2(\Gamma,\Z)|$ is only divisible by primes dividing $|\Gamma|$, and $M_\Gamma \times M_H\ra M_\Gamma$ is the quotient by all the Sylow $p$-subgroups for primes $p\nmid |\Gamma|$ (and $\rho$ is surjective since $G\ra\Gamma$ is split), it follows that $\pi_{S_\Gamma}$ is an isomorphism, and thus $S_\Gamma$ is a Schur covering group of $\Gamma$.   
We have that $|H_2(\Gamma,\Z)|=|H_2(G,\Z)|/|M_H|,$ and hence the lemma follows.
	\end{proof}
	
Now finally we are able to prove Theorem~\ref{T:comp} by showing agreement of $b(G,c,q,n)$ and $b(\Gamma,\Gamma\setminus 1,q,n)$.  We remark that we do not use Corollary~\ref{C:countB} in order to see this agreement, but we do use it for the lower bound on $b(\Gamma,\Gamma\setminus 1,q,n)$ in order to see that our main term is larger than our error term.

\begin{proof}[Proof of Theorem~\ref{T:comp}]  
We apply Proposition~\ref{P:countcomp} in two cases: first with $G_2=\Gamma$ and $c_2$ all non-trivial elements of $\Gamma$, and second with $G_1=H\rtimes \Gamma$ and $c_1$ the set of elements of $G$ that have the same order as their image in $\Gamma$. 
To see that $c_1$ generates $H\rtimes \Gamma$, 
first we observe that all non-trivial elements of $\Gamma$ are in $c_1$.  An element $(h,\gamma)$ is in $c_1$ if and only if it gives a splitting of the cyclic subgroup generated by $\gamma$, which by Schur--Zassenhaus is if and only if $h=g^{-1}\gamma(g)$ for some $g\in H$.  Since $H$ is an admissible $\Gamma$-group, it is generated by elements of the form $g^{-1}\gamma(g)$ and so $c_1$ generates $H\rtimes \Gamma$.
 Now we will compare the terms $b(G_1,c_1,q,n)$ and
$b(G_2,c_2,q,n)$.

Above we saw that the elements over $\gamma$ in $c_1$ are all conjugate.
This gives a bijection from $c_1/G_1 \ra c_2/G_2$ inducing an isomorphism  $\Z^{c_1/G_1}_{\equiv q}\isom \Z^{c_2/G_2}_{\equiv q}$.
We then have $d_{G_1,c_1}(q)=d_{G_2,c_2}(q)$.

Next we check that $G_1^{\operatorname{ab}} \ra G_2^{\operatorname{ab}}$ is an isomorphism.  Consider an element of the form 
$(g^{-1}\sigma(g),1)$ in $G_1$, for some $\sigma\in\Gamma$.  Since 
$(g^{-1}\sigma(g),1)=(g^{-1},1)(1,\sigma)(g,1)(1,\sigma^{-1}),$ it has trivial image in $G_1^{\operatorname{ab}}$.
Since $H$ is admissible, it is generated by elements of the form $g^{-1}\sigma(g)$, and thus we see that $H$ is in the kernel of the map $G_1\ra G_1^{\operatorname{ab}}$, from which it follows that $G_1^{\operatorname{ab}} \ra G_2^{\operatorname{ab}}$ is an isomorphism.

This gives a natural isomorphism $\ker(\Z^{c_1/G_1}_{\equiv q}\ra G_1^{\operatorname{ab}})\isom \ker(\Z^{c_2/G_2}_{\equiv q}\ra G_2^{\operatorname{ab}})$.  Now we have to consider the functions $W^i_{q^{-1}}$ on these lattices, which we denote $W^1_{q^{-1}}$ and $W^2_{q^{-1}}$ for the version corresponding to $G_1,c_1$ and $G_2,c_2$ respectively.  
We have defined $W^i_{q^{-1}}\colon \Z^{c_i/G_i} \ra \ker(\overline{S}^i\ra G_i),$ where $\overline{S}^i$ is the quotient, constructed in the start of this section, of the chosen Schur covering group in each case. 
We use Lemma~\ref{L:Schur} to choose compatible Schur coverings $S^1$ and $S^2$.
Since $G_1\ra G_2$ is split, we have that $Q_{c_1}\ra Q_{c_2}$ is a a surjection, and then it follows from Lemma~\ref{L:Schur} and basic facts on finite abelian groups that $H_2(G_1,c_1)\ra H_2(G_2,c_2)$ is a surjection with kernel of order relatively prime to $|\Gamma|$ (and so only divisible by primes dividing $|H|$), and we have a commutative diagram
$$
\xymatrix{
{H_2(G_1,c_1)} \ar[d] \ar[r]^<<<<\isom_<<<<{\tau_{\overline{S}^1}}& \ker(\overline{S}^1\ra G_1) \ar[d] \\
{H_2(G_2,c_2)} \ar[r]^<<<<\isom_<<<<{\tau_{\overline{S}^2}}& \ker(\overline{S}^2\ra G_2),
}
$$
where $\overline{S}^i:=H_2(G_i,\Z)/Q_{c_i}$.  
Then, for each conjugacy class in $c_2/G_2,$ we can pick one element $x$ in that class, and one preimage $\widetilde{x}$ of the element in $G_1$, and then pick $\widehat{\widetilde{x}}$ in the fiber of $\widehat{x}$.  
By our choice of compatible Schur coverings, we have
that
  \[
\xymatrix{
 \Z^{c_1/G_1} \ar[d] \ar[r]_<<<<<{ { W^1_{q^{-1}}}}& \ker(\overline{S}^1\ra G_1)\ar[d]_f \\
\Z^{c_2/G_2} \ar[r]_<<<<<{W^2_{q^{-1}}}& \ker(\overline{S}^2\ra G_2)
}
\]
is a commutative diagram.  Since the kernel of the map $f$ on the right has order relatively prime to $q-1$, any element of $\ker(\overline{S}^2\ra G_2)$ has the same number of $(q^{-1}-1)$th roots as 
any preimage in $\ker(\overline{S}^1\ra G_1)$.

We then have
$$
b(G_2,c_2,q,n)=\sum_{h\in \ker(\overline{S}^2\ra G_2)} \nr_{q^{-1}-1}(h) \#\{ \underline{m}\in \ker(\Z^{c_2/G_2}_{\equiv q,n\geq 0}\ra G_2^{\operatorname{ab}})\, \mid\, W^2_{q^{-1}}(m)=h  \}.
$$
and 
\begin{align*}
b(G_1,c_1,q,n)&=\sum_{\widetilde{h}\in \ker(\overline{S}^1\ra G_1)} \nr_{q^{-1}-1}(f(\widetilde{h})) \#\{ \underline{m}\in \ker(\Z^{c_1/G_1}_{\equiv q,n\geq 0}\ra G_1^{\operatorname{ab}}) \,\mid\, W^1_{q^{-1}}(m)=\widetilde{h}  \}\\
&=\sum_{h\in \ker(\overline{S}^2\ra G_2)} \nr_{q^{-1}-1}(h) \#\{ \underline{m}\in \ker(\Z^{c_2/G_2}_{\equiv q,n\geq 0}\ra G_2^{\operatorname{ab}}) \,\mid\, W^2_{q^{-1}}(m)=h  \}.
\end{align*}
It follows that $b(G_1,c_1,q,n)=b(G_2,c_2,q,n)$.
From this and Proposition~\ref{P:countcomp}, we conclude the first statement of the theorem.
 By Corollary \ref{C:countB}, we have the $M_a$ and $E_a$ so that the remaining statements on
$\pi_\Gamma(q,n)$ hold.  Our observations above and the proof of Corollary \ref{C:countB} show that when $n \pmod{M_a} \not\in E_a$, then $\pi_G(q,n)=0.$
\end{proof}

\subsection*{Acknowledgements} 
The authors would like to thank Brandon Alberts, Nigel Boston, Michael Bush, Brian Conrad, Jordan Ellenberg, Joseph Gunther, Jack Hall, Aaron Landesman, Jonah Leshin, Akshay Venkatesh, and  Weitong Wang for important and fruitful conversations regarding the work in this paper, and Alex Bartel and Michael Bush for comments on an earlier draft.
The first author was partially supported by NSF grants DMS-1301690 and DMS-1652116.
The second author was partially supported by an
American Institute of
Mathematics Five-Year Fellowship, a Packard Fellowship for Science and Engineering, a
Sloan Research Fellowship,  NSF grants DMS-1301690 and DMS-2052036,  and NSF Waterman Award DMS-2140043.
The third author was partially supported by NSF grant DMS-1555048.

\newcommand{\etalchar}[1]{$^{#1}$}
\def\cprime{$'$}

\end{document}